\documentclass[reqno,11pt]{amsart}
\usepackage{amsmath, amssymb, amsthm,amsfonts} 
\usepackage[english]{babel}
\usepackage{bbm}

\usepackage{graphicx}

\usepackage[utf8]{inputenc}
\usepackage[english]{babel}

\newtheorem*{assumption2}{Assumption}

\usepackage{url}
\usepackage{epstopdf}
\usepackage[a4paper,bindingoffset=0.5cm,left=2cm,right=2cm,top=2.5cm,bottom=2cm,footskip=.8cm]{geometry}
\usepackage{multirow}

\usepackage{tikz}
\usepackage{caption}
\usepackage{subcaption}
\usetikzlibrary{arrows, automata,positioning,calc,shapes,decorations.pathreplacing,decorations.markings,shapes.misc,petri,topaths}
  \usepackage{pgfplots}

\usepackage{rotating}
\usepackage{amsbsy,enumerate}
\usepackage{comment}
\usepackage{mathrsfs}

\newenvironment{customlegend}[1][]{%
    \begingroup
    \csname pgfplots@init@cleared@structures\endcsname
    \pgfplotsset{#1}%
}{%
    \csname pgfplots@createlegend\endcsname
    \endgroup
}%
\def\addlegendimage{\csname pgfplots@addlegendimage\endcsname}

\newcommand{\bs}{\boldsymbol}

\newcommand{\vb}{\vspace{5.2mm}}
\renewcommand{\hat}{\widehat}

\DeclareMathOperator*{\argmax}{arg\,max}

\def\P{\mathrm{P}}
\def\E{\mathrm{E}}

\def\Var{\mathrm{Var}}

\newcommand{\apost }{'}
\newcommand{\parrow}{\:\to_{\mathbb P}\:}
\newcommand{\darrow}{\:\to_{\rm d}\:}
\newcommand{\asarrow}{\:\to_{\rm as}\:}
\newcommand*\diff{\mathop{}\!\mathrm{d}}

\allowdisplaybreaks

\allowdisplaybreaks

\newtheorem{lemma}{Lemma}

\newtheorem{theorem}{Theorem}
\newtheorem{remark}{Remark}
\newtheorem{example}{Example}

\begin{document}

\title[Estimating customer impatience in a service system with unobserved balking]{Estimating customer impatience \\in a service system with unobserved balking\footnote{To appear in Stochastic Systems.}}

\author[Y. Inoue, L. Ravner, M. Mandjes]{Yoshiaki Inoue, Liron Ravner, Michel Mandjes}

\begin{abstract} 
This paper studies a service system in which arriving customers are provided with information about the delay they will experience. Based on this information they decide to wait for service or to leave the system. Specifically, every customer has a patience threshold and they balk if the observed delay is above the threshold.
The main objective is to estimate the parameters of the customers' patience-level distribution and the corresponding potential arrival rate, using knowledge of the actual queue-length process only. The main complication, and distinguishing feature of our setup, lies in the fact that customers who decide not to join are {\it not} observed, remarkably, we manage to devise a procedure to estimate the underlying patience and arrival rate parameters.

\noindent 
 The model is a multi-server queue with a Poisson stream of customers, enabling evaluation of the corresponding likelihood function of the state-dependent effective arrival process. We establish strong consistency of the MLE and derive the asymptotic distribution of the estimation error. Several applications and extensions of the method are discussed. The performance is further assessed through a series of numerical experiments. By fitting parameters of hyperexponential and generalized-hyperexponential distributions our method provides a robust estimation framework for any continuous patience-level distribution.

\vb


\vb

\noindent
{\sc Affiliations.} 
Yoshiaki Inoue is with Department of Information and Communications Technology, Graduate School of Engineering, Osaka University, 2-1 Yamadaoka, Suita, Osaka 565-0871, Japan ({\tt yoshiaki@comm.eng.osaka-u.ac.jp}).
Liron Ravner is with the Department of Statistics, University of Haifa, Israel ({\tt lravner@haifa.stat.ac.il}), and Korteweg-de Vries Institute for Mathematics, University of Amsterdam.
Michel Mandjes is with Korteweg-de Vries Institute for Mathematics, University of Amsterdam, Science Park 904, 1098 XH Amsterdam, the Netherlands ({\tt m.r.h.mandjes@uva.nl}), E{\sc urandom}, Eindhoven University of Technology, Eindhoven, the Netherlands, and Amsterdam Business School, Faculty of Economics and Business, University of Amsterdam, Amsterdam, the Netherlands.
The research of YI is supported in part by JSPS KAKENHI Grant Number JP18K18007. 
The research of LR and MM is partly funded by NWO Gravitation project N{\sc etworks}, grant number 024.002.003. Date: \today.

\vb

\noindent
{\sc Acknowledgments.} 
The authors are grateful to the associate editor and two referees for their helpful comments and feedback. The authors also wish to thank Onno Boxma (Eindhoven University of Technology, the Netherlands), Avi Mandelbaum (Technion, Haifa, Israel), Ran Snitkovsky (Columbia University \& CUHK Shenzhen), and Galit Yom-Tov (Technion, Haifa, Israel) for providing useful feedback. 

\end{abstract}

\maketitle

\newpage

\section{Introduction} 

In many service systems potential customers are provided with real-time information on the delay they will experience \cite{DONG,IBR}. This information, typically in terms of anticipated delay or current workload, is used by the potential customers to decide to either join the system or to balk. 
In general, one could view balking as an implicit form of admission control: the load on the service system does not explode as a consequence of the fact that during busy periods some of the potential customers decide to leave, thus mitigating the overall congestion. This means that the system is in effect self-regulating even if the \emph{potential service demand} is unknown to the administrator. To make sound decisions in designing the service system, however, it is crucial to have knowledge of the volume of the potential demand and the mechanism based on which potential customers decide to join or to leave. 
With this knowledge, the service provider can, for example, determine the optimal admission price or decide whether it is economically viable to increase the service capacity. 

This paper is motivated by the fact that it is  challenging to extract this knowledge from data when \textit{balking customers cannot be observed by the administrator}. Such complication arises naturally in many service systems. The most obvious case is a physical service facility where a visible queue (e.g., shops, parking lots, museums), sometimes even outside of the facility, will deter some of the potential customers. This has become even more prevalent recently with social distancing measures limiting the number of people that can be inside a public space at a given time. This form of data censorship occurs also in systems that provide real-time delay data via electronic means such as dedicated apps for the service. Some examples are expected travel times in transportation services (e.g., Google Maps and Uber) and expected delivery times in food delivery services (e.g.\ , UberEats in the US, or Deliveroo in Europe). Note that expected delay is often listed for the possible services (e.g., estimated arrival time of a taxi or delivery times for a list of Italian restaurants) even if a user does not indicate the specific service she is interested in, and therefore the system cannot know whether a user balked from the service due to the delay information. 

In this paper, we model such service systems with unobservable balking as an M/G/$s$\,+\,$H$ queue, where the last symbol stands for the cumulative distribution function of the customer patience. More specifically, we consider a first-come first-served (FCFS) queue with $s$ ($s\geqslant 1$) servers and potential customers arriving according to a Poisson process with rate $\lambda>0$. Potential customers have independent and identically distributed (i.i.d.) patience levels that are distributed as the non-negative random variable $Y \sim H$ and bring i.i.d.\ service requirements that are distributed as the non-negative random variable $B\sim G$.   
Customer $i$ joins the system if $Y_i$ is greater or equal to the virtual waiting time at the moment of their arrival and otherwise balks without being observed. Note that while customers are homogeneous in a statistical sense, they have individual (thus heterogeneous) patience levels, so that this model can clearly represent the types of customers that will effectively join the system. 

We set up an estimation problem by parametrizing the patience distribution as $H_{\theta}$ ($\theta \in \Theta$), where $\Theta$ denotes a compact set satisfying some regularity conditions.
The goal of this work is to estimate the pair $(\lambda,\theta)$ based on observation of the \textit{effective} queueing process, which is constructed from records of arrival and departure times not including balking customers. As some of the customers balk, the corresponding {\it effective} interarrival times are {\it not} the usual i.i.d.\ exponentially distributed random variables. The {\em effective} arrival process is in fact a non-homogeneous Poisson process whose rate depends on the virtual waiting time process.  In particular, the effective arrival rate to the system when the virtual waiting time $v\geqslant 0$ is $\lambda(1-H(v))$.  Note that the marginal effective arrival process is not a non-homogeneous Poisson process with respect to time. Therefore, the dependence structure of the observations of inter-arrival times is directly determined by that of the virtual waiting time process. The workload-dependent arrival dynamics are utilized to derive a maximum-likelihood estimator (MLE) based on the effective inter-arrival observations.  Specifically, the density of an inter-arrival time A conditional on initial virtual waiting time $v\geqslant 0$ is given by
\begin{align*}
f_{A \,|\, v}(t)
=
\lambda\, (1-H(v-t))
\exp\left({-\lambda \int_0^t (1-H(v-u))\diff u}\right),\quad t\geqslant 0.
\end{align*}
The Markovian dynamics of the queue further imply that the likelihood of the inter-arrival times conditional on the waiting times and job sizes is given by a product of densities with the above form for any sample of observations.

We further study the asymptotic properties of the proposed estimator. Importantly, the asymptotic performance of the estimator is not given by standard results due to the intricate underlying dependence structure. Indeed, so as to prove consistency and asymptotic normality, a subtle reasoning is required to make sure that the situation at hand fulfils specific regularity conditions. It is further shown that for homogeneous (deterministic) patience levels, i.e., $\P(Y=\theta)=1$ for some $\theta>0$, the asymptotic distribution of the errors is not normal, but rather exponential (with rate $n$ and not $\sqrt{n}$). The main contribution of this paper is the development of procedures for consistent estimation of the {\it total} arrival rate (corresponding to customers joining the system {\it and} balking customers, that is) and the patience-level distribution, observing the {queue-length} process only. Hence, somewhat counterintuitively, we can estimate load that has never been observed.

The main purpose of this paper is to lay foundations for statistical inference for congested service systems with unobserved balking. Although the queueing model considered in this paper is of course not a perfect replica of any specific real system, it does provide a methodological statistical approach for service systems where unobserved balking due to congestion is a key feature. Furthermore, the framework presented here can be extended to capture additional features, such as noisy waiting time information, as will be discussed in the concluding remarks.

It is also important to stress that the framework we develop in this paper is essentially different from the extensively studied problem of estimating patience parameters from observed abandonment data (see \cite{BGMSSZZ2005,MZ2013}). If balking customers are observed, then the data is comprised of waiting times and corresponding indicator variables stating whether a customer joined or balked. Specifically, when it is observed that a customer $i$  joined (resp.\ balked) given the waiting time $W_i$, the observation is of the form $Y\geqslant W_i$ (resp.\ $Y<W_i$). This data can be directly applied to construct an estimator for $\theta$, for example by applying an MLE method. If abandonment times are directly observed then even more information is available, namely the realization of $Y_i$ for every abandoning customer, and $Y\geqslant W_i$ for customers who obtained service. Note that in the above scenarios estimation of the total arrival rate is straightforward because all arrivals are observed. However, when balking is unobserved the only information available is that of inter-arrival and waiting times, and therefore indirect estimation methods are called for. Of course, there may be applications where some of the balking or reneging customers are observed, while others are not. For such systems it may be reasonable to adopt a hybrid approach that uses classical patience estimation methods together with those presented in this paper. For the sake of brevity we assume throughout that balking happens immediately and is not observed by the system operator, so that no data regarding the balking customers is available for estimation purposes.

\subsection{Contributions}
We proceed by detailing our paper's contributions. 
The focus is on developing estimators with provable performance guarantees. Notably, setting up our estimator is straightforward in the sense that it relies on a closed-form expression for the likelihood, and does not require any queueing-theoretic analysis. As a result, even for models for which stationary performance analysis is involved (or intractable), parameter inference can be done in a relatively direct manner.
Our contributions relate to (i)~our estimator and its asymptotic properties, (ii)~results and experiments for specific patience-level distributions, and (iii)~extensions and ramifications.

\begin{itemize}
\item[$\circ$]
Framing the system described above as a queueing model with impatient customers, we can evaluate the corresponding likelihood function (in terms of the observed quantities). This means that, in a parametric context, we can estimate the unknown parameters --- i.e., $\lambda$ and the parameter(s) $\theta$ corresponding to the patience level $Y$ --- relying on a maximum likelihood procedure. The estimation method has several attractive features, the most prominent one being that it does not require simultaneous estimation of the service requirement distribution or even making any distributional assumption about it. In addition, our methodology uses the fundamental queueing dynamics and does not, as is common in the literature, rely on any fluid and/or diffusion approximations. 
\item[$\circ$] In the case of a continuous patience-level distribution we prove, under appropriate regularity conditions, strong consistency and asymptotic normality of the resulting estimation errors (scaled by $\sqrt{n}$). The proof of the asymptotic normality relies on an application of a suitable version of the martingale central limit theorem. 
\item[$\circ$]
In the case that the patience level $Y$ is constant (i.e., $Y=\theta$ for some $\theta$), more refined results can be obtained. Most notably, we show that the estimation error (scaled by $n$) converges to an exponential distribution with a parameter that depends on the stationary workload density at $\theta$ and the loss probability of customers.
\item[$\circ$] A number of other special cases are dealt with, covering exponentially and (generalized) hyperexponentially distributed patience levels; we present closed-form expressions for the asymptotic variance of the estimation error. The class of generalized hyperexponentially distributions is highly relevant for various reasons, most importantly because these can be used to approximate any positive continuous distribution arbitrarily well \cite{BH1986}.
In addition, in the call center literature \cite{RJ2013} it has been observed that such distributions yield a good fit for observed customer patience {in terms of the time until abandonment.}  We provide several numerical examples that illustrate the robustness of our methodology. Indeed, our experiments show that even if the patience-level distribution is misspecified (i.e., is not generalized hyperexponential itself),
the distribution function of the estimated generalized hyperexponential distribution is still remarkably close to the true one. 

\item[$\circ$]
We also discuss several ramifications and model extensions that can be handled using our framework. These include extending the framework to {a system with} noisy delay information, estimating the proportion of impatient customers in the population, and estimating specific utility parameters such as service value and waiting time sensitivity.
\end{itemize}

\subsection{Related literature}\label{sec:lit}
There is a substantial body of literature on queues with impatient customers. Without attempting to provide an exhaustive overview, we include a brief account of the various research lines.
Naor's influential paper \cite{NA}  can be seen as a pioneering contribution to the field of `{behavioral queues}'. It presents a powerful stylized model for a queueing resource in which potential customers decide to join or to balk based on delay information. Since then performance and economic analysis of queueing systems with balking customers has been studied extensively. In particular, 
the effect of providing workload information to strategic customers was investigated in e.g.\  \cite{GZ2007,GZ2008}. Many other model variations are reviewed in \cite{H2016, HH2003}. More background on applied probability and queueing systems in a general context is provided in e.g.\ \cite{AS,DM}.
 
For various specific types of {queueing systems with impatient customers} explicit performance results are known; see, for example, \cite{BPS2011,BPSZ2010,LK2006,LK2008,MS2000,MS2004}. Importantly, however, the estimation techniques presented in our paper do not require knowledge of e.g.\ the queue's stationary distribution. It is also worth stressing that despite the fact that queues with impatient customers are well studied, to the best of our knowledge, hardly any workload-based estimation techniques have been developed. A notable exception is  \cite{GRH2009}, featuring an asymptotic analysis of statistical inference of waiting times truncated by abandonments; our setting, however, is crucially different from the one in \cite{GRH2009}, as in the latter
 abandoning customers and their sojourn times are observed.
 
There are various papers on estimation problems in the setting where all customers join, but leave when getting impatient; examples are \cite{AAES2013, AC, MZ2013, YEF}.  Importantly, this setting is crucially different from ours, in which the balking customers are not observed. In this branch of the literature the challenge lies in the fact that for customers who {\it are} served, we have just a lower bound on their impatience time, making this a `right-censored' estimation problem.

Another related research direction concerns the so-called {\it queue inference engine}: a setting in which one does not observe the arrival process of the customers, but rather the ordered service entry and service completion times (so-called `transactional data'); see for example \cite{DS98,L90}. This situation one comes across when considering e.g.\  automatic teller machines where the queue of customers waiting for the machine (and in particular their arrival times) is not observable, but the transactional data is recorded. In \cite{DS98} a multi-server queue is considered, but with the specific feature that balking is assumed to occur when all servers are busy upon arrival, with a fixed probability independent of the queue-length, making the underlying queueing dynamics fundamentally different from ours.

Strictly speaking, \cite{AN19,AN19a} do not deal with transactional data: it is assumed that the interarrival times corresponding to the non-balking customers are observed (so that its scope is somewhat similar to ours). At the methodological level, however, \cite{AN19,AN19a} crucially differ from our approach: the proposed estimation procedure relies on discrete-event simulation rather than a rigorously backed maximum likelihood procedure.
 
 Our work fits in the larger branch of the literature on inference for queues. This domain focuses on estimation problems where the queue is observed (for instance periodically), with the objective to estimate its input parameters. We refer to \cite{ANT2021} for a broad recent overview of this area. Examples are the procedures for the M/G/1 queue as presented in \cite{dBM, HP2006}. We also refer to the Poisson-probing based approach  \cite{RBM2019} for L\'evy-driven queues \cite{DM}, and  the general framework presented in \cite{DG}, as well as the corresponding hypothesis-testing problem \cite{MRav}.
 Some approaches rely on exploiting knowledge of the queue's tail probabilities \cite{CR,DMe,GZ}. A related line of research concerns the detection of stability or instability of queueing networks; see the Monte Carlo based methods of \cite{LM,MPW}.

The main conclusion from the above is that, despite the mature literature on queues with impatient customers and related estimation problems, there is a lack of methodologies that can deal with the situation in which balking customers are not observed. 

\vb

\subsection{Paper organization}\label{sec:organiz}
{This paper is organized as follows. In Section \ref{sec:model} we describe the model featuring in this paper and introduce the notation used. In addition,  we present a number of known results on queues with customer impatience that are relied on later. Then, in Section \ref{sec:arrival_workload}, we construct and analyze our maximum likelihood estimator.
For continuous patience-level distributions,  {we prove strong consistency and asymptotic normality of the estimation error. The case of deterministic patience is treated separately, leading to a similar conclusion regarding strong consistency, but with the notable difference that in this case the limiting distribution of the estimation error is exponential.} Section \ref{sec:special} provides detailed analysis of the MLE for exponential, hyperexponential and generalized hyperexponential patience distributions, {the latter cases being relevant as they can be used to approximate the distributions of arbitrary non-negative random variables}. This includes explicit analysis of the likelihood function and its corresponding asymptotic variance, as well as simulation experiments assessing the performance of the estimators. Section \ref{sec:app} illustrates how the method can be utilized for several applications. In Section \ref{sec:Conc} we provide concluding remarks and directions for followup research, including
a brief discussion of a non-parametric estimation procedure.}

\section{Model and preliminaries}\label{sec:model}

In this section we provide a detailed model description, introduce the notation used throughout this paper, state the objective of the paper, and present preliminaries. 

\subsection*{Model description}

We consider a service system with $s$ servers ($s\in{\mathbb N}$).
Potential customers arrive at the service system according to a Poisson process with rate $\lambda$.
Each customer has a service requirement with cumulative distribution function (cdf) $G(\cdot)$ and
a patience level with cdf $H(\cdot)$.  We define $B$ and $Y$ as the generic random variables corresponding to service requirements and impatience levels, respectively. The interarrival times $(T_i)_{i=1,2\ldots}$, the
service requirements $(B_i)_{i=1,2,\ldots}$, and patience levels $(Y_i)_{i=1,2,\ldots}$
are independent sequences of i.i.d.\ distributed random variables. Customer $i$ joins the queue if $Y_i$ is smaller than or equal to the waiting time she will experience. More specifically, let $V(t)$ (for $t \geqslant 0$) denote the virtual waiting time at time $t$, which is defined as the time it would take before at least one server becomes idle if no customers have joined the system after time $t$:
\begin{equation}
V(t) = \inf\{u \geqslant 0:\, Q(t) - (N_{\mathrm{D}}(t+u) - N_{\mathrm{D}}(t)) \leqslant s-1\},
\label{eq:VWT-def}
\end{equation}
where $Q(t)$ denotes the number of customers in the system (including those in service) at time $t$ and $N_{\mathrm{D}}(t)$ (for $t \geqslant 0$) denotes the total number of customers finishing their service in a time-interval $(0,t]$.
For the single-server case (i.e., if $s=1$), the virtual waiting time equals to the workload in system, i.e., the sum of remaining service requirements of customers. Note that the virtual waiting time in the multi-server case ($s = 2,3,\ldots$) does not coincide with the workload in system, unlike in the single-server case: 
\begin{itemize}
\item[(i)] When $i \in\{0,1,\ldots, s-1\}$ servers are busy, the virtual waiting time equals zero, while the workload equals some positive value and decreases with slope $i$ as time passes without effective arrivals.

\item[(ii)] When all $s$ servers are busy, the virtual waiting time and the workload decrease with slope $1$ and $s$, respectively, as time passes without effective arrivals.
\end{itemize}
Figure \ref{fig:vwtime} illustrates an example of the realization of the virtual waiting time process for $s=2$. The actual waiting time of a customer arriving at time $t$ equals $V(t-)$, so that a customer balks if the virtual waiting time just before the arrival instant exceeds her patience threshold.
The resulting system can be considered as an M/G/$s$ queue with impatient customers. 

\begin{figure}[t]
\centering
\includegraphics[width=0.95\textwidth]{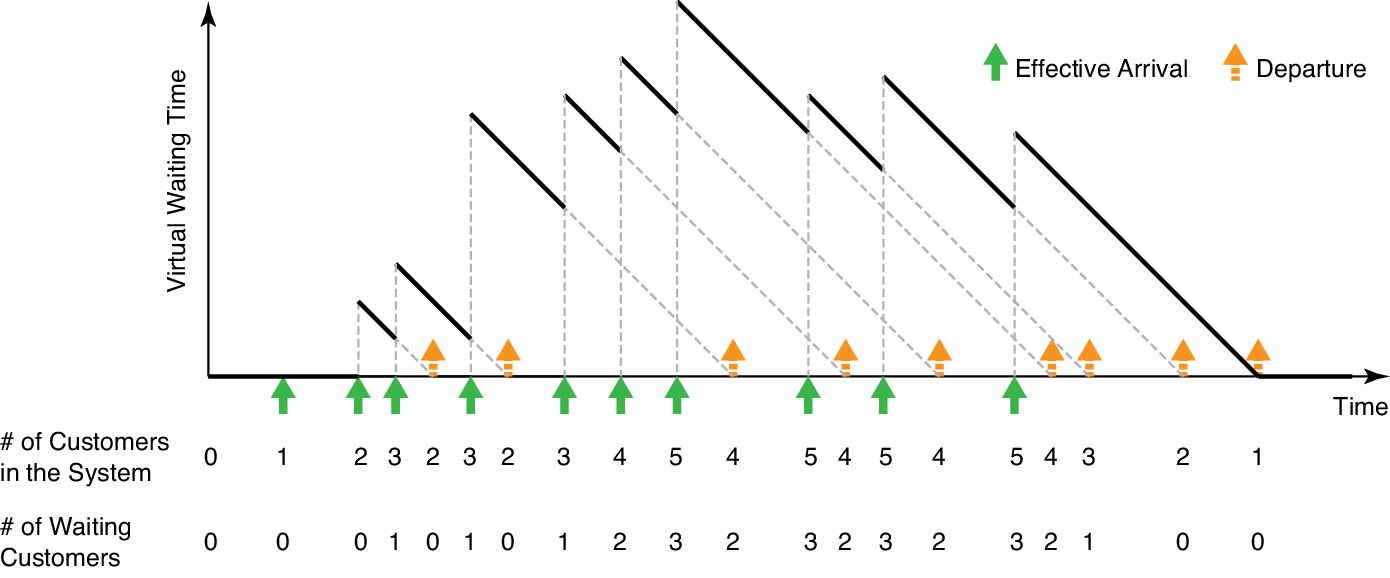}
\caption{An example of the virtual waiting time process ($s=2$).}
\label{fig:vwtime}
\end{figure}

As some customers will balk, there will be a difference between the interarrival times of the customers (covering both the balking and non-balking customers) on one hand, and the  \textit{effective} interarrival times (covering just the non-balking customers) on the other hand. 
Throughout this paper the sequence of effective interarrival times is denoted by $(A_i)_{i=1,2,\ldots}$.
Importantly, observe that the random variables 
$A_i$ are generally neither identically distributed nor independent.
 We can reconstruct them from the sequences $(T_i)_{i=1,2,\ldots}$,
$(B_i)_{i=1,2,\ldots}$, 
and $(Y_i)_{i=1,2,\ldots}$ in the following manner. For the following construction we let a non-balking customer arrive at time $t=0$, as will be further motivated in Remark~\ref{R1}. Denoting the virtual waiting time process at
time $t$ by $V(t)$, we have that $A_1=\sum_{k=1}^{j_1} T_k$ where 
\[
j_1:=\inf\left\{j\geqslant 1: \ V\left(\sum_{k=1}^j T_k-\right) \leqslant Y_j\right\}.
\] 
Similarly $A_i=\sum_{k=j_{i-1}+1}^{j_i} T_k$ where 
\[
j_i:=\inf\left\{j \geqslant 
{j_{i-1}+1}
: 
\ V\left(\sum_{k=1}^j T_k-\right)\leqslant Y_j\right\}.
\] 
Observe that $j_i$ has the interpretation of being the index of the $i$-th non-balking customer. 
Let $\tilde{A}_i := \sum_{k=1}^i A_k$ denote the $i$-th effective arrival time. 
We denote the upward-jump size in the virtual waiting time caused by the $i$-th joining customer by $X_i:= V(\tilde{A}_i) - V(\tilde{A}_i-)$.
In the single-server case ($s=1$), we have $X_i = B_{j_i}$. In the multi-server case ($s \geqslant 2$), on the other hand, $X_i$ takes a complicated form depending on both the queue-length and the residual service times of customers in service seen on the arrival instant.  For a single-server system ($s=1$),  Figure \ref{fig:workload} illustrates
the workload process during a single busy period for the case of deterministic patience, i.e., $Y=\theta$ for some $\theta>0$. The building blocks of our framework, viz.\ the interarrival and service requirements and their \textit{effective} counterparts, are highlighted in the figure. 
It is stressed that the distribution of the effective interarrival times can significantly deviate from the (exponential) distribution of the interarrival times. 
Figure \ref{fig:arrival_dist} 
illustrates this effect for two different patience-level distributions. In addition, as mentioned, the effective interarrival times are not independent, as opposed to the interarrival times.

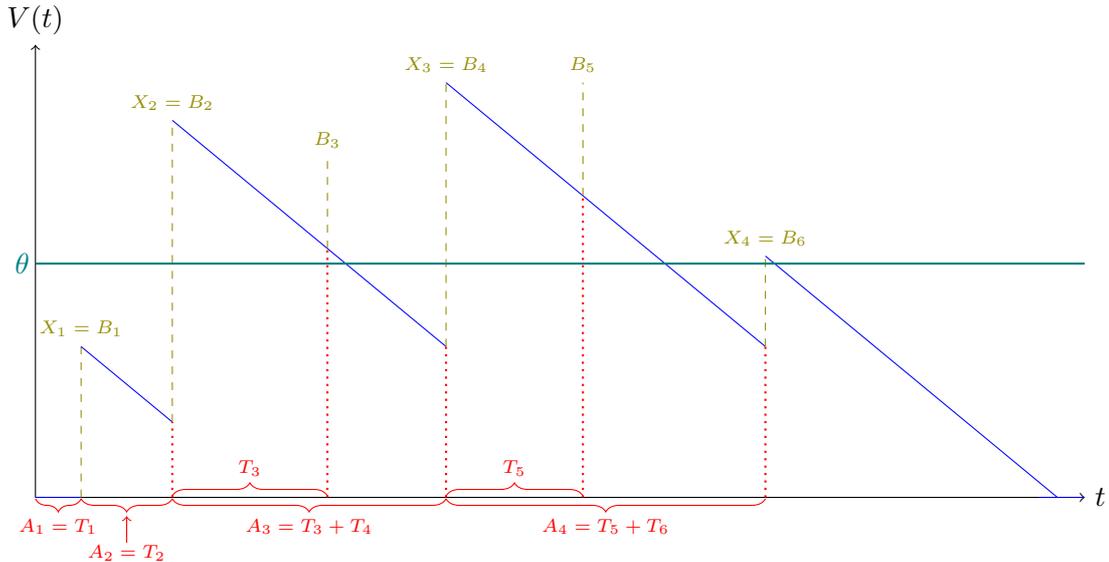
\begin{figure}[h]
\centering
\begin{tikzpicture}[x=1.2cm,y=1cm]
 \def\xmin{0}
 \def\xmax{11.5}
 \def\ymin{0}
 \def\ymax{6}
 \draw[->] (\xmin,\ymin) -- (\xmax,\ymin) node[right] {$t$} ;
 \draw[->] (\xmin,\ymin) -- (\xmin,\ymax) node[above] {$V(t)$} ;
\foreach \x in {} {
 \node at (\x,\ymin) [below] {\x};
 \draw[-] (\x,\ymin) -- (\x,{\ymin-0.05});
 }
\foreach \y in {} {
 \node at (\xmin,\y) [left] {\y};
 \draw[-] (\xmin,\y) -- ({\xmin-0.05},\y);
 }
 
\draw[- ,blue] (0,0) -- (0.5,0) ;
\draw[- ,dashed,olive]  (0.5,0) -- (0.5,2);
\draw[- ,blue] (0.5,2) -- (1.5,1) ;
\draw[- ,dashed,olive]  (1.5,1) -- (1.5,5);
\draw[- ,blue] (1.5,5) -- (4.5,2) ;
\draw[- ,dashed,olive]  (4.5,2) -- (4.5,5.5);
\draw[- ,blue] (4.5,5.5) -- (8,2) ;
\draw[- ,dashed,olive]  (8,2) -- (8,3.2);
\draw[- ,blue] (8,3.2) -- (11.2,0) ;
\draw[- ,blue]  (11,0) -- (11.5,0) ;

 \draw [ red,decorate,decoration={brace,amplitude=5pt,mirror},yshift=-0.5pt](0,0) -- (0.5,0);
 \node[draw=none,below,color=red] at (0.25,-0.15) {\tiny{$A_1=T_1$}};
 \node[draw=none,above,color=olive] at (0.5,2) {\tiny{$X_1=B_1$}};

 \draw[-,dotted,thick,red] (1.5,0) -- (1.5,1);
 \draw [ red,decorate,decoration={brace,amplitude=5pt,mirror},yshift=-0.5pt](0.5,0) -- (1.5,0);
 \node[draw=none,below,color=red] at (1,-0.5) {\tiny{$A_2=T_2$}};
 \draw[->,red] (1,-0.6) -- (1,-0.25);
 \node[draw=none,above,color=olive] at (1.5,5) {\tiny{$X_2=B_2$}};
 
  \draw[-,dotted,thick,red] (3.2,0) -- (3.2,3.3);
  \draw [ red,decorate,decoration={brace,amplitude=5pt},yshift=0.5pt](1.5,0) -- (3.2,0);
 \node[draw=none,above,color=red] at (2.35,0.15) {\tiny{$T_3$}};
 \draw[-,olive,dashed] (3.2,3.3) -- (3.2,4.5);
 \node[draw=none,above,color=olive] at (3.2,4.5) {\tiny{$B_3$}};
 
 \draw[-,dotted,thick,red] (4.5,0) -- (4.5,2);
 \draw [ red,decorate,decoration={brace,amplitude=5pt,mirror},yshift=-0.5pt](1.5,0) -- (4.5,0);
 \node[draw=none,below,color=red] at (3,-0.15) {\tiny{$A_3=T_3+T_4$}};
  \node[draw=none,above,color=olive] at (4.5,5.5) {\tiny{$X_3=B_4$}};
 
 \draw[-,dotted,thick,red] (6,0) -- (6,4);
 \draw [ red,decorate,decoration={brace,amplitude=5pt},yshift=0.5pt](4.5,0) -- (6,0);
 \node[draw=none,above,color=red] at (5.25,0.15) {\tiny{$T_5$}};
 \draw[-,olive,dashed] (6,4) -- (6,5.5);
 \node[draw=none,above,color=olive] at (6,5.5) {\tiny{$B_5$}};
 
 \draw[-,dotted,thick,red] (8,0) -- (8,2);
 \draw [ red,decorate,decoration={brace,amplitude=5pt,mirror},yshift=-0.5pt](4.5,0) -- (8,0);
 \node[draw=none,below,color=red] at (6.25,-0.15) {\tiny{$A_4=T_5+T_6$}};
 \node[draw=none,above,color=olive] at (8,3.2) {\tiny{$X_4=B_6$}};

 \draw[-, thick,teal] (0,3.1) -- (\xmax,3.1);
 \node[draw=none,left,color=teal] at (0.05,3.1) {$\theta$};
 
 \end{tikzpicture}
\caption{The virtual waiting time $V(t)$ during a busy period for a single-server system and deterministic patience levels. The dotted red lines mark arrival instants of new customers. Arrivals 3 and 5 observe a virtual waiting time level above $\theta$ and immediately balk, and thus the virtual waiting time continues to deplete as if there was no arrival. Although six customers have arrived during the busy period, only four effective interarrival times and service requirements are observed, namely
{$( (A_1,X_1), (A_2,X_2), (A_3,X_3),(A_4,X_4))$, being equal to $( (T_1,B_1), (T_2,B_2), (T_3+T_4,B_4),(T_5+T_6,B_6) )$}.}
\label{fig:workload}
\mbox{}
\end{figure}

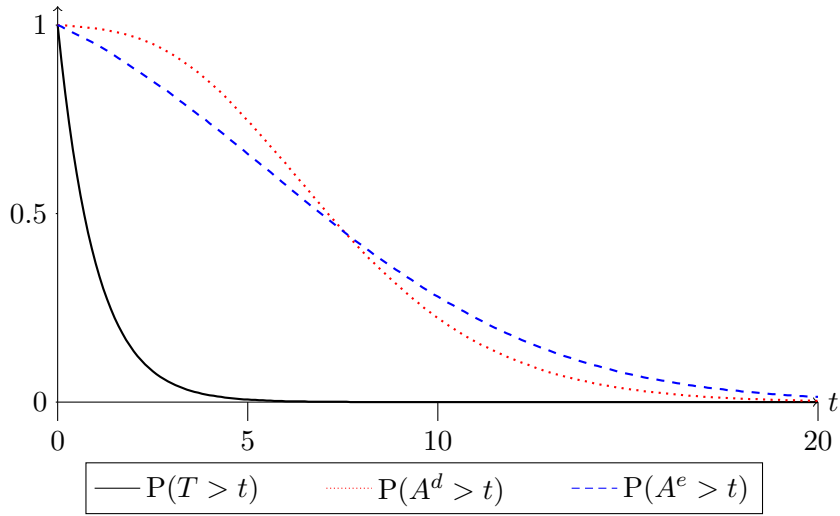
\begin{figure}
\centering
\begin{tikzpicture}[x=0.6cm,y=0.6cm]
 \def\xmin{0}
 \def\xmax{13.05}
 \def\ymin{0}
 \def\ymax{13.05}
 \draw[->] (\xmin,\ymin) -- (\xmax,\ymin) node[right] {$t$} ;
 \draw[->] (\xmin,\ymin) -- (\xmin,\ymax)  ;
 \foreach \x in {0,5,10} {
 \node at (\x,\ymin-0.05) [below] {\x};
 \draw[-] (\x,\ymin) -- (\x,{\ymin-0.05});
 }
\foreach \y in {0,5,10} {
 \node at (\xmin,\y) [left] {\y};
 \draw[-] (\xmin,\y) -- ({\xmin-0.05},\y);
 }


\draw[black,thick] ( 0 , 0 )-- ( 0.1 , 0.1 )-- ( 0.2 , 0.2 )-- ( 0.3 , 0.3 )-- ( 0.4 , 0.4 )-- ( 0.5 , 0.5 )-- ( 0.6 , 0.6 )-- ( 0.7 , 0.7 )-- ( 0.8 , 0.8 )-- ( 0.9 , 0.9 )-- ( 1 , 1 )-- ( 1.1 , 1.1 )-- ( 1.2 , 1.2 )-- ( 1.3 , 1.3 )-- ( 1.4 , 1.4 )-- ( 1.5 , 1.5 )-- ( 1.6 , 1.6 )-- ( 1.7 , 1.7 )-- ( 1.8 , 1.8 )-- ( 1.9 , 1.9 )-- ( 2 , 2 )-- ( 2.1 , 2.1 )-- ( 2.2 , 2.2 )-- ( 2.3 , 2.3 )-- ( 2.4 , 2.4 )-- ( 2.5 , 2.5 )-- ( 2.6 , 2.6 )-- ( 2.7 , 2.7 )-- ( 2.8 , 2.8 )-- ( 2.9 , 2.9 )-- ( 3 , 3 )-- ( 3.1 , 3.1 )-- ( 3.2 , 3.2 )-- ( 3.3 , 3.3 )-- ( 3.4 , 3.4 )-- ( 3.5 , 3.5 )-- ( 3.6 , 3.6 )-- ( 3.7 , 3.7 )-- ( 3.8 , 3.8 )-- ( 3.9 , 3.9 )-- ( 4 , 4 )-- ( 4.1 , 4.1 )-- ( 4.2 , 4.2 )-- ( 4.3 , 4.3 )-- ( 4.4 , 4.4 )-- ( 4.5 , 4.5 )-- ( 4.6 , 4.6 )-- ( 4.7 , 4.7 )-- ( 4.8 , 4.8 )-- ( 4.9 , 4.9 )-- ( 5 , 5 )-- ( 5.1 , 5.1 )-- ( 5.2 , 5.2 )-- ( 5.3 , 5.3 )-- ( 5.4 , 5.4 )-- ( 5.5 , 5.5 )-- ( 5.6 , 5.6 )-- ( 5.7 , 5.7 )-- ( 5.8 , 5.8 )-- ( 5.9 , 5.9 )-- ( 6 , 6 )-- ( 6.1 , 6.1 )-- ( 6.2 , 6.2 )-- ( 6.3 , 6.3 )-- ( 6.4 , 6.4 )-- ( 6.5 , 6.5 )-- ( 6.6 , 6.6 )-- ( 6.7 , 6.7 )-- ( 6.8 , 6.8 )-- ( 6.9 , 6.9 )-- ( 7 , 7 )-- ( 7.1 , 7.1 )-- ( 7.2 , 7.2 )-- ( 7.3 , 7.3 )-- ( 7.4 , 7.4 )-- ( 7.5 , 7.5 )-- ( 7.6 , 7.6 )-- ( 7.7 , 7.7 )-- ( 7.8 , 7.8 )-- ( 7.9 , 7.9 )-- ( 8 , 8 )-- ( 8.1 , 8.1 )-- ( 8.2 , 8.2 )-- ( 8.3 , 8.3 )-- ( 8.4 , 8.4 )-- ( 8.5 , 8.5 )-- ( 8.6 , 8.6 )-- ( 8.7 , 8.7 )-- ( 8.8 , 8.8 )-- ( 8.9 , 8.9 )-- ( 9 , 9 )-- ( 9.1 , 9.1 )-- ( 9.2 , 9.2 )-- ( 9.3 , 9.3 )-- ( 9.4 , 9.4 )-- ( 9.5 , 9.5 )-- ( 9.6 , 9.6 )-- ( 9.7 , 9.7 )-- ( 9.8 , 9.8 )-- ( 9.9 , 9.9 )-- ( 10 , 10 )-- ( 10.1 , 10.1 )-- ( 10.2 , 10.2 )-- ( 10.3 , 10.3 )-- ( 10.4 , 10.4 )-- ( 10.5 , 10.5 )-- ( 10.6 , 10.6 )-- ( 10.7 , 10.7 )-- ( 10.8 , 10.8 )-- ( 10.9 , 10.9 )-- ( 11 , 11 )-- ( 11.1 , 11.1 )-- ( 11.2 , 11.2 )-- ( 11.3 , 11.3 )-- ( 11.4 , 11.4 )-- ( 11.5 , 11.5 )-- ( 11.6 , 11.6 )-- ( 11.7 , 11.7 )-- ( 11.8 , 11.8 )-- ( 11.9 , 11.9 )-- ( 12 , 12 )-- ( 12.1 , 12.1 )-- ( 12.2 , 12.2 )-- ( 12.3 , 12.3 )-- ( 12.4 , 12.4 )-- ( 12.5 , 12.5 )-- ( 12.6 , 12.6 )-- ( 12.7 , 12.7 )-- ( 12.8 , 12.8 )-- ( 12.9 , 12.9 )-- ( 13 , 13 );

 \draw[red,dotted,thick] ( 0 , 0 )-- ( 0.1 , 0.01 )-- ( 0.2 , 0.02 )-- ( 0.3 , 0.02 )-- ( 0.4 , 0.03 )-- ( 0.5 , 0.04 )-- ( 0.6 , 0.05 )-- ( 0.7 , 0.06 )-- ( 0.8 , 0.07 )-- ( 0.9 , 0.08 )-- ( 1 , 0.1 )-- ( 1.1 , 0.11 )-- ( 1.2 , 0.12 )-- ( 1.3 , 0.14 )-- ( 1.4 , 0.16 )-- ( 1.5 , 0.18 )-- ( 1.6 , 0.2 )-- ( 1.7 , 0.22 )-- ( 1.8 , 0.24 )-- ( 1.9 , 0.26 )-- ( 2 , 0.29 )-- ( 2.1 , 0.31 )-- ( 2.2 , 0.34 )-- ( 2.3 , 0.37 )-- ( 2.4 , 0.4 )-- ( 2.5 , 0.44 )-- ( 2.6 , 0.47 )-- ( 2.7 , 0.51 )-- ( 2.8 , 0.54 )-- ( 2.9 , 0.58 )-- ( 3 , 0.62 )-- ( 3.1 , 0.66 )-- ( 3.2 , 0.71 )-- ( 3.3 , 0.75 )-- ( 3.4 , 0.79 )-- ( 3.5 , 0.84 )-- ( 3.6 , 0.89 )-- ( 3.7 , 0.94 )-- ( 3.8 , 0.99 )-- ( 3.9 , 1.04 )-- ( 4 , 1.09 )-- ( 4.1 , 1.15 )-- ( 4.2 , 1.2 )-- ( 4.3 , 1.26 )-- ( 4.4 , 1.31 )-- ( 4.5 , 1.37 )-- ( 4.6 , 1.43 )-- ( 4.7 , 1.49 )-- ( 4.8 , 1.55 )-- ( 4.9 , 1.61 )-- ( 5 , 1.68 )-- ( 5.1 , 1.74 )-- ( 5.2 , 1.8 )-- ( 5.3 , 1.87 )-- ( 5.4 , 1.94 )-- ( 5.5 , 2 )-- ( 5.6 , 2.07 )-- ( 5.7 , 2.14 )-- ( 5.8 , 2.21 )-- ( 5.9 , 2.28 )-- ( 6 , 2.35 )-- ( 6.1 , 2.42 )-- ( 6.2 , 2.49 )-- ( 6.3 , 2.57 )-- ( 6.4 , 2.64 )-- ( 6.5 , 2.71 )-- ( 6.6 , 2.79 )-- ( 6.7 , 2.87 )-- ( 6.8 , 2.94 )-- ( 6.9 , 3.02 )-- ( 7 , 3.1 )-- ( 7.1 , 3.17 )-- ( 7.2 , 3.25 )-- ( 7.3 , 3.33 )-- ( 7.4 , 3.41 )-- ( 7.5 , 3.49 )-- ( 7.6 , 3.57 )-- ( 7.7 , 3.65 )-- ( 7.8 , 3.73 )-- ( 7.9 , 3.82 )-- ( 8 , 3.9 )-- ( 8.1 , 3.98 )-- ( 8.2 , 4.07 )-- ( 8.3 , 4.15 )-- ( 8.4 , 4.23 )-- ( 8.5 , 4.31 )-- ( 8.6 , 4.4 )-- ( 8.7 , 4.49 )-- ( 8.8 , 4.57 )-- ( 8.9 , 4.66 )-- ( 9 , 4.74 )-- ( 9.1 , 4.83 )-- ( 9.2 , 4.91 )-- ( 9.3 , 5 )-- ( 9.4 , 5.08 )-- ( 9.5 , 5.17 )-- ( 9.6 , 5.26 )-- ( 9.7 , 5.35 )-- ( 9.8 , 5.44 )-- ( 9.9 , 5.53 )-- ( 10 , 5.62 )-- ( 10.1 , 5.7 )-- ( 10.2 , 5.79 )-- ( 10.3 , 5.88 )-- ( 10.4 , 5.97 )-- ( 10.5 , 6.07 )-- ( 10.6 , 6.16 )-- ( 10.7 , 6.25 )-- ( 10.8 , 6.35 )-- ( 10.9 , 6.44 )-- ( 11 , 6.54 )-- ( 11.1 , 6.62 )-- ( 11.2 , 6.71 )-- ( 11.3 , 6.81 )-- ( 11.4 , 6.91 )-- ( 11.5 , 7 )-- ( 11.6 , 7.09 )-- ( 11.7 , 7.19 )-- ( 11.8 , 7.28 )-- ( 11.9 , 7.37 )-- ( 12 , 7.47 )-- ( 12.1 , 7.57 )-- ( 12.2 , 7.66 )-- ( 12.3 , 7.76 )-- ( 12.4 , 7.85 )-- ( 12.5 , 7.94 )-- ( 12.6 , 8.03 )-- ( 12.7 , 8.13 )-- ( 12.8 , 8.23 )-- ( 12.9 , 8.33 )-- ( 13 , 8.42 );
 
\draw[blue,dashed,thick] ( 0 , 0 )-- ( 0.1 , 0.01 )-- ( 0.2 , 0.01 )-- ( 0.3 , 0.02 )-- ( 0.4 , 0.02 )-- ( 0.5 , 0.03 )-- ( 0.6 , 0.04 )-- ( 0.7 , 0.05 )-- ( 0.8 , 0.06 )-- ( 0.9 , 0.07 )-- ( 1 , 0.08 )-- ( 1.1 , 0.09 )-- ( 1.2 , 0.11 )-- ( 1.3 , 0.12 )-- ( 1.4 , 0.14 )-- ( 1.5 , 0.15 )-- ( 1.6 , 0.17 )-- ( 1.7 , 0.19 )-- ( 1.8 , 0.21 )-- ( 1.9 , 0.23 )-- ( 2 , 0.25 )-- ( 2.1 , 0.27 )-- ( 2.2 , 0.29 )-- ( 2.3 , 0.32 )-- ( 2.4 , 0.34 )-- ( 2.5 , 0.37 )-- ( 2.6 , 0.4 )-- ( 2.7 , 0.43 )-- ( 2.8 , 0.46 )-- ( 2.9 , 0.49 )-- ( 3 , 0.52 )-- ( 3.1 , 0.56 )-- ( 3.2 , 0.59 )-- ( 3.3 , 0.63 )-- ( 3.4 , 0.66 )-- ( 3.5 , 0.7 )-- ( 3.6 , 0.74 )-- ( 3.7 , 0.78 )-- ( 3.8 , 0.82 )-- ( 3.9 , 0.86 )-- ( 4 , 0.91 )-- ( 4.1 , 0.95 )-- ( 4.2 , 1 )-- ( 4.3 , 1.04 )-- ( 4.4 , 1.09 )-- ( 4.5 , 1.14 )-- ( 4.6 , 1.19 )-- ( 4.7 , 1.24 )-- ( 4.8 , 1.29 )-- ( 4.9 , 1.34 )-- ( 5 , 1.39 )-- ( 5.1 , 1.45 )-- ( 5.2 , 1.5 )-- ( 5.3 , 1.55 )-- ( 5.4 , 1.61 )-- ( 5.5 , 1.67 )-- ( 5.6 , 1.73 )-- ( 5.7 , 1.78 )-- ( 5.8 , 1.84 )-- ( 5.9 , 1.9 )-- ( 6 , 1.96 )-- ( 6.1 , 2.03 )-- ( 6.2 , 2.09 )-- ( 6.3 , 2.15 )-- ( 6.4 , 2.22 )-- ( 6.5 , 2.28 )-- ( 6.6 , 2.35 )-- ( 6.7 , 2.41 )-- ( 6.8 , 2.48 )-- ( 6.9 , 2.55 )-- ( 7 , 2.61 )-- ( 7.1 , 2.68 )-- ( 7.2 , 2.75 )-- ( 7.3 , 2.82 )-- ( 7.4 , 2.89 )-- ( 7.5 , 2.96 )-- ( 7.6 , 3.03 )-- ( 7.7 , 3.11 )-- ( 7.8 , 3.18 )-- ( 7.9 , 3.25 )-- ( 8 , 3.33 )-- ( 8.1 , 3.4 )-- ( 8.2 , 3.48 )-- ( 8.3 , 3.55 )-- ( 8.4 , 3.63 )-- ( 8.5 , 3.71 )-- ( 8.6 , 3.78 )-- ( 8.7 , 3.86 )-- ( 8.8 , 3.94 )-- ( 8.9 , 4.02 )-- ( 9 , 4.1 )-- ( 9.1 , 4.18 )-- ( 9.2 , 4.26 )-- ( 9.3 , 4.34 )-- ( 9.4 , 4.42 )-- ( 9.5 , 4.5 )-- ( 9.6 , 4.58 )-- ( 9.7 , 4.66 )-- ( 9.8 , 4.74 )-- ( 9.9 , 4.83 )-- ( 10 , 4.91 )-- ( 10.1 , 4.99 )-- ( 10.2 , 5.08 )-- ( 10.3 , 5.16 )-- ( 10.4 , 5.24 )-- ( 10.5 , 5.33 )-- ( 10.6 , 5.41 )-- ( 10.7 , 5.5 )-- ( 10.8 , 5.58 )-- ( 10.9 , 5.66 )-- ( 11 , 5.75 )-- ( 11.1 , 5.83 )-- ( 11.2 , 5.92 )-- ( 11.3 , 6.01 )-- ( 11.4 , 6.09 )-- ( 11.5 , 6.18 )-- ( 11.6 , 6.27 )-- ( 11.7 , 6.35 )-- ( 11.8 , 6.44 )-- ( 11.9 , 6.53 )-- ( 12 , 6.62 )-- ( 12.1 , 6.71 )-- ( 12.2 , 6.8 )-- ( 12.3 , 6.89 )-- ( 12.4 , 6.97 )-- ( 12.5 , 7.06 )-- ( 12.6 , 7.15 )-- ( 12.7 , 7.24 )-- ( 12.8 , 7.34 )-- ( 12.9 , 7.43 )-- ( 13 , 7.52 );
 
 \end{tikzpicture}
 \begin{tikzpicture}
    \begin{customlegend}
    [legend entries={ $-\log\P(T>t)$,$-\log\P(A^d>t)$,$-\log\P(A^e>t)$},legend columns=-1,legend style={/tikz/every even column/.append style={column sep=0.8cm}}]   
    \addlegendimage{black} 
    \addlegendimage{red,densely dotted}    
    \addlegendimage{blue,densely dashed}    
    \end{customlegend}
\end{tikzpicture}
\caption{Logarithmic scale tail distribution (smooth black) of the potential interarrival time $T$ is exponential with rate $\lambda=1$. Simulation of the tail distributions of the stationary effective interarrival times is illustrated for two cases: $A^d$ corresponds to deterministic patience levels with $\theta=3$ (dotted red), while $A^e$ corresponds to exponentially distributed patience levels with mean~$3$ (dashed blue). The service requirements in the example are Erlang distributed with parameters $(5,1.5)$. The plotted stationary effective inter-arrival distributions are given by the empirical distribution obtained by simulating $n=10^5$ effective arrivals for each of the cases.} \label{fig:arrival_dist}
\end{figure}

We focus on the case that the patience-level distribution is determined by a parameter $\theta$ that is a vector in $\mathbb{R}^p$ for some (known) $p\in{\mathbb N}$, entailing the estimation problem is parametric. In addition, the arrival rate $\lambda$ is to be estimated, which is evidently of a parametric nature as well. Alternatively, $\theta$ could be a function from some more general space, rendering the estimation problem non-parametric; in Section \ref{sec:Conc} 
we briefly discuss an approach that can be used in this setting. Throughout this paper, we assume that we observe the full queue-length process. This in particular means that we observe the arrival and departure epochs of the non-balking customers ({and that we do not observe the balking customers at all}). The objective is to use this information to somehow learn the true arrival rate and the parameters of the patience-level distribution. More concretely, we wish to devise statistical procedures for estimating the arrival rate $\lambda$ and the distribution of the patience $Y$ (both corresponding to non-observable quantities) with provable performance properties.

\subsection*{Stationary distribution}

We assume that the mean service time $\E[B]$ is finite, and that, with $\rho := \lambda\E[B]/s$ denoting the traffic intensity, 
\begin{equation}
\rho\lim_{x \to \infty}(1-H(x)) < 1.
\label{eq:stability}
\end{equation}

Under these assumptions, we have that the queue is stable and that a stationary queue-length $Q := Q(\infty)$ and a stationary virtual waiting time $V:=V(\infty)$ exist (see \cite{BH1981} and \cite{BBH1984}.) Note that $\lim_{x \to \infty}H(x) < 1$ holds only if there are customers with infinite patience, i.e., {$\P(Y=\infty)>0$. If $\P(Y < \infty) = 1$}, on the other hand, then we have $\lim_{x \to \infty}H(x) = 1$, so that the system is stable irrespective of the value of the traffic intensity $\rho$. As we will see, the MLE method to be developed relies only on the transient dynamics of the system, so that it can be applied even if the stability condition \eqref{eq:stability} is not satisfied. On the other hand, we will utilize the system stability to discuss the asymptotic properties of the proposed estimator.

For the M/G/$s$ queue with impatient customers, exact expressions for the distributions of the stationary queue-length $Q$ and virtual waiting time $V$ are not known in the literature. However, we can derive the following relations they satisfy, in a similar way to the single-server case \cite{K1961}. Let $q_n := \P(Q=n)$ ($n = 0,1,\ldots$) denote the probability mass function of the stationary queue-length. The stationary virtual waiting time distribution has probability mass $\pi_0 := \P(V=0) = \sum_{n=0}^{s-1}q_n$ at zero and it is absolutely continuous on $(0,\infty)$. 
Let $X_{\mid y}$ denote a generic random variable for the stationary upward-jump size $X$ conditioned that it takes a positive value and that the immediately preceding virtual waiting time equals $y$:
\[
\P(X_{\mid y} \leqslant x) = \lim_{t \to \infty}\P(V(t)-V(t-) \leqslant x \mid V(t-) = y, V(t) > y).
\]
With the level-crossing argument, we can verify that the density function $v(x)$ of the stationary virtual waiting time satisfies the following integral equation (cf.\ \cite{BPS2011}): 
\begin{equation}
v(x) 
= 
\lambda q_{s-1} \,\overline{J}_{\mid 0}(x)
+ \lambda \int_0^x v(y)\, \P(Y\geqslant y)\,\overline{J}_{\mid y}(x-y) \diff y,
\quad
x \geqslant 0,
\label{eq:v-Volterra-multi}
\end{equation}
where $\overline{J}_{\mid y}(x) := \P(X_{\mid y} > x)$
denotes the complementary cdf of the conditional upward-jump size $X_{\mid y}$.

Owing to the {\sc pasta} property, the virtual waiting time seen by an arriving customer has the same distribution as the stationary virtual waiting time $V$.
The stationary waiting time $W$ of a non-balking customer is thus given 
by a conditional random variable $[V \,|\, V \leqslant Y]$. 
We can thus identify $\P(W=0)$ and the probability density function (pdf) $w(\cdot)$ of $W$ in terms of the
stationary virtual waiting time distribution $v(\cdot)$: 
\begin{equation}
\P(W=0) = \frac{\pi_0 }{1-P_{\ell}},
\quad
w(x) = \frac{v(x)\,\P(Y\geqslant x)}{1-P_{\ell}},
\;\;
x \geqslant 0,
\label{eq:w-general}
\end{equation}
where $P_{\ell}$ denotes the loss probability given by
\begin{equation}
P_{\ell} 
= 
\int_0^{\infty} v(y)\,\P(Y < y)\diff y.
\label{eq:P_loss}
\end{equation}
For $s=1$, \eqref{eq:v-Volterra-multi} simplifies to
\begin{equation}
v(x) 
= 
\lambda \pi_0 \,\overline{G}(x) 
+ \lambda \int_0^x v(y) \,\P(Y\geqslant y)\,\overline{G}(x-y) \diff y,
\quad
x \geqslant 0,
\label{eq:v-Volterra-single}
\end{equation}
where $\overline{G}(x) := 1 - G(x)$ denotes the complementary cdf of service requirements.
A solution of the integral equation (\ref{eq:v-Volterra-single}) is, as can be found in \cite{BBH1984,IT2015}, given in terms of the pdf
$g_{\mathrm{e}}(x) := \overline{G}(x)/\E[X]$ of the equilibrium
distribution (i.e., the residual lifetime distribution) corresponding to the service requirements:
\begin{align}
v(x) &= \pi_0 \sum_{n=1}^{\infty} u_n(x),
\quad
x \geqslant 0,
\label{eq:v(x)-general}
\\
\pi_0 &= \left(
\sum_{n=1}^{\infty} \int_0^{\infty} u_n(x) \diff x
\right)^{-1},
\nonumber
\\
u_1(x) &= \rho g_{\mathrm{e}}(x),
\quad
u_n(x) = \rho \int_0^x u_{n-1}(y) \,\P(Y\geqslant y) \,g_{\mathrm{e}}(x-y)
\diff y,
\;\;
n = 2,3,\ldots.
\nonumber
\end{align}
In some special cases, the expression (\ref{eq:v(x)-general})
for the virtual waiting time density $v(\cdot)$ can be further simplified, as shown in the examples below.

\begin{example}[Constant patience levels]
\em 
We consider the case that the patience levels take a constant value $\theta_0$ (for some $\theta_0 > 0$).
Let $g_{\mathrm{e}}^{(n)}(\cdot)$ (for $x \geqslant 0$, $n = 1,2,\ldots$) denote
the $n$-fold convolution of the pdf $g_{\mathrm{e}}(\cdot)$ of the
equilibrium distribution of the service times, and
$G_{\mathrm{e}}^{(n)}(x)$ the corresponding cdf. 
Also, define $g_{\mathrm{e},\theta_0}^{(n)}(\cdot)$ by
\[
g_{\mathrm{e},\theta_0}^{(n)}(x)
:=
\left\{
\begin{array}{lll}
\displaystyle
g_{\mathrm{e}}^{(n)}(x)/G_{\mathrm{e}}^{(n)}(\theta_0), & & x \leqslant \theta_0,
\\
0, & & x > \theta_0.
\end{array}
\right.
\]
We then have, by \cite{DT1985,IT2015}, with as usual `$\star$' denoting the convolution operator,
\begin{align}
\pi_0 
&= 
\left(
1 + \rho + 
\rho \sum_{n=1}^{\infty}\rho^n
G_{\mathrm{e}}^{(n)}(\theta_0)
\right)^{-1},
\label{eq:pi_0-constant-patience}
\\
v(x)
&=
\left\{
\begin{array}{l@{\qquad}l}
\pi_0 \displaystyle \sum_{n=1}^{\infty} \rho^n g_{\mathrm{e}}^{(n)}(x), 
& 
0 \leqslant x \leqslant \theta_0,
\\[2ex]
\pi_0 \rho g_{\mathrm{e}}(x)
+ 
\pi_0 \displaystyle\sum_{n=2}^{\infty}
\rho^n G_{\mathrm{e}}^{(n-1)}(\theta_0)
\cdot
\big[g_{\mathrm{e},\theta_0}^{(n-1)} \star g_{\mathrm{e}}\big](x),
& 
x > \theta_0,
\end{array}
\right.
\label{eq:v-constant-patience}
\end{align}
%
\end{example}

\begin{example}[Exponential service times]
\em 
If the service times follow an exponential distribution with mean
$1/\mu$, we have by \cite{IT2015,S1979} that
\begin{align*}
\pi_0 
&= 
\left( 
1 
+ \lambda \int_0^{\infty} 
\exp\left[-\mu x + \lambda \int_0^x \P(Y\geqslant y) \diff y \right]
\diff x
\right)^{-1},
\\
v(x) 
&=
\pi_0 \lambda \exp\left[
-\mu x + \lambda \int_0^x \P(Y\geqslant y) \diff y 
\right],
\quad
x \geqslant 0.
\end{align*}
Observe that if the patience distribution is also exponential with rate $\theta$, then the expressions can be simplified by substituting $ \int_0^x \P(Y\geqslant y) \diff y=\exp(-\theta x)$.
\end{example}

\section{Parametric estimation procedure}\label{sec:arrival_workload}
In the setting considered, we focus on the first non-balking $n+1$ customers. 
We set the time origin $t=0$ to the time instant that a non-balking
customer joins the system.
We record the effective arrival times $\tilde{A}_1,\tilde{A}_2,\ldots,\tilde{A}_n$ and departure times. From this information, we can fully reconstruct the
virtual waiting time process using (\ref{eq:VWT-def}); in particular, we obtain the sequence of virtual waiting times $V(0), V(\tilde{A}_1),\ldots, V(\tilde{A}_n)$ observed 
by the non-balking customers and the sizes of upward jumps $X_{0},X_{1},\ldots,X_{n}$ caused by them. In this section the goal is to estimate, in a parametric context, the arrival rate and patience-level distribution from the data.

The crucial observation is that we can construct a likelihood
function of the patience-level distribution using an observed sample of
effective interarrival times $A_1,A_2,\ldots,A_n$, 
waiting times $W_0 = V(0),W_1 = V(\tilde{A}_1), W_2 = V(\tilde{A}_2),
\ldots,W_n = V(\tilde{A}_n)$, and upward-jump sizes
$X_0,X_1,\ldots,X_n$, which can be reconstructed from the observed 
sequences mentioned in the previous paragraph. We work under the
natural assumption that the observed system is in stationarity, so
that $W_0$ is the stationary waiting time of a non-balking
customer and $X_0$ is a stationary upward-jump size.

\begin{remark}\label{R1}{\em 
To make the sequence of waiting times $W_1,W_2,\ldots,W_n$
stationary, we should let the time origin coincide with an
arrival instant of a non-balking customer. The validity of this statement can be argued as follows.

We first point out that the sequence of waiting times is in general not a Markov process; only for $s=1$ it is.
We obtain a Markov process when considering the vector of residual service times at effective arrival instants.  The waiting times and upward jump sizes are in fact measurable functions of this process. Therefore, in what follows we let the underlying Markov process of residuals be stationary.

First it is observed that assuming the virtual waiting time at time $t=0$ being stationary does \emph{not} imply that the waiting time
$W_1$ of the first customer follows the stationary distribution $W$.
To see this, consider (for simplicity) an ordinary M/G/$s$ queue (without customer
impatience, that is). If the virtual waiting time at time $t=0$ follows the stationary
virtual waiting time distribution $V$, then the first customer arriving after time
$t=0$ finds an idle server with probability
\[\P(W_1 = 0) = \P(V=0)+\int_0^{\infty}e^{-\lambda x}v(x)\diff x > \P(V=0).\]
Since the stationary waiting time $W$ has the same distribution as
the stationary virtual waiting time $V$ in the ordinary M/G/$s$ queue (due to {\sc pasta}), we find 
the inequality \[\P(W_1 = 0) > \P(V=0) = \P(W=0).\] This shows that the
waiting time $W_1$ of the first-arriving customer is biased, in that it is not distributed as $W$.

Intuitively, the bias is a consequence of the fact that a customer arriving after 
a long interarrival time is more likely chosen as the first arriving
customer in our experiment (cf.\ the well-known inspection paradox in renewal theory), and such a
customer tends to observe the system less congested than time-average.
It is easily seen that letting time $t=0$ correspond to the arrival instant of a non-balking customer
makes the sequence of waiting times $W_1,W_2,\ldots,W_n$
stationary.   
$\hfill\diamond$}
\end{remark}

Let $\overline{H}(x) := 1-H(x)$ denote the complementary cdf of patience levels.
Also, we denote by $\tilde{H}(x)$ the probability of a customer joining when the virtual waiting time equals $x\geqslant 0$:
\begin{align*}
\tilde{H}(x) := \P(Y \geqslant x)
= \overline{H}(x)+\P(Y= x).
\end{align*}
Note that $\tilde{H}(x)=\overline{H}(x)=\P(Y>x)$ for continuous patience distributions.

We next describe the construction of the likelihood function.
We first recall that the effective interarrival times $A_1,A_2,\ldots,A_n$
are {\it not} exponentially distributed, due to the fact that between two effective arrival
instants there may have been arriving customers who observed a virtual waiting time level
exceeding their patience level (as depicted in Figure~\ref{fig:arrival_dist}). Despite this, we can still 
characterize the 
effective interarrival time
distribution in terms of the observed quantities.  To see this, suppose that, for some $t \geqslant 0$, there have been no
effective arrivals in the interval $(\tilde{A}_{i-1}, \tilde{A}_{i-1}+t]$.
An arrival in $(\tilde{A}_{i-1}+t, \tilde{A}_{i-1}+t+\varDelta t]$
then occurs with probability $\lambda \,\varDelta t+ o(\varDelta t)$, as $\varDelta t\downarrow 0$.
The corresponding customer joins the system (i.e., becomes effective) with probability
$\tilde{H}(V(\tilde{A}_{i-1}) - t)$ because (i)~the virtual waiting time seen on
the arrival equals $\max\{0,V(\tilde{A}_{i-1}) - t\}$ and (ii)~$\tilde{H}(x) =
1$ for $x < 0$. Therefore, the occurrence of the next
(i.e., the $i$-th) effective arrival follows a time-inhomogeneous
Poisson process with time-dependent intensity 
$\lambda\, \tilde{H}(V(\tilde{A}_{i-1}) - t)$.
Noting that $V(\tilde{A}_{i-1}) = W_{i-1}+X_{i-1}$ (for $i =
1,2,\ldots$), we thus conclude that
\begin{equation}
\P(A_i > t \,|\, W_{i-1}+X_{i-1}=v)
=
\exp\left({-\lambda \int_0^t \tilde{H}(v-u)\diff u}\right).
\label{eq:A_i-CCDF}
\end{equation}
This representation of the distribution of the effective interarrival times facilitates the evaluation of the likelihood. 
The (conditional) likelihood function for
${\boldsymbol A}:=(A_1,A_2,\ldots,A_n)$ given
${\boldsymbol W}:=(W_1,W_2,\ldots,W_n)$ and
${\boldsymbol X}:=(X_1,X_2,\ldots,X_n)$ is given
by the product of the conditional densities 
\begin{align}\label{eq:density_A}
f_{A_i \,|\, v}(t)
&=
-\frac{\diff}{\diff t} \P(A_i > t \,|\, W_{i-1}+X_{i-1}=v) 
\\
&=
\lambda\, \tilde{H}(v-t)
\exp\left({-\lambda \int_0^t \tilde{H}(v-u)\diff u}\right).
\end{align}
This yields the conditional likelihood function
\begin{align*}
L_n(H;{\boldsymbol A},{\boldsymbol W} \,|\, {\boldsymbol X})
&=
\prod_{i=1}^n 
\lambda\,
\tilde{H}(W_{i-1}+X_{i-1}-A_i)
\exp\left({-\lambda \int_0^{A_i} \tilde{H}(W_{i-1}+X_{i-1}-u)\diff u}\right)
\\
&=
\lambda^n
\prod_{i=1}^n 
\tilde{H}(W_i)
\exp\left({-\lambda \int_0^{A_i} \tilde{H}(W_{i-1}+X_{i-1}-u)\diff u}\right),
\end{align*}
where we used a Lindley-type recursion, generalized to the multi-server queue:
\begin{equation}\label{eq:Vk_recursion}
W_{k}=\max\{W_{k-1}+X_{k-1}- A_k,0\},
\quad
k = 1,2,\ldots.
\end{equation}

In the rest of this section, we assume that the patience-level distribution
is characterized by a finite-dimensional vector of parameters $\theta \in \Theta\subseteq\mathbb{R}^p$. For $\theta \in \Theta$, let $H_{\theta}(\cdot)$, $\overline{H}_{\theta}(\cdot)$, and $\tilde{H}_{\theta}(\cdot)$ denote $H(\cdot)$, $\overline{H}(\cdot)$, and $\tilde{H}(\cdot)$ given parameter $\theta$, respectively.
Let any $H_\theta(\cdot)$, with $\theta\in\Theta$, be identifiable in the conventional Kullback-Leibler sense. The maximum likelihood estimator (MLE) of $(\lambda,\theta)$ is
then given by
\[
(\hat\lambda_n,\hat{\theta}_n)\in\argmax_{(\lambda,\theta)\in\,{\mathbb R}_+\times \Theta}
L_n(\lambda,\theta;{\boldsymbol A},{\boldsymbol W} \,|\, {\boldsymbol X}),
\]
where
\begin{equation}
L_n(\lambda,\theta;{\boldsymbol A},{\boldsymbol W} \,|\, {\boldsymbol X})
:=
\lambda^n \prod_{i=1}^n 
\tilde{H}_\theta(W_i)
\exp\left({-\lambda \int_0^{A_i} \tilde{H}_\theta(W_{i-1}+X_{i-1}-u)\diff u}\right).
\label{eq:MLE-workload-general}
\end{equation}
In the sequel we provide an asymptotic analysis describing the performance of this MLE. 
We first focus, in Sections \ref{sec:mle_continuous}--\ref{sec:mle_constant}, on the parameters pertaining to the patience only; i.e., the estimation does not cover the arrival rate $\lambda$.
The object of study is, in self-evident notation, for $\lambda>0$ given,
\[
\hat{\theta}_n\in\argmax_{\theta\in\Theta}
L_n(\theta;{\boldsymbol A},{\boldsymbol W} \,|\, {\boldsymbol X}).
\]
Extending this to a procedure to also include estimation of the arrival rate is relatively straightforward; we get back to it in Section \ref{sec:lambda_estimation}.

\begin{remark}{\em 
Observe that $L$ is a conditional likelihood and not a full likelihood. This is due to the fact that the upward jump-sizes $X_i$ have an elaborate distribution, both marginally and jointly with the waiting times. The exception is the single-server case where $X_i\sim B$ and is independent of the previous waiting and arrival times. Hence, for $s=1$, the maximization of
$L_n(H;{\boldsymbol A},{\boldsymbol W} \,|\, {\boldsymbol X})$ with respect to them
is equivalent to that of the unconditional
likelihood function for $({\boldsymbol A},{\boldsymbol W},{\boldsymbol X})$. It is noted that the  procedure to estimate $\theta$ does not require the simultaneous estimation of the service-requirement or upward-jump size distributions. 
$\hfill\diamond$}
\end{remark}

In more detail, the remainder of this section is organized as follows. 
First we establish (in Section \ref{sec:mle_continuous}) conditions for strong consistency
and asymptotic normality of our estimator of the patience parameters, with the errors scaled by $\sqrt{n}$, focusing on
continuous patience-level distributions. We then consider (in Section \ref{sec:mle_constant}) the case of deterministic patience and establish consistency (independently) for
this case as well. It is noted that in the latter case the asymptotic errors
are {\it not} normally but rather exponentially distributed, with the errors scaled by $n$.
This is due to the fact that for deterministic patience the MLE is obtained on the boundary of
the sample data, much like in the well-known case of estimating the parameter $\theta$ of
a uniform distribution on $[0,\theta]$. As mentioned, Section \ref{sec:lambda_estimation} discusses
the estimation of the arrival rate.

\subsection{Continuous patience-level distribution}\label{sec:mle_continuous}

This subsection covers the asymptotic performance of the MLE for the case that
the patience-level distribution is continuous and parametric. 
Let $\theta_0$ denote the true parameter. Throughout the following
analysis the underlying probability measure is $\P_{\theta_0}$, i.e., the probability measure
corresponding to the true patience-level distribution. For the asymptotic
results we make the following assumptions, which will be discussed in Section \ref{sec:cont_mle_assumptions}.

\begin{assumption2}
The following assumptions are imposed:
\begin{enumerate} 
\item[{\rm (A1)}] The parameter space $\Theta\subset \mathbb{R}^p$ is a compact set such that the true parameter lies in the interior of the set, i.e., $\theta_0\in\Theta^\mathrm{o}$.
\item[{\rm (A2)}] The observation period commences at $t=0$ which corresponds to the stationary arrival instant of a non-balking customer, and thus the sequence of waiting times $W_1,W_2,\ldots,W_n$ is stationary.
\item[{\rm (A3)}] Let 
$\overline{H}_{\inf}(x) := \inf_{\theta \in \Theta} \overline{H}_{\theta}(x)$.
Depending on whether the right endpoint \[h_{\sup} := \inf\{x \geqslant 0: \overline{H}_{\inf}(x) = 0\}\]
takes a finite value or not, we assume one of the following properties:
{\rm (i)}~If $h_{\sup} = \infty$, then there exists a positive non-decreasing
function $f(\cdot)$ such that $\lim_{x \to \infty}f(x) \in (0,\infty]$ and for some constants $c_1, c_2 \in [0, \infty)$,
\begin{equation*}
\lim_{x \to \infty} \frac{e^{-f(x)}}{\overline{H}_{\inf}(x)} = c_1,
\qquad
\lim_{x \to \infty} f(x) \overline{H}_{\theta_0}(x) = c_2.
\end{equation*}
{\rm (ii)}~If $h_{\sup} < \infty$, then there exists a positive
non-decreasing function $f(\cdot)$ such that $\lim_{x \to h_{\sup}-}f(x) \in (0,\infty]$
and for some constants $c_1, c_2 \in [0, \infty)$,
\begin{equation*}
\lim_{x \to h_{\sup}-} 
\frac{e^{-f(x)}}{\overline{H}_{\inf}(x)} = c_1,
\qquad
\lim_{x \to h_{\sup}-} f(x) \overline{H}_{\theta_0}(x) = c_2.
\end{equation*}
\item[{\rm (A4)}] The gradient vector and Hessian matrix of ${H}_\theta$ are continuous with respect to $\theta$. With $\Psi_1(\theta)$ denoting the Hessian matrix {of} the log-likelihood corresponding to a single waiting time $($in stationarity$)$, $\E\Psi_1(\theta)$ has finite elements in all coordinates for any $\theta\in\Theta$.
\item[{\rm (A5)}] The collection of functions $\{H_\theta(x):\ \theta\in\Theta\}$ is equicontinuous: for any $x\geqslant 0$ and $\epsilon>0$ there exists a $\delta>0$ such that $|H_\theta(x)-H_\theta(y)|<\epsilon$ for any $y$ such that $|x-y|<\delta$, uniformly over all $\theta\in\Theta$.
\end{enumerate}
\end{assumption2}

We now state the main results of this subsection. In the sequel, $\mathrm{N}({\mu},\Sigma)$ 
denotes a normally distributed random variable with mean vector $\mu$ and covariance matrix $\Sigma.$

\begin{theorem}\label{thm:cont_mle_consistency}
If Assumptions {\rm (A1)--(A3)} and {\rm (A5)} are satisfied, then, as $n\to\infty$,
\begin{equation}\label{eq:V_MLE_consistent}
\hat{\theta}_n\asarrow \theta_0 .
\end{equation}
\end{theorem}

\begin{theorem}\label{thm:cont_mle_normal}
If Assumptions {\rm (A1)--(A5)} hold, then, as $n\to\infty$, \begin{equation}\label{eq:cont_mle_normal}
\sqrt{n}(\hat{\theta}_n-\theta_0)\darrow \mathrm{N}\left(0,I^{-1}(\theta_0)\right),
\end{equation}
where  {$I(\theta_0):=-\E \Psi_1(\theta_0)$} and $I^{-1}(\theta_0)$ denotes the inverse of $I(\theta_0)$.
\end{theorem}

The proofs of the above theorems are provided in Section
\ref{sec:cont_mle_proofs}, after the detailed discussion of (A1)--(A5) that we give in Section \ref{sec:cont_mle_assumptions}.

\subsubsection{Discussion of assumptions}\label{sec:cont_mle_assumptions}
We next provide more background on the assumptions imposed. In addition we discuss their possible relaxation.
\begin{enumerate}
\item Assumption (A1) is natural, as it requires that the parameter space is big enough so that it contains the true parameter. In practice the parameter space can be adjusted on-the-fly, for example if for large $n$ we obtain boundary solutions for the first order conditions.
\item Assumption (A2) facilitates the use of known results on the convergence of stationary dependent sequences. Note that, as long as a stationary virtual waiting time distribution exists (for which we have given the condition as Eq.\ \eqref{eq:stability} in Section \ref{sec:model}), the same results should hold without making this assumption because of the regenerative nature of the process. However, without this assumption the conditions for both consistency and asymptotic normality are harder to verify. In particular, the crucial step for consistency is the uniform convergence of the log-likelihood, established in Lemma \ref{lemma:unifrom_ll} below. More elaborate conditions for uniform convergence are detailed in for example \cite{A1987,PP1989,R1962}.
\item If there are patient customers who do not balk regardless of the waiting time, that is, $\lim_{x \to \infty}\overline{H}_{\theta_0}(x) > 0$, then Assumption (A3) requires that the parameter space $\Theta$ is chosen so that the existence of patient customers ($\lim_{x \to \infty}\overline{H}_{\theta}(x) > 0$) is assumed for all $\theta \in \Theta$, which is a reasonable assumption in modeling a service system where both patient and impatient customers exist. In the case of $\lim_{x \to \infty}\overline{H}_{\theta_0}(x) = 0$, on the other hand, Assumption (A3) requires that the decay rate of
$\overline{H}_{\theta}(x)$ does not vary too strongly among the
distributions $H_{\theta}$, with $\theta$ in the parameter space $\Theta$. In
practice, this assumption is seldom violated. For example, 
if the true patience-level distribution decays exponentially, i.e., 
$\lim_{x \to \infty} e^{\nu x}\overline{H}_{\theta_0}(x) = c$ for some
$\nu > 0$ and $c > 0$, then (A3) is satisfied if the
infimum tail function $\overline{H}_{\inf}(x)$ does not decay faster than
the doubly exponential function $e^{-e^{\nu x}}$ (which evidently decays 
exceptionally fast).
As another example, supposing that the true patience-level distribution 
obeys a power law, i.e., $\lim_{x \to \infty} x^k
\overline{H}_{\theta_0}(x) = c$ for some $k > 0$ and $c > 0$,
then (A3) is satisfied if $\overline{H}_{\inf}(x)$ does not decay
faster than $e^{-x^k}$.

\item Assumption (A4) enables the construction of a standard martingale CLT for the asymptotic distribution of the estimation error of the MLE. It is not a necessary condition, but in cases where the assumption is not satisfied one is typically required to apply ad-hoc analysis to derive an asymptotic distribution. In Section \ref{sec:mle_constant} we show that for a deterministic patience level, that does not satisfy (A4), the asymptotic distribution of the error is exponential and not normal. Assumption (A2) implies that the expectation of the Hessian matrix of the log-likelihood $\Psi_1(\theta)$ is the covariance matrix of the gradient, and thus it is always positive definite. As a consequence, one just needs to verify that the coordinates are finite for (A4) to hold.
\item The equicontinuity assumption (A5) enables concise analysis and can be verified for many continuous patience-level  distributions. For example, if there exists a bounded density for every $H_\theta(\cdot)$, then it is Lipschitz continuous, and a uniform bound is given by the supremum of the constants in the compact set $\Theta$. The assumption does not hold if the support of the distribution depends on $\theta$. If $H_\theta(x)$ has some discontinuities (with respect to $x$), one could pursue replacing (A5) by an appropriate upper semi-continuity assumption; see \cite[Thm.  16b]{book_F1996} and \cite{PP1989}. We will present such an example in Section~\ref{sec:mle_types}. 
Our asymptotic results may hold by replacing (A5) by (A4) and verifying additional measurability conditions on $\inf_{\eta \in \mathcal{B}_\theta} \overline{H}_{\eta}(x)$ for a neighborhood $\mathscr{B}_\theta$ of any $\theta\in\Theta$ (see \cite[Corr. 2]{A1987}). 
\end{enumerate} 

\subsubsection{Proofs}\label{sec:cont_mle_proofs}

From (\ref{eq:MLE-workload-general}), first note that we can express the log-likelihood $\ell_n(\theta;{\boldsymbol A},{\boldsymbol W}\,|\,
{\boldsymbol X}):=\log L_n(\theta;{\boldsymbol A},{\boldsymbol W}\,|\, {\boldsymbol X})$ by
\begin{equation}\label{eq:V_log_likelihood}
n\log\lambda+\sum_{i=1}^n \left(\log\overline{H}_\theta(W_i)-\lambda \int_0^{A_i} \overline{H}_\theta(W_{i-1}+X_{i-1}-u)\diff u\right){\boldsymbol 1}_{\{\overline{H}_\theta(W_i)>0\}}.
\end{equation}
Observe that we replaced $\tilde{H}_\theta(\cdot)$ by $\overline{H}_\theta(\cdot)$ because $H_\theta(\cdot)$ is continuous.

With $W:=\lim_{i\to\infty}W_i$ existing, as $n\to\infty$, the continuous mapping theorem yields
\begin{equation}\label{eq:ell_theta_as}
\frac{1}{n}
\ell_n(\theta;{\boldsymbol A},{\boldsymbol W} \mid {\boldsymbol X})
\asarrow 
\E[\ell_1(\theta;A_1,W_0,W_1 \mid X_0)],
\end{equation}
where $A_1$ is the effective interarrival time when the initial virtual waiting time is $W_0+X_0$ and {$W_0=_{\rm d} W$, i.e., the stationary waiting time.}
From now on we use the compact notations 
$\ell_n(\theta):=\ell_n(\theta;{\boldsymbol A},{\boldsymbol W} \,|\, {\boldsymbol X})$ and $\ell(\theta):=\E[\ell_1(\theta;A_1,W_1,W_0 \,|\, X_0)]$. The density of the effective interarrival times \eqref{eq:density_A} is uniquely determined by the function $H_\theta(\cdot)$, and thus so is the likelihood \eqref{eq:MLE-workload-general}. This is because the function \[H_{\theta}(\cdot):=\left\{H_{\theta}(x):=\int_{[0,x]}\diff \nu_\theta(u): x\geqslant 0\right\}\] is uniquely determined by $\theta$, as a consequence of the fact that for every $\theta\in\Theta$
the measure $\nu_\theta(\cdot)$ corresponds to a different distribution (in the almost-everywhere sense, that is). As a consequence, $\{H_\theta(\cdot):\ \theta\in\Theta\}$ is a parametric collection of distribution functions such that there is no pair $\theta_1,\theta_2\in\Theta$ for which $H_{\theta_1}(\cdot)=H_{\theta_2}(\cdot)$ almost everywhere. This entails that the model is identifiable in the Kullback-Leibler sense (see e.g.\ \cite[Ch. 17]{book_F1996}), and hence
\begin{equation}\label{eq:cont_MLE_unique}
\ell(\theta)-\ell(\theta_0)<0, \ \forall \theta\neq\theta_0 .
\end{equation}

The key step in the proof of Theorem \ref{thm:cont_mle_consistency} is establishing a uniform version of \eqref{eq:ell_theta_as}; then strong consistency follows by the methodology of Wald \cite[Chapters 16--17]{book_F1996}. The virtual waiting time observations are not independent but by (A2) they are stationary, so that we can apply a uniform law for stationary sequences that is commonly used in the econometrics literature. We specifically rely on a theorem taken from \cite{R1962} and its extensions, in particular the ones developed in \cite{A1987,PP1989}. The uniform convergence is established in Lemma~\ref{lemma:unifrom_ll} (proven in the appendix).

\begin{lemma}\label{lemma:unifrom_ll}
If Assumptions {\rm (A1)--(A3)} and {\rm (A5)} hold, then, as $n\to\infty$,
\begin{equation}\label{eq:cont_uniform_ell_theta}
\sup_{\theta\in\Theta}\left|\frac{1}{n}\ell_n(\theta)-\ell(\theta)\right|\asarrow 0.
\end{equation}
\end{lemma}

\begin{proof}[Proof of Theorem \ref{thm:cont_mle_consistency}]
Observe that the strong consistency follows from the identifiability property \eqref{eq:cont_MLE_unique} and from Lemma~\ref{lemma:unifrom_ll}. Note in particular that the proof of \cite[Thm.\ 17]{book_F1996} does not rely on i.i.d.\ observations once uniform convergence has been established, and so the same steps can be applied here. In particular, a sufficient condition (relying on uniform convergence for strong consistency when observations are dependent) can also be found in \cite{HM1986}: for every $\theta\neq\theta_0$ there exists a neighborhood $\mathscr{B}_\delta(\theta)$ such that
\begin{equation}\label{eq:cont_mle_condition2}
\lim_{n\to\infty}\sup_{\gamma\in\mathscr{B}_\delta(\theta)}\left(\frac{1}{n}\ell_n(\gamma)-\frac{1}{n}\ell_n(\theta_0)\right)<0,
\end{equation}
almost surely. This follows directly from \eqref{eq:cont_MLE_unique} and \eqref{eq:cont_uniform_ell_theta}.
\end{proof}

We continue by proving asymptotic normality that was stated in Theorem \ref{thm:cont_mle_normal}. This amounts to showing that the estimation error $\sqrt{n}(\hat{\theta}_n-\theta_0)$ is asymptotically normal as $n\to\infty$, assuming that (A1)--(A5) are satisfied. To this end we apply the well-known delta method and an appropriate version of the martingale CLT (see e.g.~\cite[Thm.~12.6]{book_H1997}). 

\begin{proof}[Proof of Theorem \ref{thm:cont_mle_normal}]
For any given $v$, we let $\nabla \overline{H}_\theta(v)\in\mathbb{R}^p$ and $\nabla^2 \overline{H}_{\theta}(v)\in\mathbb{R}^{p\times p}$ denote the gradient and Hessian, respectively, of $\overline{H}_\theta$. These are both continuous with respect to any coordinate $\theta_k$ of $\eta$, as a consequence of (A4). Thus, both the gradient and the Hessian of the log-likelihood are continuous functions. These are given by $\nabla\ell_n(\theta):=\dot{\ell}_n(\theta)$ and $\nabla^2\ell_n(\theta):=\sum_{i=1}^n\Psi_i(\theta)$, respectively, where, for $k=1,\ldots,p$,
\begin{equation}\label{eq:cont_mle_dLL}
(\dot{\ell}_n(\theta))_k=\sum_{i=1}^n\left(\frac{\nabla_k\overline{H}_\theta(W_i)}{\overline{H}_\theta(W_i)}-\lambda\int_0^{A_i}\nabla_k\overline{H}_\theta(W_{i-1}+X_{i-1}-u)\diff u\right){\boldsymbol 1}_{\{\overline{H}_\theta(W_i)>0\}},
\end{equation} 
and, for $k,l=1,\ldots,p$,
\begin{equation}\label{eq:cont_mle_dLL2}
\begin{split}
\left(\Psi_i(\theta)\right)_{kl} &= \Bigg(\frac{\overline{H}_\theta(W_i)\nabla_{kl}^2\overline{H}_\theta(W_i)-\nabla_k\overline{H}_\theta(W_i)\nabla_l\overline{H}_\theta(W_i)}{\overline{H}_\theta(W_i)^2} \\ 
& \quad \quad   -\lambda\int_0^{A_i}\nabla_{kl}^2\overline{H}_\theta(W_{i-1}+X_{i-1}-u)\diff u\Bigg){\boldsymbol 1}_{\{\overline{H}_\theta(W_i)>0\}}
\end{split}.
\end{equation} 

Following the lines of the standard delta method \cite[Ch.\ 3]{book_vdV1998}, we consider the expansion of $\dot{\ell}_n(\cdot)$ at the MLE $\hat{\theta}_n$ around the true parameter:
\begin{equation}\label{eq:cont_mle_taylor}
\frac{1}{n}\dot{\ell}_n(\hat{\theta}_n)=\frac{1}{n}\dot{\ell}_n(\theta_0)+(\hat{\theta}_n-\theta_0)\int_{0}^1\frac{1}{n}\sum_{i=1}^n\Psi_i(\alpha\hat{\theta}_n+(1-\alpha)\theta_0)\diff \alpha.
\end{equation}
If $\theta_0\in\Theta$, then due to the strong consistency that we found in Theorem \ref{thm:cont_mle_consistency} we have, as $n\to\infty$, that 
\[
\hat{\theta}_n\asarrow \theta_0.
\] 
Furthermore, as $\theta_0\in\Theta$ by assumption (A1), the smoothness assumption (A4) implies that, as $n$ grows large, $\ell_n(\theta)$ converges to a concave function and the MLE is given by a sequence of roots $\hat{\theta}_n$ satisfying $\dot{\ell}_n(\hat{\theta}_n)=0$, i.e., from some $N$ on there is no boundary solution for any $n\geqslant N$ with probability one. Therefore \eqref{eq:cont_mle_taylor} yields, with
\[
B_n:=-\int_{0}^1\frac{1}{n}\sum_{i=1}^n\Psi_i(\alpha\hat{\theta}_n+(1-\alpha)\theta_0)\diff \alpha,
\]
that
\begin{equation}\label{eq:cont_mle_delta}
\lim_{n\to\infty}\sqrt{n}(\hat{\theta}_n-\theta_0) B_n \overset{\mathrm{d}}{=}\lim_{n\to\infty}\frac{1}{\sqrt{n}}\dot{\ell}_n(\theta_0),
\end{equation}
i.e., both sides of \eqref{eq:cont_mle_delta} converge to the same distribution (if the limits exist).
In Lemma \ref{lemma:cont_mle_clt} we apply the stationarity of $W_i$ in (A2) to show that $B_n\asarrow I(\theta_0)$, where $I(\theta_0)=-\E \Psi_1(\theta_0)$. Furthermore, $\dot{\ell}_n(\theta_0)/\sqrt{n}$ is shown to satisfy a martingale CLT with asymptotic variance $I(\theta_0)$. As $\E \nabla\ell_1(\theta_0)=0$, Assumption (A2) implies that the asymptotic variance equals the stationary covariance of the gradient 
\begin{equation}\label{eq:psi_covariance}
I(\theta_0)=\E\nabla^2(\ell_1(\theta_0))=\E[\nabla(\ell_1(\theta_0))\nabla(\ell_1(\theta_0))^\top],
\end{equation}
hence $I(\theta_0)$ is a positive definite matrix (see \cite[Ch. 18]{book_F1996} for more details). Assumption (A4) further demands that the elements of $I(\theta_0)$ are all finite, and then combining the above and applying Slutsky's theorem to \eqref{eq:cont_mle_delta} yields Theorem \ref{thm:cont_mle_normal}. \end{proof}

{We are thus left with showing Lemma~\ref{lemma:cont_mle_clt} below; its proof is given in the appendix.}

\begin{lemma}\label{lemma:cont_mle_clt}
If {\rm (A1)--(A5)} hold and $I(\theta_0)=-\E \Psi_1(\theta_0)$ is a positive definite matrix with finite elements, then {\em (a)} as $n\to\infty$, $B_n\asarrow I(\theta_0)$, and {\em (b)} as $n\to\infty$,
\begin{equation}\label{eq:cont_dl_clt}
\frac{1}{\sqrt{n}}\dot{\ell}_n(\theta_0)\darrow \mathrm{N}(0,I(\theta_0)).
\end{equation}
\end{lemma}

\subsection{Constant patience-level MLE}\label{sec:mle_constant}

We next consider the case where all customers have
the same patience level $\theta_0$, i.e., $H(y)={\boldsymbol 1}_{\{y\geqslant
\theta_0\}}$; observe that in this case the stability condition \eqref{eq:stability} is always satisfied, but the continuity assumptions of Section~\ref{sec:mle_continuous} do not apply. 
As we will see, in this case the properties of the MLE  are markedly different from those identified in Section \ref{sec:mle_continuous}.

In this setting, the likelihood function 
(\ref{eq:MLE-workload-general}) reduces to, with the event ${\mathscr E}_i$ denoting $\{W_{i-1}+X_{i-1} \leqslant \theta \}$ and ${\mathscr E}_i^{\rm c}$ its complement,
\begin{align*}
L(\theta;{\boldsymbol A},{\boldsymbol W} \,|\, {\boldsymbol X})
&=
\lambda^n
\prod_{i=1}^n 
{\boldsymbol 1}_{\{W_i \leqslant \theta\}}
\biggl(
{\boldsymbol 1}_{\{{\mathscr E}_i\}}\,
\mathrm{e}^{-\lambda \int_0^{A_i} \diff u}
+
{\boldsymbol 1}_{\{{\mathscr E}_i^{\rm c}\}}\,
\mathrm{e}^{-\lambda \int_{W_{i-1}+X_{i-1}-\theta}^{A_i} 
\diff u}
\Biggr)
\\
&=
\lambda^n
\prod_{i=1}^n 
{\boldsymbol 1}_{\{W_i \leqslant \theta\}}
\biggl(
{\boldsymbol 1}_{\{{\mathscr E}_i\}}\,
\mathrm{e}^{-\lambda A_i}+{\boldsymbol 1}_{\{{\mathscr E}_i^{\rm c}\}}\,\mathrm{e}^{-\lambda (A_i+\theta-W_{i-1}-X_{i-1})}
\Biggr).
\end{align*}
{Observe that $L(\theta;{\boldsymbol A},{\boldsymbol W} \,|\, {\boldsymbol X})$ (i) equals
zero for $\theta\in[0,\max_{i=1,2,\ldots,n}W_i)$, (ii) has an upward discontinuity at $\theta =\max_{i=1,2,\ldots,n}W_i$ and decreases for \[\theta\in\Big[\max_{i=1,2,\ldots,n}W_i,\max_{i=1,2,\ldots,n}\{W_i+X_i\}\Big),\] and (iii) takes a constant value for $\theta\geqslant \max_{i=1,2,\ldots,n}\{W_i+X_i\}$.}
We thus find the intuitively appealing property that the MLE is given by
\begin{equation}
\hat{\theta}_n
=
\max_{i=1,2,\ldots,n}W_i\ , 
\end{equation} 
i.e., the maximum virtual waiting time at jump times. Note that this estimator resembles the MLE of the parameter $\theta$ when the observations are uniformly distributed on $[0,\theta]$. 

While the estimator $\hat{\theta}_n$ is clearly biased (which follows from $\E_{\theta_0}[\hat{\theta}_n] < \theta_0$ for all $n$), we show that it converges almost surely to $\theta_0$ as $n\to\infty$. In addition, we prove that the estimation error scaled at rate $n$ converges to an exponential random variable.
\begin{theorem}\label{thm:mle_constant}
As $n\to\infty$,
\begin{equation}
\hat{\theta}_n \asarrow \theta_0 \,,
\label{eq:maxV-limit}
\end{equation}
and
\begin{equation}
\P\left(n(\theta_0 - \hat{\theta}_n) \geqslant x\right)\to \mathrm{e}^{-v(\theta_0)x/(1-P_{\ell})},
\label{eq:hat-theta-scaled-limit}
\end{equation}
where $P_{\ell}$ and $v(\theta_0)$ denote the stationary loss probability and the virtual waiting time density at level $\theta_0$.
In addition, the asymptotic variance of the estimation error agrees with the variance of the limiting exponential distribution, and is given by
\begin{align}
\lim_{n\to\infty}n^2
\Var\left[\theta_0 - \hat{\theta}_n\right]
&=
\lim_{n \to \infty}
\left\{
\E\left[\{n(\theta_0 - \hat{\theta}_n)\}^2\right]
-
\left(
\E\left[n(\theta_0 - \hat{\theta}_n)\right]
\right)^2
\label{eq:var_MLE-def}
\right\}\\
\label{eq:var_MLE}&
=
\left(
\frac{1-P_{\ell}}{v(\theta_0)}
\right)^2.
\end{align}
\end{theorem}

{The proof of Theorem~\ref{thm:mle_constant}, which is given in the appendix, constructs lower and upper bounds for the MLE and
establishes that they both converge to $\theta_0$. The same bounds are
also used to characterize the asymptotic distribution of the estimation
error. }
%
%

\subsection{Estimating the arrival rate}\label{sec:lambda_estimation}

Where in the preceding subsections we focused on estimating the patience-level distribution (in a parametric context) for a given value of the arrival rate $\lambda$, we now discuss the estimation of $\lambda$. Let $\lambda_0$ denote the true value of $\lambda$. There are several ways to estimate it, the most basic one relying on the idle period observations.  Denote the duration of the $k$-th idle period by $I_k$, where in the multi-server setting ($s>1$) an idle period refers to time intervals such that at least one server is idle. Denote by $E_k$ the total number of effective arrivals during idle period $k=1,2,\ldots$ (meaning, for instance, $E_k=1$ for all $k= 1,2,\ldots$ when $s=1$). Let $C_n$ denote the number of idle periods observed up until (and including) the $n$-th effective arrival. Then we propose the estimator
\begin{equation}\label{eq:lambda_MLE}
\hat{\lambda}_n=\frac{\sum_{k=1}^{C_n} E_k}{\sum_{k=1}^{C_n} I_k}\ .
\end{equation}
The rationale behind this estimator is that there is no balking during idle periods as the virtual waiting time is zero. Hence, the arrival process during these idle times is homogeneous Poisson with rate $\lambda$.
Therefore, the estimator $\hat{\lambda}_n$ is a standard MLE of the rate parameter of an exponentially distributed random variable, and satisfies all desired asymptotic properties.

Of course, the above procedure does not exploit a substantial amount of potentially useful data that is collected during busy periods. For the case of a deterministic patience level $\theta_0$, the above estimator is easily improved upon. Recall that $\hat{\theta}_n$ was defined as $\max\{W_1,W_2,\ldots,W_n\}$, with the immediate consequence that $\hat\theta_n\leqslant \theta_0$.
Supposing we observe an arrival such that immediately after this arrival the virtual waiting time level is still below the current value of the estimator, then the next arrival is an effective arrival, and hence occurs after an exponentially distributed random variable with rate $\lambda$. In this way we generate more observations, thus allowing to estimate $\lambda$ with a better precision. Observe that the new observations and the idle times form an i.i.d.\ sequence. 

For the case of a continuous patience-level distribution one may use \eqref{eq:lambda_MLE}, or alternatively set up a joint MLE for $\lambda$ and $\theta$. The latter has clear advantages, in particular for small samples or heavily loaded systems. As the log-likelihood function \eqref{eq:V_log_likelihood} is a smooth and concave function with respect to $\lambda$, the MLE of the arrival rate for any estimator $\breve{\theta}_n$ of $\theta$ is
\begin{equation}\label{eq:MLE_lambda}
\hat{\lambda}_n(\breve{\theta}_n)=\breve{\theta}_n\left[\frac{1}{n}\sum_{i=1}^n \int_0^{A_i} \overline{H}_{\breve{\theta}_n}(W_{i-1}+X_{i-1}-u)\diff u\,{\boldsymbol 1}_{\{\overline{H}_{\breve{\theta}_n}(W_i)>0\}}\right]^{-1}.
\end{equation}
The asymptotic results of Section \ref{sec:mle_continuous} can therefore be extended in a straightforward manner, so as to cover the joint estimation of $(\lambda,\theta)$. We demonstrate this in Section \ref{sec:mle_exponential}, 
where we detail a procedure for jointly estimating the arrival rate and the patience parameter for the case of an exponentially distributed patience level.

\section{Exponential and generalized hyperexponential patience}\label{sec:special}

In this section we discuss a robust and practical approach for estimating continuous patience distributions. In our approach, this is achieved by fitting the MLE of a generalized hyperexponential (GHE) distribution. This approach is attractive because the class of GHE distributions is known to be dense in the space of non-negative continuous distributions (see for instance \cite{BH1986}), which, in practical terms, means that any non-negative continuous distribution can be approximated arbitrarily closely by a GHE distribution.  
In our simulation-based experiments  highly accurate estimates are obtained, even when the baseline patience distribution itself is {\it not} GHE. 

We first provide a detailed analysis of the joint MLE for the arrival rate and the single parameter of an exponential patience distribution, to then move to the cases of hyperexponential and generalized hyperexponential patience distributions.  We also present a heuristic search method that fits an approximate GHE distribution for any continuous patience distribution.

\subsection{Exponentially distributed patience}\label{sec:mle_exponential}

Suppose that the arrival rate $\lambda_0$ is unknown and that the patience-level distribution is exponential, i.e., $H_{\theta_0}(x)=1-\mathrm{e}^{-\theta_0 x}$ for an unknown $\theta_0$ (and $x\geqslant 0$). Due to \eqref{eq:V_log_likelihood} it can be verified, by distinguishing between the cases $W_0+X_0\geqslant A_1$ and $W_0+X_0< A_1$, that the log-likelihood for a single observation is
\[
\begin{split}
\ell_1(\lambda,\theta) &= \log\lambda-\theta W_1-\lambda\int_{0}^{A_1} e^{-\theta(W_0+X_0-u)^+}\diff u \\
&=  \log\lambda-\theta W_1-\frac{\lambda}{\theta}\left(\mathrm{e}^{-\theta W_1}-\mathrm{e}^{-\theta(W_0+X_0)}+\theta(A_1-W_0-X_0)^+\right)  ;
\end{split}
\]
here  the Lindley recursion \eqref{eq:Vk_recursion} has been used. 

Our objective is to analyze the MLE of both $\lambda$ and $\theta$ for a sample of size $n$, which we denote by $(\hat{\lambda}_n,\hat{\theta}_n)$. We assume (A1)--(A2) hold, i.e., compact parameter space and stationary $W_0$. In addition, (A3) holds because $H_\theta$ is exponential for all $\theta\in\Theta$. The log-likelihood $\ell_1(\lambda,\theta)$ is Lipschitz continuous (with respect to the observations) and so (A5) holds. We thus have that Theorem \ref{thm:cont_mle_consistency} holds and the MLE for $(\lambda,\theta)$ is strongly consistent. 

Clearly this is a smooth function. We proceed by computing the gradient and Hessian with respect to $(\theta,\lambda)$. It takes some elementary calculus to verify that gradient is given by
\begin{align}
\label{eq:exp_theta_gradient}
\nabla(\ell_1(\lambda,\theta))_\lambda&=\frac{1}{\lambda}-\frac{1}{\theta}\left(\mathrm{e}^{-\theta W_1}-\mathrm{e}^{-\theta(W_0+X_0)}+\theta(A_1-W_0-X_0)^+\right),\\
\nabla(\ell_1(\lambda,\theta))_\theta&=-W_1-\frac{\lambda}{\theta^2}\left[(\theta(W_0+X_0)+1)\mathrm{e}^{-\theta(W_0+X_0)}-(\theta W_1+1)\mathrm{e}^{-\theta W_1}\right].
\end{align}
For any given value of $\theta$, $\ell_1(\lambda,\theta)$ is a concave function in $\lambda$, and therefore the MLE of \eqref{eq:MLE_lambda} is given by
\begin{equation}\label{eq:exp_lambda_mle}
\hat{\lambda}_n(\theta)=\left({\frac{1}{n}\sum_{i=1}^n\left(\mathrm{e}^{-\theta W_i}-\mathrm{e}^{-\theta(W_{i-1}+X_{i-1})}+\theta(A_i-W_{i-1}-X_{i-1})^+\right)}\right)^{-1}\,{\theta}.
\end{equation}
The MLE $\hat{\theta}_n$ now follows by maximizing $\frac{1}{n}\sum_{i=1}^n \ell_1(\theta,\hat{\lambda}_n(\theta))$ for $\theta\in\Theta$. The optimizing $\theta$ is obtained either on the boundary of the parameter space $\Theta$ or by solving the first order condition $\nabla(\ell_1(\theta,\hat{\lambda}_n(\theta))_\theta=0$. As $n\to\infty$ we are guaranteed to find an interior solution as long as $\theta_0\in\Theta^{\mathrm{o}}$.

Taking second derivatives, we find that the entries of the Hessian matrix are given by
\begin{align*}
\nabla^2(\ell_1(\lambda,\theta))_{\lambda,\lambda} &= -\frac{1}{\lambda^2}, \\
\nabla^2(\ell_1(\lambda,\theta))_{\theta,\lambda} &=
\nabla^2(\ell_1(\lambda,\theta))_{\lambda,\theta} = -\frac{1}{\theta^2}\left[(\theta(W_0+X_0)+1)\mathrm{e}^{-\theta(W_0+X_0)}-(\theta W_1+1)\mathrm{e}^{-\theta W_1}\right], \\
\nabla^2(\ell_1(\lambda,\theta))_{\theta,\theta} &= -\frac{\lambda}{\theta^3}\Big[(1+(\theta W_1+1)^2)\mathrm{e}^{- \theta W_1} \nonumber   -(1+(\theta(W_0+X_0)+1)^2)\mathrm{e}^{- \theta (W_0+X_0)}\Big].
\end{align*}
As $w^k\mathrm{e}^{-\theta w}$ is a bounded function on $w\in[0,\infty)$ for any $\theta>0$ and $k\in{\mathbb N}$, we conclude that $-\E\nabla^2(\ell_1(\lambda_0,\theta_0))<\infty$ and therefore (A4) applies, which implies that Theorem \ref{thm:cont_mle_normal} holds. In particular, as $n\to\infty$,
\[
\sqrt{n}\left(\begin{array}{c}
\hat{\lambda}_n-\lambda_0 \\
\hat{\theta}_n-\theta_0
\end{array} \right)\darrow \mathrm{N}(0,I(\lambda_0,\theta_0)),
\]
and the asymptotic covariance is given by the inverse of the Hessian,
\[
I(\lambda_0,\theta_0)=-\E\nabla^2(\ell_1(\lambda_0,\theta_0)).
\]

We performed numerical experiments to assess the performance of the estimation procedure. 
Table \ref{tbl:exp_confidence} presents approximated confidence intervals for the maximum-likelihood estimators $(\hat{\lambda}_n,\hat{\theta}_n)$; in addition we also evaluated the estimator of the arrival rate based on idle periods, denoted by $\tilde{\lambda}_n$, as discussed in Section \ref{sec:lambda_estimation}. The confidence intervals are evaluated for three congestion levels, namely $\rho=\lambda\,\E X\in\{0.5,1,2\}$. In all experiments the sample size (of observed waiting times) was $n=1\hspace{0.3mm}000$. Evidently, however, the heavier the system load, the fewer the number of idle periods: on average $C_n=C_{1\hspace{0.23mm}000}$ equals $665, 457,$ and $267$ for the three congestion levels. 

The numerical output is summarized in Table \ref{tbl:exp_confidence}. 
A first observation is that the MLE for $\hat\theta_n$ is more accurate as the system load increases. This is because the patience threshold is reached more often and therefore the data is more informative. Furthermore, for a lightly loaded system the MLE $\hat\lambda_n$ and idle-period based estimator $\tilde\lambda_n$ yield similar (accurate) results, as opposed to the high-load regime in which  the two estimators behave quite differently: 
\begin{itemize}
\item[$\circ$]
As a consequence of the fact that idle periods are observed considerably less frequently for a high load, $\tilde\lambda_n$
becomes substantially less accurate. 
\item[$\circ$]
The accuracy of $\hat\lambda_n$, however, is only slightly reduced in the high-load regime. This is potentially due to the better accuracy of the estimation of $\theta$, being jointly estimated with $\lambda$ in the MLE procedure.
\end{itemize}
In the most heavily loaded example ($\rho=2$, that is) the 95\% confidence interval of the MLE is very similar to the 80\% confidence interval of the idle-period based estimator. 
We conclude that even though the observations of the waiting times and the effective interarrival observations are highly dependent, the MLE $\hat\lambda_n$ provides substantially better confidence intervals than the idle-period based estimator $\tilde\lambda_n$.

\begin{table}[h]
\centering
\footnotesize{
\begin{tabular}{|c|c|c|c|c|c|} \hline \hline
\multicolumn{2}{|l|}{Estimator $|$ Load}    & $80\%$  & $90\%$  & $95\%$ & $99\%$  \\ \hline 
 \multirow{3.3}*{$\hat{\theta}_n$}      & $\rho=0.5$ &  $[0.423,0.592]$  & $[0.403,0.620]$ & $[0.386,0.645]$ & $[0.352,0.691]$   \\
      & $\rho=1$ & $[0.454,0.555]$ & $[0.441,0.571]$ &$[0.431,0.586]$ & $[0.407,0.612]$  \\
      & $\rho=2$ & $[0.469,0.536]$ & $[0.460,0.546]$ &$[0.453,0.555]$ & $[0.440,0.573]$  \\ \hline 
  \multirow{3.3}*{$\hat{\lambda}_n$}      & $\rho=0.5$ & $[0.958,1.048]$  & $[0.948,1.061]$ & $[0.937,1.074]$ & $[0.918,1.099]$   \\
      & $\rho=1$ & $[0.955,1.053]$ & $[0.944,1.067]$ &$[0.933,1.081]$ & $[0.912,1.104]$  \\
      & $\rho=2$ & $[0.951,1.061]$ & $[0.936,1.079]$ &$[0.924,1.096]$ & $[0.901,1.128]$  \\ \hline
  \multirow{3.3}*{$\tilde{\lambda}_n$}      & $\rho=0.5$ & $[0.953,1.052]$  & $[0.941,1.069]$ & $[0.929,1.082]$ & $[0.909,1.110]$   \\
      & $\rho=1$ & $[0.944,1.063]$ & $[0.928,1.083]$ &$[0.915,1.101]$ & $[0.890,1.134]$  \\
      & $\rho=2$ & $[0.928,1.085]$ & $[0.908,1.110]$ &$[0.892,1.134]$ & $[0.862,1.182]$  \\ \hline \hline
\end{tabular}\vspace{2mm}
\caption{Confidence intervals for the MLE\,s $(\hat{\theta}_n,\hat{\lambda}_n)$, as well as the arrival rate estimated from idle periods $\tilde{\lambda}_n$, for different confidence levels. A total of $n=1\,000$ waiting observations were generated, $M=10\,000$ times for each parameter setting. The parameter of the exponential patience parameter was $\theta=0.5$, the arrival rate was $\lambda=1$, and the service requirements followed a Gamma distribution with parameters $(\eta,\mu)$ with $\mu=1$ and $\eta\in\{0.5,1,2\}$. }
\label{tbl:exp_confidence}}
\end{table}

Table \ref{tbl:exp_confidence_multi} presents empirical confidence intervals for the maximum likelihood estimators of $(\lambda,\theta)$ for a multi-server system with $s=5$.  The confidence intervals are wider than in the single-server case, even though the sample size was taken twice as big, being indicative of the fact that the variance of the estimation error in a system with multiple servers is higher. As in the single-server case, lower system load increases the variance of the estimation error for the patience parameter $\theta$. This effect is even stronger in the multi-server setting due to the fact that balking occurs only when all 5 servers are working, and such a state is not frequently observed. For higher load the estimation of the patience parameter is much more accurate, as was the case for the single-server system. For the arrival rate we observe high accuracy for all load levels.

\begin{table}[h]
\centering
\footnotesize{
\begin{tabular}{|c|c|c|c|c|c|} \hline \hline
\multicolumn{2}{|l|}{Estimator $|$ Load}    & $80\%$  & $90\%$  & $95\%$ & $99\%$  \\ \hline 
 \multirow{3.3}*{$\hat{\theta}_n$}      & $\rho=0.5$ &  $[0.282,0.568]$  & $[0.248, 0.618]$ & $[0.221, 0.664]$ & $[0.172,0.771]$   \\
      & $\rho=1$ & $[0.364, 0.443]$ & $[0.354, 0.455]$ &$[0.345, 0.466]$ & $[0.326, 0.486]$  \\
      & $\rho=2$ & $[0.382, 0.421 ]$ & $[ 0.376, 0.427]$ &$[ 0.371, 0.432]$ & $[0.363, 0.441 ]$  \\ \hline 
  \multirow{3.3}*{$\hat{\lambda}_n$}      & $\rho=0.5$ & $[0.972,1.032]$  & $[0.963,1.040$ & $[0.957,1.048]$ & $[0.945,1.061]$   \\
      & $\rho=1$ & $[0.970, 1.035]$ & $[0.961, 1.045]$ &$[0.953, 1.053]$ & $[0.939, 1.071]$  \\
      & $\rho=2$ & $[0.964, 1.041 ]$ & $[0.953, 1.053]$ &$[0.945, 1.063]$ & $[0.929, 1.084 ]$  \\ \hline
\end{tabular}\vspace{2mm}
\caption{Empirical confidence intervals for the MLE\,s $(\hat{\theta}_n,\hat{\lambda}_n)$ for different confidence levels. A total of $n=2\,000$ waiting observations were generated, $M=10\,000$ times for each parameter setting. The parameter of the exponential patience parameter was $\theta=0.4$, the arrival rate was $\lambda=1$. The data was simulated for a system with $s=5$ servers, and the service requirements followed a Gamma distribution with parameters $(\eta,\mu)$ with $\mu=0.8$ and $\eta\in\{2,4,8\}$, corresponding to loads of  $\rho=\frac{\eta}{s\mu}\in\{0.5,1,2\}$. }
\label{tbl:exp_confidence_multi}}
\end{table}

\subsection{(Generalized) hyperexponential patience}\label{sec:mle_hyper}
This subsection focuses on the case of the patience-level distribution being generalized hyperexponential (GHE), meaning that the corresponding cdf can be written as a mixture of exponential cdf\,s, with weights that sum to~1 but that are not necessarily positive --- this in contrast with the standard hyperexponential (HE) distribution, where the weights are assumed to be positive. This case is particularly relevant because,
as argued in \cite{RJ2013}, this distribution has a good empirical fit to patience data. Another motivation for considering this distribution lies in the known fact that the cdf of any continuous positive random variable can be approximated arbitrarily accurately (in terms of a suitably defined metric) by a GHE cdf \cite{BH1986}.
 
For convenience we now assume the arrival rate $\lambda$ to be known, but it is noted that it can be estimated in a similar manner as in the exponential case discussed in Section \ref{sec:mle_exponential}. Suppose the degree of the GHE distribution is $p\in{\mathbb N}$:
\[
\overline{H}_\theta(x)=\sum_{k=1}^p \alpha_k \mathrm{e}^{-\beta_k x}, \ x\geqslant 0,
\]
where $\theta=(\alpha_1,\ldots,\alpha_p,\beta_1,\ldots,\beta_p)\in{\mathbb R}^{2p}$, $\sum_{k=1}^p\alpha_k=1$ (where we, importantly, do not assume the $\alpha_k$ to be positive) and $\beta_k>0$ for $k=1,\ldots,p$. Without loss of generality we assume that $\alpha_1>\alpha_2>\ldots>\alpha_p$. Denote the coordinates of the true parameter $\theta_0$ by $(\alpha_{0k},\beta_{0k})$ for $k=1,\ldots,p$.
It means that we are to identify the $2p$-dimensional parameter vector
\[\theta_0=(\alpha_{01},\ldots,\alpha_{0p},\beta_{01},\ldots,\beta_{0p}).\] Similarly to the previous example of the exponential distribution, Assumptions (A1)--(A3) and (A5) are satisfied when assuming stationary waiting times and a compact parameter space $\Theta$. Therefore, the MLE is strongly consistent by Theorem \ref{thm:cont_mle_consistency}. {One should be cautious when fitting a GHE distribution because of further conditions to be imposed on the parameters to make sure $\overline{H}_\theta(\cdot)$ is a proper cdf. Even though convergence to the true parameters of the distribution is eventually guaranteed, for finite samples the estimated parameters may not yield a proper distribution. For an in-depth discussion on these conditions and the identifiability of GHE distributions, we refer \cite[Section 3]{BH1986}.}

Using the representation \eqref{eq:V_log_likelihood}, the log-likelihood pertaining to a single observation is calculated in a similar manner as in the exponential case, yielding 
\[
\ell_1(\theta) = \log\lambda+\log\left(\sum_{k=1}^p \alpha_k \mathrm{e}^{-\beta_k W_1}\right)-\lambda\sum_{k=1}^p \frac{\alpha_k}{\beta_k}\left(\mathrm{e}^{-\beta_k W_1}-\mathrm{e}^{-\beta_k(W_0+X_0)}\right)-\lambda(A_1-W_0-X_0)^+.\]
As follows with some elementary algebra, the gradient is then given by, for $j=1,\ldots,p$, 
\begin{align*}
\nabla(\ell_1(\theta))_{\alpha_j}&=\frac{\mathrm{e}^{-\beta_j W_1}}{\sum_{k=1}^p \alpha_k \mathrm{e}^{-\beta_k W_1}}- \frac{\lambda}{\beta_j}\left(\mathrm{e}^{-\beta_j W_1}-\mathrm{e}^{-\beta_j(W_0+X_0)}\right),\\
\nabla(\ell_1(\theta))_{\beta_j}&=-\frac{W_1\alpha_j\mathrm{e}^{-\beta_j W_1}}{\sum_{k=1}^p \alpha_k \mathrm{e}^{-\beta_k W_1}}-\frac{\lambda \alpha_j}{\beta_j^2}\left[(\beta_j(W_0+X_0)+1)\mathrm{e}^{-\beta_j(W_0+X_0)}-(\beta_j W_1+1)\mathrm{e}^{-\beta_j W_1}\right].
\end{align*}
The Hessian can be derived by computing the matrix of second derivatives. For the evaluation of the asymptotic covariance, a convenient alternative is to apply \eqref{eq:psi_covariance}, i.e.,
\[
I(\theta_0)=-\E\nabla^2(\ell_1(\theta_0))=\E[\nabla(\ell_1(\theta_0))\nabla(\ell_1(\theta_0))^\top],
\]
so as to avoid the symbolic evaluation of the matrix of second derivatives. The entries of $I(\theta_0)$ are finite because these are combinations of terms of the following types:
\begin{itemize}
\item[$\circ$]
Products of polynomials (of degree at most 2) and exponentials. For example, we come across  a term that is, up to a multiplicative constant, $W_0^2\mathrm{e}^{-\beta_j W_0}$. Observe that the mapping $x\mapsto x^k e^{-x}$ is bounded for $x\geqslant 0.$
\item[$\circ$]
Ratios of the form, e.g., up to a multiplicative constant
\[
\left|\frac{W_1\mathrm{e}^{-\beta_j W_1}}{\sum_{k=1}^p \alpha_k \mathrm{e}^{-\beta_k W_1}}\right|\leqslant W_1.
\]
\end{itemize}

In addition, it is verified that if the patience-level distribution is light-tailed (which is the case for the GHE distribution), then the stationary waiting time $W$ is also light-tailed and has finite moments. Indeed, note that (\ref{eq:v-Volterra-multi}) implies 
\[
v(x) \leqslant \lambda \pi_0 
+ \lambda \int_0^{\infty} v(y)\overline{H}(y)\diff y = \lambda(1-P_{\ell}),
\] 
so that we have from (\ref{eq:w-general}) that the density of $W$ is bounded as $w(x) \leqslant \lambda \overline{H}(x)$, i.e., if the patience level distribution is light-tailed, so is $W$.
As a consequence, from the above and Theorem~\ref{thm:cont_mle_normal}, the estimation errors (scaled by $\sqrt{n}$, that is) converge to a multivariate normal distribution. 

Although there are theoretical guarantees for the asymptotic performance of the MLE, computation of the MLE is not straightforward, even for small parameter spaces such as $p=3$. It requires maximizing $\ell_n(\theta)/n$, i.e., solving a non-linear and non-concave program in $p\times(p-1)$ variables, with both equality and inequality constraints. This is computationally highly challenging, and standard optimization methods may lead to local maxima. In particular, observe that $\nabla(\ell_1(\theta))_{\beta_j}$ is very small for large values of $\beta_j$, which implies that $\ell_n(\theta)/n$ displays very `flat' behavior for large values of these $\beta_j$. Our experiments revealed that a direct implementation using standard optimization packages often led to points that were even not local maxima. As a consequence, we decided to write dedicated code to compute the MLE.

In light of the inherent complexity of maximizing the likelihood function, in our numerical experiments we have applied the following nested two-step heuristic optimization method.
\begin{enumerate}
\item[(a)] For each vector of ${\boldsymbol \alpha}=(\alpha_1,\ldots,\alpha_{p})$ the objective functions $\ell_n(\theta)/n$ was maximized with respect to ${\boldsymbol \beta}=(\beta_1,\ldots,\beta_{p})$ using a conventional coordinate descent algorithm. This step was typically fast and accurate, as for given ${\boldsymbol \alpha}$ the objective function behaves nicely. Now we have reduced the problem to an optimization over ${\boldsymbol \alpha}.$
\item[(b)] Then a standard L-BFGS quasi-Newton method (see e.g.\ \cite{BLNZ1995}) is applied to $\ell_n(\theta)/n$, so as to search for the optimal vector ${\boldsymbol \alpha}$, with ${\boldsymbol \beta}$ being parameterized by ${\boldsymbol \alpha}$. The optimization is carried out with the following constraints that ensure that the parameters yield a proper distribution \cite{BH1986}:
\begin{align*}
\left\lbrace {\boldsymbol \alpha}: \alpha_1>\alpha_2>\ldots>\alpha_p, \ \sum_{i=1}^p\alpha_i = 1 , \
\sum_{i=1}^p \alpha_i \beta_i>0 \right\rbrace .
\end{align*}

\end{enumerate}Note that there is no firm guarantee that this heuristic method converges to the optimal parameters. To overcome this, the search for the optimizing vector ${\boldsymbol \alpha}$ was conducted multiple times for different initial values. The resulting method turned out time-consuming, especially for a relatively large sample sizes $n$ and/or a relatively large dimension $p$ of the parameter space. However, it typically returns considerably more robust results than off-the-shelf optimization routines (that in addition tended to converge very slowly).

Table \ref{tbl:hyper_confidence} presents the marginal confidence intervals obtained by the normal approximation for an example of a HE cdf $H_\theta(\cdot)$ with $p=2$. Observe that the variance of the estimation error is large, even for a substantial sample size. The accuracy is much higher for the low rate of $\beta_1=0.25$ than the higher rate of $\beta_2=1$. This may be explained by the fact that the likelihood function is almost flat for high values of $\beta_2$. 
\begin{table}
\centering
\footnotesize{
\begin{tabular}{|c|c|c|c|c|c|} \hline \hline
\multicolumn{2}{|l|}{Estimator $|$ Load}    & $80\%$  & $90\%$  & $95\%$ & $99\%$  \\ \hline 
 \multirow{3.3}*{$\hat{\alpha}_{1n}$}      & $\rho=0.5$ &  $[0.500,0.908]$  & $[0.500,0.967]$ & $[0.500,1.000]$ & $[0.500,1.000]$   \\
      & $\rho=1$ & $[0.609,0.791]$ & $[0.583,0.817]$ &$[0.560,0.840]$ & $[0.516,0.884]$  \\
      & $\rho=2$ & $[0.660,0.740]$ & $[0.649,0.750]$ &$[0.639,0.761]$ & $[0.620,0.780]$  \\ \hline 
  \multirow{3.3}*{$\hat{\beta}_{1n}$}      & $\rho=0.5$ &  $[0.183,0.317]$  & $[0.164,0.336]$ & $[0.147,0.353]$ & $[0.115,0.385]$   \\
      & $\rho=1$ & $[0.223,0.277]$ & $[0.216,0.284]$ &$[0.209,0.291]$ & $[0.196,0.304]$  \\
      & $\rho=2$ & $[0.240,0.260]$ & $[0.237,0.262]$ &$[0.235,0.265]$ & $[ 0.230,0.270]$  \\ \hline
  \multirow{3.3}*{$\hat{\beta}_{2n}$}      & $\rho=0.5$ &  $[0.539,1.460]$  & $[0.408,1.592]$ & $[0.295,1.705]$ & $[0.073,1.930]$   \\
      & $\rho=1$ & $[0.768,1.232]$ & $[0.702,1.298]$ &$[0.645,1.355]$ & $[0.533,1.467]$  \\
      & $\rho=2$ & $[0.870,1.130]$ & $[0.833,1.167]$ &$[0.801,1.199]$ & $[0.739,1.261]$  \\ \hline \hline
\end{tabular}\vspace{2mm}
\caption{Confidence intervals for the MLE\,s $(\hat{\alpha}_{1n},\hat{\beta}_{1n},\hat{\beta}_{2n})$ based on a normal approximation for a sample size of $n=100\,000$, for different confidence levels. The true values of the parameters $(\alpha_{1},\alpha_{2},\beta_1,\beta_2)$ of the HE patience were $(0.7,0.3,0.25,1)$, the arrival rate was $\lambda=1$, and the service requirements followed a Gamma distribution with parameters $(\eta,\mu)$ with $\mu=1$ and $\eta\in\{0.5,1,2\}$.
}
\label{tbl:hyper_confidence}}
\end{table}

The results presented in Table \ref{tbl:hyper_confidence} show that the estimates are concentrated around the true values. The variance of the estimation error, however, is high, even more so given the large number of observations used. It should be realized, though, that HE distributions are not always easily distinguishable, in the sense that seemingly different HE distributions may lead to a very similar cdf. Clearly, from a practical standpoint the main question is not whether the correct parameters are recovered, but rather whether the cdf corresponding to the estimated parameters provides a good fit with the true cdf. 
In this respect the main conclusion of our experiments is that our MLE procedure performs remarkably well. For example, in Figure~\ref{fig:H_MLE_fit} we illustrate this by presenting the true cdf and the estimated cdf based on four random samples (using $n=10\,000$ observations). We have used the same parameters as in Table \ref{tbl:hyper_confidence}, with $\rho=1$.
We observe that, although the fitted parameters greatly differ from the correct ones, the fit of the (complementary) cdf is highly accurate. 

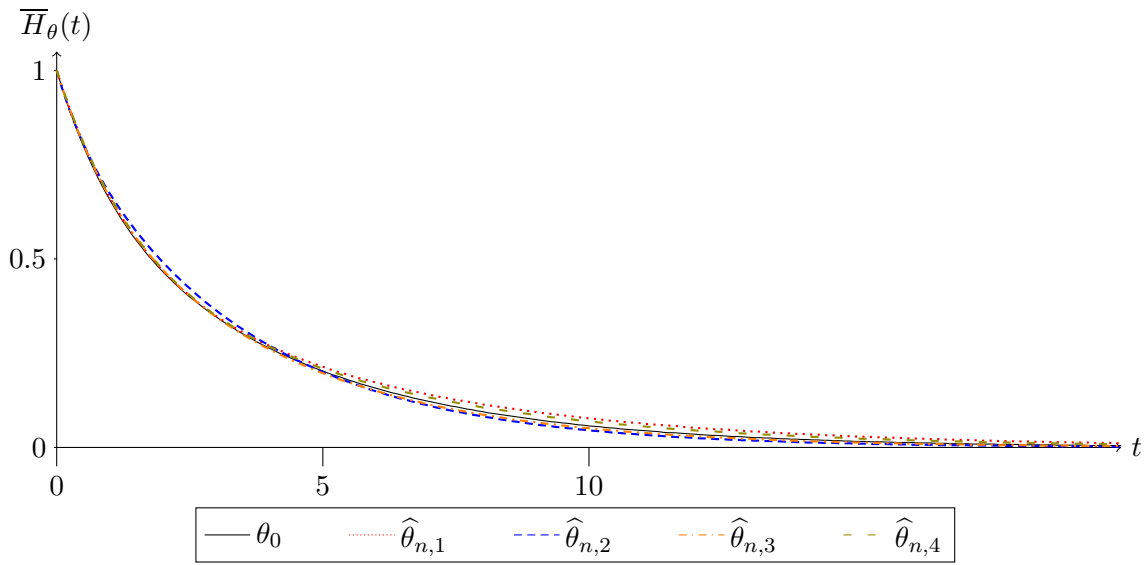
\begin{figure}[h]
\begin{subfigure}{.98\linewidth}
\centering
\begin{tikzpicture}[x=0.6cm,y=4cm]
 \def\xmin{0}
 \def\xmax{20}
 \def\ymin{0}
 \def\ymax{1.05}
 \draw[->] (\xmin,\ymin) -- (\xmax,\ymin) node[right] {$t$} ;
 \draw[->] (\xmin,\ymin) -- (\xmin,\ymax)  node[above] {$\overline  H_\theta(t)$};
 \foreach \x in {0,5,10} {
 \node at (\x,\ymin-0.05) [below] {\x};
 \draw[-] (\x,\ymin) -- (\x,{\ymin-0.05});
 }
\foreach \y in {0,0.5,1} {
 \node at (\xmin,\y) [left] {\y};
 \draw[-] (\xmin,\y) -- ({\xmin-0.05},\y);
 }


\draw[black] ( 0 , 1 )-- ( 0.1 , 0.954 )-- ( 0.2 , 0.911 )-- ( 0.3 , 0.872 )-- ( 0.4 , 0.834 )-- ( 0.5 , 0.8 )-- ( 0.6 , 0.767 )-- ( 0.7 , 0.737 )-- ( 0.8 , 0.708 )-- ( 0.9 , 0.681 )-- ( 1 , 0.656 )-- ( 1.1 , 0.632 )-- ( 1.2 , 0.609 )-- ( 1.3 , 0.588 )-- ( 1.4 , 0.567 )-- ( 1.5 , 0.548 )-- ( 1.6 , 0.53 )-- ( 1.7 , 0.512 )-- ( 1.8 , 0.496 )-- ( 1.9 , 0.48 )-- ( 2 , 0.465 )-- ( 2.1 , 0.451 )-- ( 2.2 , 0.437 )-- ( 2.3 , 0.424 )-- ( 2.4 , 0.411 )-- ( 2.5 , 0.399 )-- ( 2.6 , 0.388 )-- ( 2.7 , 0.377 )-- ( 2.8 , 0.366 )-- ( 2.9 , 0.356 )-- ( 3 , 0.346 )-- ( 3.1 , 0.336 )-- ( 3.2 , 0.327 )-- ( 3.3 , 0.318 )-- ( 3.4 , 0.309 )-- ( 3.5 , 0.301 )-- ( 3.6 , 0.293 )-- ( 3.7 , 0.285 )-- ( 3.8 , 0.277 )-- ( 3.9 , 0.27 )-- ( 4 , 0.263 )-- ( 4.1 , 0.256 )-- ( 4.2 , 0.249 )-- ( 4.3 , 0.243 )-- ( 4.4 , 0.237 )-- ( 4.5 , 0.231 )-- ( 4.6 , 0.225 )-- ( 4.7 , 0.219 )-- ( 4.8 , 0.213 )-- ( 4.9 , 0.208 )-- ( 5 , 0.203 )-- ( 5.1 , 0.197 )-- ( 5.2 , 0.192 )-- ( 5.3 , 0.188 )-- ( 5.4 , 0.183 )-- ( 5.5 , 0.178 )-- ( 5.6 , 0.174 )-- ( 5.7 , 0.169 )-- ( 5.8 , 0.165 )-- ( 5.9 , 0.161 )-- ( 6 , 0.157 )-- ( 6.1 , 0.153 )-- ( 6.2 , 0.149 )-- ( 6.3 , 0.145 )-- ( 6.4 , 0.142 )-- ( 6.5 , 0.138 )-- ( 6.6 , 0.135 )-- ( 6.7 , 0.131 )-- ( 6.8 , 0.128 )-- ( 6.9 , 0.125 )-- ( 7 , 0.122 )-- ( 7.1 , 0.119 )-- ( 7.2 , 0.116 )-- ( 7.3 , 0.113 )-- ( 7.4 , 0.11 )-- ( 7.5 , 0.108 )-- ( 7.6 , 0.105 )-- ( 7.7 , 0.102 )-- ( 7.8 , 0.1 )-- ( 7.9 , 0.097 )-- ( 8 , 0.095 )-- ( 8.1 , 0.092 )-- ( 8.2 , 0.09 )-- ( 8.3 , 0.088 )-- ( 8.4 , 0.086 )-- ( 8.5 , 0.084 )-- ( 8.6 , 0.082 )-- ( 8.7 , 0.08 )-- ( 8.8 , 0.078 )-- ( 8.9 , 0.076 )-- ( 9 , 0.074 )-- ( 9.1 , 0.072 )-- ( 9.2 , 0.07 )-- ( 9.3 , 0.068 )-- ( 9.4 , 0.067 )-- ( 9.5 , 0.065 )-- ( 9.6 , 0.064 )-- ( 9.7 , 0.062 )-- ( 9.8 , 0.06 )-- ( 9.9 , 0.059 )-- ( 10 , 0.057 )-- ( 10.1 , 0.056 )-- ( 10.2 , 0.055 )-- ( 10.3 , 0.053 )-- ( 10.4 , 0.052 )-- ( 10.5 , 0.051 )-- ( 10.6 , 0.049 )-- ( 10.7 , 0.048 )-- ( 10.8 , 0.047 )-- ( 10.9 , 0.046 )-- ( 11 , 0.045 )-- ( 11.1 , 0.044 )-- ( 11.2 , 0.043 )-- ( 11.3 , 0.042 )-- ( 11.4 , 0.04 )-- ( 11.5 , 0.039 )-- ( 11.6 , 0.039 )-- ( 11.7 , 0.038 )-- ( 11.8 , 0.037 )-- ( 11.9 , 0.036 )-- ( 12 , 0.035 )-- ( 12.1 , 0.034 )-- ( 12.2 , 0.033 )-- ( 12.3 , 0.032 )-- ( 12.4 , 0.032 )-- ( 12.5 , 0.031 )-- ( 12.6 , 0.03 )-- ( 12.7 , 0.029 )-- ( 12.8 , 0.029 )-- ( 12.9 , 0.028 )-- ( 13 , 0.027 )-- ( 13.1 , 0.026 )-- ( 13.2 , 0.026 )-- ( 13.3 , 0.025 )-- ( 13.4 , 0.025 )-- ( 13.5 , 0.024 )-- ( 13.6 , 0.023 )-- ( 13.7 , 0.023 )-- ( 13.8 , 0.022 )-- ( 13.9 , 0.022 )-- ( 14 , 0.021 )-- ( 14.1 , 0.021 )-- ( 14.2 , 0.02 )-- ( 14.3 , 0.02 )-- ( 14.4 , 0.019 )-- ( 14.5 , 0.019 )-- ( 14.6 , 0.018 )-- ( 14.7 , 0.018 )-- ( 14.8 , 0.017 )-- ( 14.9 , 0.017 )-- ( 15 , 0.016 )-- ( 15.1 , 0.016 )-- ( 15.2 , 0.016 )-- ( 15.3 , 0.015 )-- ( 15.4 , 0.015 )-- ( 15.5 , 0.015 )-- ( 15.6 , 0.014 )-- ( 15.7 , 0.014 )-- ( 15.8 , 0.013 )-- ( 15.9 , 0.013 )-- ( 16 , 0.013 )-- ( 16.1 , 0.013 )-- ( 16.2 , 0.012 )-- ( 16.3 , 0.012 )-- ( 16.4 , 0.012 )-- ( 16.5 , 0.011 )-- ( 16.6 , 0.011 )-- ( 16.7 , 0.011 )-- ( 16.8 , 0.01 )-- ( 16.9 , 0.01 )-- ( 17 , 0.01 )-- ( 17.1 , 0.01 )-- ( 17.2 , 0.009 )-- ( 17.3 , 0.009 )-- ( 17.4 , 0.009 )-- ( 17.5 , 0.009 )-- ( 17.6 , 0.009 )-- ( 17.7 , 0.008 )-- ( 17.8 , 0.008 )-- ( 17.9 , 0.008 )-- ( 18 , 0.008 )-- ( 18.1 , 0.008 )-- ( 18.2 , 0.007 )-- ( 18.3 , 0.007 )-- ( 18.4 , 0.007 )-- ( 18.5 , 0.007 )-- ( 18.6 , 0.007 )-- ( 18.7 , 0.007 )-- ( 18.8 , 0.006 )-- ( 18.9 , 0.006 )-- ( 19 , 0.006 )-- ( 19.1 , 0.006 )-- ( 19.2 , 0.006 )-- ( 19.3 , 0.006 )-- ( 19.4 , 0.005 )-- ( 19.5 , 0.005 )-- ( 19.6 , 0.005 )-- ( 19.7 , 0.005 )-- ( 19.8 , 0.005 )-- ( 19.9 , 0.005 )-- ( 20 , 0.005 );

 \draw[red,dotted,thick] ( 0 , 1 )-- ( 0.1 , 0.956 )-- ( 0.2 , 0.916 )-- ( 0.3 , 0.877 )-- ( 0.4 , 0.841 )-- ( 0.5 , 0.806 )-- ( 0.6 , 0.774 )-- ( 0.7 , 0.744 )-- ( 0.8 , 0.715 )-- ( 0.9 , 0.688 )-- ( 1 , 0.662 )-- ( 1.1 , 0.638 )-- ( 1.2 , 0.615 )-- ( 1.3 , 0.593 )-- ( 1.4 , 0.572 )-- ( 1.5 , 0.552 )-- ( 1.6 , 0.534 )-- ( 1.7 , 0.516 )-- ( 1.8 , 0.499 )-- ( 1.9 , 0.483 )-- ( 2 , 0.468 )-- ( 2.1 , 0.453 )-- ( 2.2 , 0.439 )-- ( 2.3 , 0.426 )-- ( 2.4 , 0.414 )-- ( 2.5 , 0.402 )-- ( 2.6 , 0.39 )-- ( 2.7 , 0.379 )-- ( 2.8 , 0.368 )-- ( 2.9 , 0.358 )-- ( 3 , 0.348 )-- ( 3.1 , 0.339 )-- ( 3.2 , 0.33 )-- ( 3.3 , 0.322 )-- ( 3.4 , 0.313 )-- ( 3.5 , 0.305 )-- ( 3.6 , 0.298 )-- ( 3.7 , 0.29 )-- ( 3.8 , 0.283 )-- ( 3.9 , 0.276 )-- ( 4 , 0.269 )-- ( 4.1 , 0.263 )-- ( 4.2 , 0.257 )-- ( 4.3 , 0.251 )-- ( 4.4 , 0.245 )-- ( 4.5 , 0.239 )-- ( 4.6 , 0.234 )-- ( 4.7 , 0.229 )-- ( 4.8 , 0.223 )-- ( 4.9 , 0.219 )-- ( 5 , 0.214 )-- ( 5.1 , 0.209 )-- ( 5.2 , 0.204 )-- ( 5.3 , 0.2 )-- ( 5.4 , 0.196 )-- ( 5.5 , 0.191 )-- ( 5.6 , 0.187 )-- ( 5.7 , 0.183 )-- ( 5.8 , 0.18 )-- ( 5.9 , 0.176 )-- ( 6 , 0.172 )-- ( 6.1 , 0.169 )-- ( 6.2 , 0.165 )-- ( 6.3 , 0.162 )-- ( 6.4 , 0.158 )-- ( 6.5 , 0.155 )-- ( 6.6 , 0.152 )-- ( 6.7 , 0.149 )-- ( 6.8 , 0.146 )-- ( 6.9 , 0.143 )-- ( 7 , 0.14 )-- ( 7.1 , 0.137 )-- ( 7.2 , 0.134 )-- ( 7.3 , 0.132 )-- ( 7.4 , 0.129 )-- ( 7.5 , 0.126 )-- ( 7.6 , 0.124 )-- ( 7.7 , 0.122 )-- ( 7.8 , 0.119 )-- ( 7.9 , 0.117 )-- ( 8 , 0.114 )-- ( 8.1 , 0.112 )-- ( 8.2 , 0.11 )-- ( 8.3 , 0.108 )-- ( 8.4 , 0.106 )-- ( 8.5 , 0.104 )-- ( 8.6 , 0.102 )-- ( 8.7 , 0.1 )-- ( 8.8 , 0.098 )-- ( 8.9 , 0.096 )-- ( 9 , 0.094 )-- ( 9.1 , 0.092 )-- ( 9.2 , 0.09 )-- ( 9.3 , 0.089 )-- ( 9.4 , 0.087 )-- ( 9.5 , 0.085 )-- ( 9.6 , 0.084 )-- ( 9.7 , 0.082 )-- ( 9.8 , 0.08 )-- ( 9.9 , 0.079 )-- ( 10 , 0.077 )-- ( 10.1 , 0.076 )-- ( 10.2 , 0.074 )-- ( 10.3 , 0.073 )-- ( 10.4 , 0.071 )-- ( 10.5 , 0.07 )-- ( 10.6 , 0.069 )-- ( 10.7 , 0.067 )-- ( 10.8 , 0.066 )-- ( 10.9 , 0.065 )-- ( 11 , 0.064 )-- ( 11.1 , 0.062 )-- ( 11.2 , 0.061 )-- ( 11.3 , 0.06 )-- ( 11.4 , 0.059 )-- ( 11.5 , 0.058 )-- ( 11.6 , 0.057 )-- ( 11.7 , 0.056 )-- ( 11.8 , 0.054 )-- ( 11.9 , 0.053 )-- ( 12 , 0.052 )-- ( 12.1 , 0.051 )-- ( 12.2 , 0.05 )-- ( 12.3 , 0.049 )-- ( 12.4 , 0.048 )-- ( 12.5 , 0.048 )-- ( 12.6 , 0.047 )-- ( 12.7 , 0.046 )-- ( 12.8 , 0.045 )-- ( 12.9 , 0.044 )-- ( 13 , 0.043 )-- ( 13.1 , 0.042 )-- ( 13.2 , 0.042 )-- ( 13.3 , 0.041 )-- ( 13.4 , 0.04 )-- ( 13.5 , 0.039 )-- ( 13.6 , 0.038 )-- ( 13.7 , 0.038 )-- ( 13.8 , 0.037 )-- ( 13.9 , 0.036 )-- ( 14 , 0.036 )-- ( 14.1 , 0.035 )-- ( 14.2 , 0.034 )-- ( 14.3 , 0.034 )-- ( 14.4 , 0.033 )-- ( 14.5 , 0.032 )-- ( 14.6 , 0.032 )-- ( 14.7 , 0.031 )-- ( 14.8 , 0.03 )-- ( 14.9 , 0.03 )-- ( 15 , 0.029 )-- ( 15.1 , 0.029 )-- ( 15.2 , 0.028 )-- ( 15.3 , 0.028 )-- ( 15.4 , 0.027 )-- ( 15.5 , 0.027 )-- ( 15.6 , 0.026 )-- ( 15.7 , 0.026 )-- ( 15.8 , 0.025 )-- ( 15.9 , 0.025 )-- ( 16 , 0.024 )-- ( 16.1 , 0.024 )-- ( 16.2 , 0.023 )-- ( 16.3 , 0.023 )-- ( 16.4 , 0.022 )-- ( 16.5 , 0.022 )-- ( 16.6 , 0.022 )-- ( 16.7 , 0.021 )-- ( 16.8 , 0.021 )-- ( 16.9 , 0.02 )-- ( 17 , 0.02 )-- ( 17.1 , 0.02 )-- ( 17.2 , 0.019 )-- ( 17.3 , 0.019 )-- ( 17.4 , 0.018 )-- ( 17.5 , 0.018 )-- ( 17.6 , 0.018 )-- ( 17.7 , 0.017 )-- ( 17.8 , 0.017 )-- ( 17.9 , 0.017 )-- ( 18 , 0.016 )-- ( 18.1 , 0.016 )-- ( 18.2 , 0.016 )-- ( 18.3 , 0.016 )-- ( 18.4 , 0.015 )-- ( 18.5 , 0.015 )-- ( 18.6 , 0.015 )-- ( 18.7 , 0.014 )-- ( 18.8 , 0.014 )-- ( 18.9 , 0.014 )-- ( 19 , 0.014 )-- ( 19.1 , 0.013 )-- ( 19.2 , 0.013 )-- ( 19.3 , 0.013 )-- ( 19.4 , 0.013 )-- ( 19.5 , 0.012 )-- ( 19.6 , 0.012 )-- ( 19.7 , 0.012 )-- ( 19.8 , 0.012 )-- ( 19.9 , 0.011 )-- ( 20 , 0.011 ) ;
 
\draw[blue,thick,densely dashed] ( 0 , 1 )-- ( 0.1 , 0.953 )-- ( 0.2 , 0.911 )-- ( 0.3 , 0.873 )-- ( 0.4 , 0.838 )-- ( 0.5 , 0.805 )-- ( 0.6 , 0.775 )-- ( 0.7 , 0.747 )-- ( 0.8 , 0.721 )-- ( 0.9 , 0.696 )-- ( 1 , 0.673 )-- ( 1.1 , 0.651 )-- ( 1.2 , 0.63 )-- ( 1.3 , 0.61 )-- ( 1.4 , 0.591 )-- ( 1.5 , 0.572 )-- ( 1.6 , 0.555 )-- ( 1.7 , 0.538 )-- ( 1.8 , 0.521 )-- ( 1.9 , 0.506 )-- ( 2 , 0.49 )-- ( 2.1 , 0.476 )-- ( 2.2 , 0.461 )-- ( 2.3 , 0.448 )-- ( 2.4 , 0.434 )-- ( 2.5 , 0.422 )-- ( 2.6 , 0.409 )-- ( 2.7 , 0.397 )-- ( 2.8 , 0.385 )-- ( 2.9 , 0.374 )-- ( 3 , 0.363 )-- ( 3.1 , 0.352 )-- ( 3.2 , 0.342 )-- ( 3.3 , 0.332 )-- ( 3.4 , 0.322 )-- ( 3.5 , 0.313 )-- ( 3.6 , 0.304 )-- ( 3.7 , 0.295 )-- ( 3.8 , 0.286 )-- ( 3.9 , 0.278 )-- ( 4 , 0.27 )-- ( 4.1 , 0.262 )-- ( 4.2 , 0.254 )-- ( 4.3 , 0.247 )-- ( 4.4 , 0.239 )-- ( 4.5 , 0.232 )-- ( 4.6 , 0.226 )-- ( 4.7 , 0.219 )-- ( 4.8 , 0.213 )-- ( 4.9 , 0.206 )-- ( 5 , 0.2 )-- ( 5.1 , 0.195 )-- ( 5.2 , 0.189 )-- ( 5.3 , 0.183 )-- ( 5.4 , 0.178 )-- ( 5.5 , 0.173 )-- ( 5.6 , 0.168 )-- ( 5.7 , 0.163 )-- ( 5.8 , 0.158 )-- ( 5.9 , 0.153 )-- ( 6 , 0.149 )-- ( 6.1 , 0.145 )-- ( 6.2 , 0.14 )-- ( 6.3 , 0.136 )-- ( 6.4 , 0.132 )-- ( 6.5 , 0.128 )-- ( 6.6 , 0.125 )-- ( 6.7 , 0.121 )-- ( 6.8 , 0.117 )-- ( 6.9 , 0.114 )-- ( 7 , 0.111 )-- ( 7.1 , 0.107 )-- ( 7.2 , 0.104 )-- ( 7.3 , 0.101 )-- ( 7.4 , 0.098 )-- ( 7.5 , 0.095 )-- ( 7.6 , 0.093 )-- ( 7.7 , 0.09 )-- ( 7.8 , 0.087 )-- ( 7.9 , 0.085 )-- ( 8 , 0.082 )-- ( 8.1 , 0.08 )-- ( 8.2 , 0.078 )-- ( 8.3 , 0.075 )-- ( 8.4 , 0.073 )-- ( 8.5 , 0.071 )-- ( 8.6 , 0.069 )-- ( 8.7 , 0.067 )-- ( 8.8 , 0.065 )-- ( 8.9 , 0.063 )-- ( 9 , 0.061 )-- ( 9.1 , 0.059 )-- ( 9.2 , 0.058 )-- ( 9.3 , 0.056 )-- ( 9.4 , 0.054 )-- ( 9.5 , 0.053 )-- ( 9.6 , 0.051 )-- ( 9.7 , 0.05 )-- ( 9.8 , 0.048 )-- ( 9.9 , 0.047 )-- ( 10 , 0.045 )-- ( 10.1 , 0.044 )-- ( 10.2 , 0.043 )-- ( 10.3 , 0.042 )-- ( 10.4 , 0.04 )-- ( 10.5 , 0.039 )-- ( 10.6 , 0.038 )-- ( 10.7 , 0.037 )-- ( 10.8 , 0.036 )-- ( 10.9 , 0.035 )-- ( 11 , 0.034 )-- ( 11.1 , 0.033 )-- ( 11.2 , 0.032 )-- ( 11.3 , 0.031 )-- ( 11.4 , 0.03 )-- ( 11.5 , 0.029 )-- ( 11.6 , 0.028 )-- ( 11.7 , 0.027 )-- ( 11.8 , 0.027 )-- ( 11.9 , 0.026 )-- ( 12 , 0.025 )-- ( 12.1 , 0.024 )-- ( 12.2 , 0.024 )-- ( 12.3 , 0.023 )-- ( 12.4 , 0.022 )-- ( 12.5 , 0.022 )-- ( 12.6 , 0.021 )-- ( 12.7 , 0.02 )-- ( 12.8 , 0.02 )-- ( 12.9 , 0.019 )-- ( 13 , 0.019 )-- ( 13.1 , 0.018 )-- ( 13.2 , 0.018 )-- ( 13.3 , 0.017 )-- ( 13.4 , 0.017 )-- ( 13.5 , 0.016 )-- ( 13.6 , 0.016 )-- ( 13.7 , 0.015 )-- ( 13.8 , 0.015 )-- ( 13.9 , 0.014 )-- ( 14 , 0.014 )-- ( 14.1 , 0.013 )-- ( 14.2 , 0.013 )-- ( 14.3 , 0.013 )-- ( 14.4 , 0.012 )-- ( 14.5 , 0.012 )-- ( 14.6 , 0.012 )-- ( 14.7 , 0.011 )-- ( 14.8 , 0.011 )-- ( 14.9 , 0.011 )-- ( 15 , 0.01 )-- ( 15.1 , 0.01 )-- ( 15.2 , 0.01 )-- ( 15.3 , 0.009 )-- ( 15.4 , 0.009 )-- ( 15.5 , 0.009 )-- ( 15.6 , 0.009 )-- ( 15.7 , 0.008 )-- ( 15.8 , 0.008 )-- ( 15.9 , 0.008 )-- ( 16 , 0.008 )-- ( 16.1 , 0.007 )-- ( 16.2 , 0.007 )-- ( 16.3 , 0.007 )-- ( 16.4 , 0.007 )-- ( 16.5 , 0.007 )-- ( 16.6 , 0.006 )-- ( 16.7 , 0.006 )-- ( 16.8 , 0.006 )-- ( 16.9 , 0.006 )-- ( 17 , 0.006 )-- ( 17.1 , 0.006 )-- ( 17.2 , 0.005 )-- ( 17.3 , 0.005 )-- ( 17.4 , 0.005 )-- ( 17.5 , 0.005 )-- ( 17.6 , 0.005 )-- ( 17.7 , 0.005 )-- ( 17.8 , 0.004 )-- ( 17.9 , 0.004 )-- ( 18 , 0.004 )-- ( 18.1 , 0.004 )-- ( 18.2 , 0.004 )-- ( 18.3 , 0.004 )-- ( 18.4 , 0.004 )-- ( 18.5 , 0.004 )-- ( 18.6 , 0.004 )-- ( 18.7 , 0.003 )-- ( 18.8 , 0.003 )-- ( 18.9 , 0.003 )-- ( 19 , 0.003 )-- ( 19.1 , 0.003 )-- ( 19.2 , 0.003 )-- ( 19.3 , 0.003 )-- ( 19.4 , 0.003 )-- ( 19.5 , 0.003 )-- ( 19.6 , 0.003 )-- ( 19.7 , 0.003 )-- ( 19.8 , 0.002 )-- ( 19.9 , 0.002 )-- ( 20 , 0.002 );
 
\draw[orange,thick,dashdotted] ( 0 , 1 )-- ( 0.1 , 0.956 )-- ( 0.2 , 0.914 )-- ( 0.3 , 0.875 )-- ( 0.4 , 0.839 )-- ( 0.5 , 0.805 )-- ( 0.6 , 0.773 )-- ( 0.7 , 0.743 )-- ( 0.8 , 0.714 )-- ( 0.9 , 0.688 )-- ( 1 , 0.662 )-- ( 1.1 , 0.638 )-- ( 1.2 , 0.616 )-- ( 1.3 , 0.594 )-- ( 1.4 , 0.574 )-- ( 1.5 , 0.554 )-- ( 1.6 , 0.536 )-- ( 1.7 , 0.518 )-- ( 1.8 , 0.501 )-- ( 1.9 , 0.485 )-- ( 2 , 0.47 )-- ( 2.1 , 0.455 )-- ( 2.2 , 0.441 )-- ( 2.3 , 0.427 )-- ( 2.4 , 0.414 )-- ( 2.5 , 0.402 )-- ( 2.6 , 0.39 )-- ( 2.7 , 0.378 )-- ( 2.8 , 0.367 )-- ( 2.9 , 0.356 )-- ( 3 , 0.346 )-- ( 3.1 , 0.336 )-- ( 3.2 , 0.326 )-- ( 3.3 , 0.317 )-- ( 3.4 , 0.308 )-- ( 3.5 , 0.299 )-- ( 3.6 , 0.291 )-- ( 3.7 , 0.283 )-- ( 3.8 , 0.275 )-- ( 3.9 , 0.267 )-- ( 4 , 0.26 )-- ( 4.1 , 0.253 )-- ( 4.2 , 0.246 )-- ( 4.3 , 0.239 )-- ( 4.4 , 0.232 )-- ( 4.5 , 0.226 )-- ( 4.6 , 0.22 )-- ( 4.7 , 0.214 )-- ( 4.8 , 0.208 )-- ( 4.9 , 0.202 )-- ( 5 , 0.197 )-- ( 5.1 , 0.192 )-- ( 5.2 , 0.186 )-- ( 5.3 , 0.181 )-- ( 5.4 , 0.176 )-- ( 5.5 , 0.172 )-- ( 5.6 , 0.167 )-- ( 5.7 , 0.163 )-- ( 5.8 , 0.158 )-- ( 5.9 , 0.154 )-- ( 6 , 0.15 )-- ( 6.1 , 0.146 )-- ( 6.2 , 0.142 )-- ( 6.3 , 0.138 )-- ( 6.4 , 0.134 )-- ( 6.5 , 0.131 )-- ( 6.6 , 0.127 )-- ( 6.7 , 0.124 )-- ( 6.8 , 0.121 )-- ( 6.9 , 0.117 )-- ( 7 , 0.114 )-- ( 7.1 , 0.111 )-- ( 7.2 , 0.108 )-- ( 7.3 , 0.105 )-- ( 7.4 , 0.103 )-- ( 7.5 , 0.1 )-- ( 7.6 , 0.097 )-- ( 7.7 , 0.095 )-- ( 7.8 , 0.092 )-- ( 7.9 , 0.09 )-- ( 8 , 0.087 )-- ( 8.1 , 0.085 )-- ( 8.2 , 0.083 )-- ( 8.3 , 0.081 )-- ( 8.4 , 0.078 )-- ( 8.5 , 0.076 )-- ( 8.6 , 0.074 )-- ( 8.7 , 0.072 )-- ( 8.8 , 0.07 )-- ( 8.9 , 0.069 )-- ( 9 , 0.067 )-- ( 9.1 , 0.065 )-- ( 9.2 , 0.063 )-- ( 9.3 , 0.062 )-- ( 9.4 , 0.06 )-- ( 9.5 , 0.058 )-- ( 9.6 , 0.057 )-- ( 9.7 , 0.055 )-- ( 9.8 , 0.054 )-- ( 9.9 , 0.052 )-- ( 10 , 0.051 )-- ( 10.1 , 0.05 )-- ( 10.2 , 0.048 )-- ( 10.3 , 0.047 )-- ( 10.4 , 0.046 )-- ( 10.5 , 0.045 )-- ( 10.6 , 0.043 )-- ( 10.7 , 0.042 )-- ( 10.8 , 0.041 )-- ( 10.9 , 0.04 )-- ( 11 , 0.039 )-- ( 11.1 , 0.038 )-- ( 11.2 , 0.037 )-- ( 11.3 , 0.036 )-- ( 11.4 , 0.035 )-- ( 11.5 , 0.034 )-- ( 11.6 , 0.033 )-- ( 11.7 , 0.032 )-- ( 11.8 , 0.031 )-- ( 11.9 , 0.031 )-- ( 12 , 0.03 )-- ( 12.1 , 0.029 )-- ( 12.2 , 0.028 )-- ( 12.3 , 0.027 )-- ( 12.4 , 0.027 )-- ( 12.5 , 0.026 )-- ( 12.6 , 0.025 )-- ( 12.7 , 0.025 )-- ( 12.8 , 0.024 )-- ( 12.9 , 0.023 )-- ( 13 , 0.023 )-- ( 13.1 , 0.022 )-- ( 13.2 , 0.022 )-- ( 13.3 , 0.021 )-- ( 13.4 , 0.02 )-- ( 13.5 , 0.02 )-- ( 13.6 , 0.019 )-- ( 13.7 , 0.019 )-- ( 13.8 , 0.018 )-- ( 13.9 , 0.018 )-- ( 14 , 0.017 )-- ( 14.1 , 0.017 )-- ( 14.2 , 0.016 )-- ( 14.3 , 0.016 )-- ( 14.4 , 0.016 )-- ( 14.5 , 0.015 )-- ( 14.6 , 0.015 )-- ( 14.7 , 0.014 )-- ( 14.8 , 0.014 )-- ( 14.9 , 0.014 )-- ( 15 , 0.013 )-- ( 15.1 , 0.013 )-- ( 15.2 , 0.013 )-- ( 15.3 , 0.012 )-- ( 15.4 , 0.012 )-- ( 15.5 , 0.012 )-- ( 15.6 , 0.011 )-- ( 15.7 , 0.011 )-- ( 15.8 , 0.011 )-- ( 15.9 , 0.01 )-- ( 16 , 0.01 )-- ( 16.1 , 0.01 )-- ( 16.2 , 0.01 )-- ( 16.3 , 0.009 )-- ( 16.4 , 0.009 )-- ( 16.5 , 0.009 )-- ( 16.6 , 0.009 )-- ( 16.7 , 0.008 )-- ( 16.8 , 0.008 )-- ( 16.9 , 0.008 )-- ( 17 , 0.008 )-- ( 17.1 , 0.008 )-- ( 17.2 , 0.007 )-- ( 17.3 , 0.007 )-- ( 17.4 , 0.007 )-- ( 17.5 , 0.007 )-- ( 17.6 , 0.007 )-- ( 17.7 , 0.006 )-- ( 17.8 , 0.006 )-- ( 17.9 , 0.006 )-- ( 18 , 0.006 )-- ( 18.1 , 0.006 )-- ( 18.2 , 0.006 )-- ( 18.3 , 0.005 )-- ( 18.4 , 0.005 )-- ( 18.5 , 0.005 )-- ( 18.6 , 0.005 )-- ( 18.7 , 0.005 )-- ( 18.8 , 0.005 )-- ( 18.9 , 0.005 )-- ( 19 , 0.005 )-- ( 19.1 , 0.004 )-- ( 19.2 , 0.004 )-- ( 19.3 , 0.004 )-- ( 19.4 , 0.004 )-- ( 19.5 , 0.004 )-- ( 19.6 , 0.004 )-- ( 19.7 , 0.004 )-- ( 19.8 , 0.004 )-- ( 19.9 , 0.004 )-- ( 20 , 0.003 ); 
 
\draw[olive,loosely dashed,thick] ( 0 , 1 )-- ( 0.1 , 0.957 )-- ( 0.2 , 0.917 )-- ( 0.3 , 0.879 )-- ( 0.4 , 0.843 )-- ( 0.5 , 0.809 )-- ( 0.6 , 0.778 )-- ( 0.7 , 0.747 )-- ( 0.8 , 0.719 )-- ( 0.9 , 0.692 )-- ( 1 , 0.666 )-- ( 1.1 , 0.642 )-- ( 1.2 , 0.619 )-- ( 1.3 , 0.597 )-- ( 1.4 , 0.576 )-- ( 1.5 , 0.556 )-- ( 1.6 , 0.538 )-- ( 1.7 , 0.52 )-- ( 1.8 , 0.503 )-- ( 1.9 , 0.487 )-- ( 2 , 0.471 )-- ( 2.1 , 0.456 )-- ( 2.2 , 0.442 )-- ( 2.3 , 0.429 )-- ( 2.4 , 0.416 )-- ( 2.5 , 0.404 )-- ( 2.6 , 0.392 )-- ( 2.7 , 0.38 )-- ( 2.8 , 0.37 )-- ( 2.9 , 0.359 )-- ( 3 , 0.349 )-- ( 3.1 , 0.339 )-- ( 3.2 , 0.33 )-- ( 3.3 , 0.321 )-- ( 3.4 , 0.313 )-- ( 3.5 , 0.304 )-- ( 3.6 , 0.296 )-- ( 3.7 , 0.289 )-- ( 3.8 , 0.281 )-- ( 3.9 , 0.274 )-- ( 4 , 0.267 )-- ( 4.1 , 0.261 )-- ( 4.2 , 0.254 )-- ( 4.3 , 0.248 )-- ( 4.4 , 0.242 )-- ( 4.5 , 0.236 )-- ( 4.6 , 0.23 )-- ( 4.7 , 0.225 )-- ( 4.8 , 0.219 )-- ( 4.9 , 0.214 )-- ( 5 , 0.209 )-- ( 5.1 , 0.204 )-- ( 5.2 , 0.199 )-- ( 5.3 , 0.195 )-- ( 5.4 , 0.19 )-- ( 5.5 , 0.186 )-- ( 5.6 , 0.182 )-- ( 5.7 , 0.177 )-- ( 5.8 , 0.173 )-- ( 5.9 , 0.17 )-- ( 6 , 0.166 )-- ( 6.1 , 0.162 )-- ( 6.2 , 0.158 )-- ( 6.3 , 0.155 )-- ( 6.4 , 0.151 )-- ( 6.5 , 0.148 )-- ( 6.6 , 0.145 )-- ( 6.7 , 0.142 )-- ( 6.8 , 0.138 )-- ( 6.9 , 0.135 )-- ( 7 , 0.132 )-- ( 7.1 , 0.13 )-- ( 7.2 , 0.127 )-- ( 7.3 , 0.124 )-- ( 7.4 , 0.121 )-- ( 7.5 , 0.119 )-- ( 7.6 , 0.116 )-- ( 7.7 , 0.114 )-- ( 7.8 , 0.111 )-- ( 7.9 , 0.109 )-- ( 8 , 0.106 )-- ( 8.1 , 0.104 )-- ( 8.2 , 0.102 )-- ( 8.3 , 0.1 )-- ( 8.4 , 0.098 )-- ( 8.5 , 0.096 )-- ( 8.6 , 0.093 )-- ( 8.7 , 0.091 )-- ( 8.8 , 0.09 )-- ( 8.9 , 0.088 )-- ( 9 , 0.086 )-- ( 9.1 , 0.084 )-- ( 9.2 , 0.082 )-- ( 9.3 , 0.08 )-- ( 9.4 , 0.079 )-- ( 9.5 , 0.077 )-- ( 9.6 , 0.075 )-- ( 9.7 , 0.074 )-- ( 9.8 , 0.072 )-- ( 9.9 , 0.071 )-- ( 10 , 0.069 )-- ( 10.1 , 0.068 )-- ( 10.2 , 0.066 )-- ( 10.3 , 0.065 )-- ( 10.4 , 0.064 )-- ( 10.5 , 0.062 )-- ( 10.6 , 0.061 )-- ( 10.7 , 0.06 )-- ( 10.8 , 0.058 )-- ( 10.9 , 0.057 )-- ( 11 , 0.056 )-- ( 11.1 , 0.055 )-- ( 11.2 , 0.054 )-- ( 11.3 , 0.052 )-- ( 11.4 , 0.051 )-- ( 11.5 , 0.05 )-- ( 11.6 , 0.049 )-- ( 11.7 , 0.048 )-- ( 11.8 , 0.047 )-- ( 11.9 , 0.046 )-- ( 12 , 0.045 )-- ( 12.1 , 0.044 )-- ( 12.2 , 0.043 )-- ( 12.3 , 0.042 )-- ( 12.4 , 0.042 )-- ( 12.5 , 0.041 )-- ( 12.6 , 0.04 )-- ( 12.7 , 0.039 )-- ( 12.8 , 0.038 )-- ( 12.9 , 0.037 )-- ( 13 , 0.037 )-- ( 13.1 , 0.036 )-- ( 13.2 , 0.035 )-- ( 13.3 , 0.034 )-- ( 13.4 , 0.034 )-- ( 13.5 , 0.033 )-- ( 13.6 , 0.032 )-- ( 13.7 , 0.032 )-- ( 13.8 , 0.031 )-- ( 13.9 , 0.03 )-- ( 14 , 0.03 )-- ( 14.1 , 0.029 )-- ( 14.2 , 0.028 )-- ( 14.3 , 0.028 )-- ( 14.4 , 0.027 )-- ( 14.5 , 0.027 )-- ( 14.6 , 0.026 )-- ( 14.7 , 0.026 )-- ( 14.8 , 0.025 )-- ( 14.9 , 0.024 )-- ( 15 , 0.024 )-- ( 15.1 , 0.023 )-- ( 15.2 , 0.023 )-- ( 15.3 , 0.022 )-- ( 15.4 , 0.022 )-- ( 15.5 , 0.022 )-- ( 15.6 , 0.021 )-- ( 15.7 , 0.021 )-- ( 15.8 , 0.02 )-- ( 15.9 , 0.02 )-- ( 16 , 0.019 )-- ( 16.1 , 0.019 )-- ( 16.2 , 0.019 )-- ( 16.3 , 0.018 )-- ( 16.4 , 0.018 )-- ( 16.5 , 0.017 )-- ( 16.6 , 0.017 )-- ( 16.7 , 0.017 )-- ( 16.8 , 0.016 )-- ( 16.9 , 0.016 )-- ( 17 , 0.016 )-- ( 17.1 , 0.015 )-- ( 17.2 , 0.015 )-- ( 17.3 , 0.015 )-- ( 17.4 , 0.014 )-- ( 17.5 , 0.014 )-- ( 17.6 , 0.014 )-- ( 17.7 , 0.014 )-- ( 17.8 , 0.013 )-- ( 17.9 , 0.013 )-- ( 18 , 0.013 )-- ( 18.1 , 0.012 )-- ( 18.2 , 0.012 )-- ( 18.3 , 0.012 )-- ( 18.4 , 0.012 )-- ( 18.5 , 0.011 )-- ( 18.6 , 0.011 )-- ( 18.7 , 0.011 )-- ( 18.8 , 0.011 )-- ( 18.9 , 0.01 )-- ( 19 , 0.01 )-- ( 19.1 , 0.01 )-- ( 19.2 , 0.01 )-- ( 19.3 , 0.01 )-- ( 19.4 , 0.009 )-- ( 19.5 , 0.009 )-- ( 19.6 , 0.009 )-- ( 19.7 , 0.009 )-- ( 19.8 , 0.009 )-- ( 19.9 , 0.008 )-- ( 20 , 0.008 ); 
 \end{tikzpicture}
 \end{subfigure}
 
 \begin{tikzpicture}
    \begin{customlegend}
    [legend entries={ $\theta_{0}$,$\hat{\theta}_{n,1}$,$\hat{\theta}_{n,2}$,$\hat{\theta}_{n,3}$,$\hat{\theta}_{n,4}$},legend columns=-1,legend style={/tikz/every even column/.append style={column sep=0.8cm}}]   
    \addlegendimage{black} 
    \addlegendimage{red,densely dotted}    
    \addlegendimage{blue,densely dashed}    
    \addlegendimage{orange,dashdotted}    
    \addlegendimage{olive, loosely dashed}  
    \end{customlegend}
\end{tikzpicture}

\caption{Tail distribution of HE customer patience compared with the estimated counterparts. The true values of the parameters $(\alpha_1,\alpha_2,\beta_1,\beta_2)$ are given by $(0.7,0.3,0.25,1)$. The fitted distributions $H_{\hat{\theta}_n^{(i)}}(t)$, with $i=1,2,3,4$, are based on independent samples of $n=10\,000$ observations with the HE MLE for $p=2$. The MLE parameters of the fitted functions are: $\hat{\theta}_{n,1}=( 0.529,0.471,0.193, 0.736),$ $\hat{\theta}_{n,2}=( 0.884, 0.116,0.297,1.969)$, $\hat{\theta}_{n,3}=( 0.75,0.25,0.269,1.031)$, and $\hat{\theta}_{n,4}=(0.573,0.427,0.212,0.746)$. The arrival rate in the example is $\lambda=1$, and the service requirements are Gamma distributed with parameters $(3,2)$.} \label{fig:H_MLE_fit}
\end{figure}

{In the remainder of this subsection we further study the performance of the HE MLE. Our experiments lead to the conclusion that HE MLE provides a good fit even in cases where the true distribution is {\it not} HE (i.e., in misspecified scenarios).}
{The examples mainly focus on {fitting the target distribution by a (conventional) HE distribution, rather than a GHE distribution, by computing the MLE}. 
The motivation behind this choice is that, when working with a GHE distribution, there is the additional complication that the estimated parameters may not yield a proper cdf. More specifically, the space of valid parameters is hard to characterize, let alone to be coded in terms of constraints of an optimization problem. Nevertheless, we would still like to {exploit} the extra versatility that the GHE class offers, so as to improve the fit (when compared to the HE MLE). This is especially {relevant for scenarios in which the model is misspecified, bearing in mind that, as mentioned earlier, in principle any distribution can be arbitrarily accurately approximated by a GHE distribution.}

With the above considerations in mind, we implemented the following (seemingly na\"{\i}ve) heuristic model selection method. We generate various {\it random} weight vectors ${\boldsymbol \alpha}$ (equipped with a random size $p\in{\mathbb N}$). Then, for each of them, we optimize the likelihood function over ${\boldsymbol \beta}$ using only step (a) above. The best model is then chosen by comparing the various combinations of ${\boldsymbol \alpha}$ and $p$ relying on the Akaike Information Criterion (AIC). The AIC encompasses both the log-likelihood and, in order to avoid overfitting, a penalty for the number of parameters \cite{BUR}. Because step (a) can be performed highly efficiently, a main advantage of this heuristic is that it works very fast, even for bigger values of $p$, thus providing us with a technique to fit general continuous distributions. Recall from \cite{BH1986} that if we order the components of the weight vector (i.e., $\alpha_1>\alpha_2>\ldots>\alpha_p$), then a sufficient condition for $({\boldsymbol \alpha},{\boldsymbol \beta})$ to yield a cdf is that $\alpha_1>0$ and $\sum_{i=1}^p \alpha_i \beta_i>0$; we use this principle to select feasible solutions produced by the above heuristic.}

Figures \ref{fig:H_MLE_hyper}-\ref{fig:H_gamma} illustrate the estimated survival function of using the HE MLE with varying weights for three examples of patience-level distributions. In the first example the patience-level distribution is indeed HE with $p=4$, whereas the second and third example are misspecified (corresponding to a lognormal and Gamma patience-level distribution, respectively). For each distribution the MLE was computed using $p=1,$ $2,$ and $4$ (where $p=3$ has been left out because it is barely distinguishable from $p=4$).
\begin{itemize}
\item[$\circ$]
As was the case in the setting of Figure \ref{fig:H_MLE_fit}, the experiments corresponding to the HE patience-level distribution show that, even for $p=4$ and as many as $n=10\,000$ waiting-time observations, the MLE does not accurately capture the true parameters. Nevertheless, the fit in terms of the cdf, as displayed in Figure \ref{fig:H_MLE_hyper}, is remarkably good. We in addition performed the GHE fitting heuristic, which also provide a highly accurate fit and is considerably faster to compute than the MLE for $p=4$. 
\item[$\circ$]
The case of lognormal patience, as illustrated by Figure \ref{fig:H_MLE_lognormal}, shows that the fit is quite good for all estimators except for $p=1$ (i.e., an exponential distribution). 
\item[$\circ$]Figure \ref{fig:H_gamma} presents the fitted distributions for Gamma patience. The fit of the HE MLE is decent, but, importantly, in this case the GHE heuristic performs considerably better than the HE MLE (for $p=1,2,4$). The GHE heuristic selects a model with $p=10$ weights.
\end{itemize}

\begin{figure}[h]
\centering
\begin{subfigure}{.48\linewidth}
\begin{tikzpicture}[xscale=0.35,yscale=3.5]
 \def\xmin{0}
 \def\xmax{20}
 \def\ymin{0}
 \def\ymax{1.05}
 \draw[->] (\xmin,\ymin) -- (\xmax,\ymin) node[right] {$t$} ;
 \draw[->] (\xmin,\ymin) -- (\xmin,\ymax)  ;
 \foreach \x in {0,5,10,20} {
 \node at (\x,\ymin-0.05) [below] {\x};
 \draw[-] (\x,\ymin) -- (\x,{\ymin-0.05});
 }
\foreach \y in {0,0.5,1} {
 \node at (\xmin,\y) [left] {\y};
 \draw[-] (\xmin,\y) -- ({\xmin-0.05},\y);
 }


\draw[black] ( 0 , 1 )-- ( 0.1 , 0.932 )-- ( 0.2 , 0.872 )-- ( 0.3 , 0.817 )-- ( 0.4 , 0.768 )-- ( 0.5 , 0.724 )-- ( 0.6 , 0.684 )-- ( 0.7 , 0.648 )-- ( 0.8 , 0.614 )-- ( 0.9 , 0.584 )-- ( 1 , 0.556 )-- ( 1.1 , 0.531 )-- ( 1.2 , 0.507 )-- ( 1.3 , 0.485 )-- ( 1.4 , 0.465 )-- ( 1.5 , 0.446 )-- ( 1.6 , 0.429 )-- ( 1.7 , 0.412 )-- ( 1.8 , 0.397 )-- ( 1.9 , 0.383 )-- ( 2 , 0.369 )-- ( 2.1 , 0.356 )-- ( 2.2 , 0.344 )-- ( 2.3 , 0.333 )-- ( 2.4 , 0.322 )-- ( 2.5 , 0.312 )-- ( 2.6 , 0.302 )-- ( 2.7 , 0.293 )-- ( 2.8 , 0.284 )-- ( 2.9 , 0.276 )-- ( 3 , 0.268 )-- ( 3.1 , 0.26 )-- ( 3.2 , 0.252 )-- ( 3.3 , 0.245 )-- ( 3.4 , 0.239 )-- ( 3.5 , 0.232 )-- ( 3.6 , 0.226 )-- ( 3.7 , 0.22 )-- ( 3.8 , 0.214 )-- ( 3.9 , 0.209 )-- ( 4 , 0.203 )-- ( 4.1 , 0.198 )-- ( 4.2 , 0.193 )-- ( 4.3 , 0.188 )-- ( 4.4 , 0.184 )-- ( 4.5 , 0.179 )-- ( 4.6 , 0.175 )-- ( 4.7 , 0.171 )-- ( 4.8 , 0.167 )-- ( 4.9 , 0.163 )-- ( 5 , 0.159 )-- ( 5.1 , 0.156 )-- ( 5.2 , 0.152 )-- ( 5.3 , 0.149 )-- ( 5.4 , 0.145 )-- ( 5.5 , 0.142 )-- ( 5.6 , 0.139 )-- ( 5.7 , 0.136 )-- ( 5.8 , 0.133 )-- ( 5.9 , 0.131 )-- ( 6 , 0.128 )-- ( 6.1 , 0.125 )-- ( 6.2 , 0.123 )-- ( 6.3 , 0.12 )-- ( 6.4 , 0.118 )-- ( 6.5 , 0.116 )-- ( 6.6 , 0.113 )-- ( 6.7 , 0.111 )-- ( 6.8 , 0.109 )-- ( 6.9 , 0.107 )-- ( 7 , 0.105 )-- ( 7.1 , 0.103 )-- ( 7.2 , 0.101 )-- ( 7.3 , 0.099 )-- ( 7.4 , 0.098 )-- ( 7.5 , 0.096 )-- ( 7.6 , 0.094 )-- ( 7.7 , 0.092 )-- ( 7.8 , 0.091 )-- ( 7.9 , 0.089 )-- ( 8 , 0.088 )-- ( 8.1 , 0.086 )-- ( 8.2 , 0.085 )-- ( 8.3 , 0.083 )-- ( 8.4 , 0.082 )-- ( 8.5 , 0.081 )-- ( 8.6 , 0.08 )-- ( 8.7 , 0.078 )-- ( 8.8 , 0.077 )-- ( 8.9 , 0.076 )-- ( 9 , 0.075 )-- ( 9.1 , 0.074 )-- ( 9.2 , 0.072 )-- ( 9.3 , 0.071 )-- ( 9.4 , 0.07 )-- ( 9.5 , 0.069 )-- ( 9.6 , 0.068 )-- ( 9.7 , 0.067 )-- ( 9.8 , 0.066 )-- ( 9.9 , 0.065 )-- ( 10 , 0.064 )-- ( 10.1 , 0.063 )-- ( 10.2 , 0.063 )-- ( 10.3 , 0.062 )-- ( 10.4 , 0.061 )-- ( 10.5 , 0.06 )-- ( 10.6 , 0.059 )-- ( 10.7 , 0.058 )-- ( 10.8 , 0.058 )-- ( 10.9 , 0.057 )-- ( 11 , 0.056 )-- ( 11.1 , 0.055 )-- ( 11.2 , 0.055 )-- ( 11.3 , 0.054 )-- ( 11.4 , 0.053 )-- ( 11.5 , 0.053 )-- ( 11.6 , 0.052 )-- ( 11.7 , 0.051 )-- ( 11.8 , 0.051 )-- ( 11.9 , 0.05 )-- ( 12 , 0.049 )-- ( 12.1 , 0.049 )-- ( 12.2 , 0.048 )-- ( 12.3 , 0.047 )-- ( 12.4 , 0.047 )-- ( 12.5 , 0.046 )-- ( 12.6 , 0.046 )-- ( 12.7 , 0.045 )-- ( 12.8 , 0.045 )-- ( 12.9 , 0.044 )-- ( 13 , 0.044 )-- ( 13.1 , 0.043 )-- ( 13.2 , 0.043 )-- ( 13.3 , 0.042 )-- ( 13.4 , 0.042 )-- ( 13.5 , 0.041 )-- ( 13.6 , 0.041 )-- ( 13.7 , 0.04 )-- ( 13.8 , 0.04 )-- ( 13.9 , 0.039 )-- ( 14 , 0.039 )-- ( 14.1 , 0.038 )-- ( 14.2 , 0.038 )-- ( 14.3 , 0.038 )-- ( 14.4 , 0.037 )-- ( 14.5 , 0.037 )-- ( 14.6 , 0.036 )-- ( 14.7 , 0.036 )-- ( 14.8 , 0.035 )-- ( 14.9 , 0.035 )-- ( 15 , 0.035 )-- ( 15.1 , 0.034 )-- ( 15.2 , 0.034 )-- ( 15.3 , 0.034 )-- ( 15.4 , 0.033 )-- ( 15.5 , 0.033 )-- ( 15.6 , 0.032 )-- ( 15.7 , 0.032 )-- ( 15.8 , 0.032 )-- ( 15.9 , 0.031 )-- ( 16 , 0.031 )-- ( 16.1 , 0.031 )-- ( 16.2 , 0.03 )-- ( 16.3 , 0.03 )-- ( 16.4 , 0.03 )-- ( 16.5 , 0.029 )-- ( 16.6 , 0.029 )-- ( 16.7 , 0.029 )-- ( 16.8 , 0.029 )-- ( 16.9 , 0.028 )-- ( 17 , 0.028 )-- ( 17.1 , 0.028 )-- ( 17.2 , 0.027 )-- ( 17.3 , 0.027 )-- ( 17.4 , 0.027 )-- ( 17.5 , 0.027 )-- ( 17.6 , 0.026 )-- ( 17.7 , 0.026 )-- ( 17.8 , 0.026 )-- ( 17.9 , 0.025 )-- ( 18 , 0.025 )-- ( 18.1 , 0.025 )-- ( 18.2 , 0.025 )-- ( 18.3 , 0.024 )-- ( 18.4 , 0.024 )-- ( 18.5 , 0.024 )-- ( 18.6 , 0.024 )-- ( 18.7 , 0.023 )-- ( 18.8 , 0.023 )-- ( 18.9 , 0.023 )-- ( 19 , 0.023 )-- ( 19.1 , 0.022 )-- ( 19.2 , 0.022 )-- ( 19.3 , 0.022 )-- ( 19.4 , 0.022 )-- ( 19.5 , 0.022 )-- ( 19.6 , 0.021 )-- ( 19.7 , 0.021 )-- ( 19.8 , 0.021 )-- ( 19.9 , 0.021 )-- ( 20 , 0.02 );

 \draw[red,dotted,thick] ( 0 , 1 )-- ( 0.1 , 0.958 )-- ( 0.2 , 0.917 )-- ( 0.3 , 0.878 )-- ( 0.4 , 0.841 )-- ( 0.5 , 0.806 )-- ( 0.6 , 0.772 )-- ( 0.7 , 0.739 )-- ( 0.8 , 0.708 )-- ( 0.9 , 0.678 )-- ( 1 , 0.649 )-- ( 1.1 , 0.622 )-- ( 1.2 , 0.595 )-- ( 1.3 , 0.57 )-- ( 1.4 , 0.546 )-- ( 1.5 , 0.523 )-- ( 1.6 , 0.501 )-- ( 1.7 , 0.48 )-- ( 1.8 , 0.459 )-- ( 1.9 , 0.44 )-- ( 2 , 0.421 )-- ( 2.1 , 0.404 )-- ( 2.2 , 0.386 )-- ( 2.3 , 0.37 )-- ( 2.4 , 0.354 )-- ( 2.5 , 0.34 )-- ( 2.6 , 0.325 )-- ( 2.7 , 0.311 )-- ( 2.8 , 0.298 )-- ( 2.9 , 0.286 )-- ( 3 , 0.274 )-- ( 3.1 , 0.262 )-- ( 3.2 , 0.251 )-- ( 3.3 , 0.24 )-- ( 3.4 , 0.23 )-- ( 3.5 , 0.22 )-- ( 3.6 , 0.211 )-- ( 3.7 , 0.202 )-- ( 3.8 , 0.194 )-- ( 3.9 , 0.185 )-- ( 4 , 0.178 )-- ( 4.1 , 0.17 )-- ( 4.2 , 0.163 )-- ( 4.3 , 0.156 )-- ( 4.4 , 0.149 )-- ( 4.5 , 0.143 )-- ( 4.6 , 0.137 )-- ( 4.7 , 0.131 )-- ( 4.8 , 0.126 )-- ( 4.9 , 0.12 )-- ( 5 , 0.115 )-- ( 5.1 , 0.11 )-- ( 5.2 , 0.106 )-- ( 5.3 , 0.101 )-- ( 5.4 , 0.097 )-- ( 5.5 , 0.093 )-- ( 5.6 , 0.089 )-- ( 5.7 , 0.085 )-- ( 5.8 , 0.082 )-- ( 5.9 , 0.078 )-- ( 6 , 0.075 )-- ( 6.1 , 0.072 )-- ( 6.2 , 0.069 )-- ( 6.3 , 0.066 )-- ( 6.4 , 0.063 )-- ( 6.5 , 0.06 )-- ( 6.6 , 0.058 )-- ( 6.7 , 0.055 )-- ( 6.8 , 0.053 )-- ( 6.9 , 0.051 )-- ( 7 , 0.049 )-- ( 7.1 , 0.047 )-- ( 7.2 , 0.045 )-- ( 7.3 , 0.043 )-- ( 7.4 , 0.041 )-- ( 7.5 , 0.039 )-- ( 7.6 , 0.037 )-- ( 7.7 , 0.036 )-- ( 7.8 , 0.034 )-- ( 7.9 , 0.033 )-- ( 8 , 0.032 )-- ( 8.1 , 0.03 )-- ( 8.2 , 0.029 )-- ( 8.3 , 0.028 )-- ( 8.4 , 0.027 )-- ( 8.5 , 0.025 )-- ( 8.6 , 0.024 )-- ( 8.7 , 0.023 )-- ( 8.8 , 0.022 )-- ( 8.9 , 0.021 )-- ( 9 , 0.02 )-- ( 9.1 , 0.02 )-- ( 9.2 , 0.019 )-- ( 9.3 , 0.018 )-- ( 9.4 , 0.017 )-- ( 9.5 , 0.016 )-- ( 9.6 , 0.016 )-- ( 9.7 , 0.015 )-- ( 9.8 , 0.014 )-- ( 9.9 , 0.014 )-- ( 10 , 0.013 )-- ( 10.1 , 0.013 )-- ( 10.2 , 0.012 )-- ( 10.3 , 0.012 )-- ( 10.4 , 0.011 )-- ( 10.5 , 0.011 )-- ( 10.6 , 0.01 )-- ( 10.7 , 0.01 )-- ( 10.8 , 0.009 )-- ( 10.9 , 0.009 )-- ( 11 , 0.009 )-- ( 11.1 , 0.008 )-- ( 11.2 , 0.008 )-- ( 11.3 , 0.008 )-- ( 11.4 , 0.007 )-- ( 11.5 , 0.007 )-- ( 11.6 , 0.007 )-- ( 11.7 , 0.006 )-- ( 11.8 , 0.006 )-- ( 11.9 , 0.006 )-- ( 12 , 0.006 )-- ( 12.1 , 0.005 )-- ( 12.2 , 0.005 )-- ( 12.3 , 0.005 )-- ( 12.4 , 0.005 )-- ( 12.5 , 0.005 )-- ( 12.6 , 0.004 )-- ( 12.7 , 0.004 )-- ( 12.8 , 0.004 )-- ( 12.9 , 0.004 )-- ( 13 , 0.004 )-- ( 13.1 , 0.003 )-- ( 13.2 , 0.003 )-- ( 13.3 , 0.003 )-- ( 13.4 , 0.003 )-- ( 13.5 , 0.003 )-- ( 13.6 , 0.003 )-- ( 13.7 , 0.003 )-- ( 13.8 , 0.003 )-- ( 13.9 , 0.002 )-- ( 14 , 0.002 )-- ( 14.1 , 0.002 )-- ( 14.2 , 0.002 )-- ( 14.3 , 0.002 )-- ( 14.4 , 0.002 )-- ( 14.5 , 0.002 )-- ( 14.6 , 0.002 )-- ( 14.7 , 0.002 )-- ( 14.8 , 0.002 )-- ( 14.9 , 0.002 )-- ( 15 , 0.002 )-- ( 15.1 , 0.001 )-- ( 15.2 , 0.001 )-- ( 15.3 , 0.001 )-- ( 15.4 , 0.001 )-- ( 15.5 , 0.001 )-- ( 15.6 , 0.001 )-- ( 15.7 , 0.001 )-- ( 15.8 , 0.001 )-- ( 15.9 , 0.001 )-- ( 16 , 0.001 )-- ( 16.1 , 0.001 )-- ( 16.2 , 0.001 )-- ( 16.3 , 0.001 )-- ( 16.4 , 0.001 )-- ( 16.5 , 0.001 )-- ( 16.6 , 0.001 )-- ( 16.7 , 0.001 )-- ( 16.8 , 0.001 )-- ( 16.9 , 0.001 )-- ( 17 , 0.001 )-- ( 17.1 , 0.001 )-- ( 17.2 , 0.001 )-- ( 17.3 , 0.001 )-- ( 17.4 , 0.001 )-- ( 17.5 , 0.001 )-- ( 17.6 , 0 )-- ( 17.7 , 0 )-- ( 17.8 , 0 )-- ( 17.9 , 0 )-- ( 18 , 0 )-- ( 18.1 , 0 )-- ( 18.2 , 0 )-- ( 18.3 , 0 )-- ( 18.4 , 0 )-- ( 18.5 , 0 )-- ( 18.6 , 0 )-- ( 18.7 , 0 )-- ( 18.8 , 0 )-- ( 18.9 , 0 )-- ( 19 , 0 )-- ( 19.1 , 0 )-- ( 19.2 , 0 )-- ( 19.3 , 0 )-- ( 19.4 , 0 )-- ( 19.5 , 0 )-- ( 19.6 , 0 )-- ( 19.7 , 0 )-- ( 19.8 , 0 )-- ( 19.9 , 0 )-- ( 20 , 0 );
 
\draw[blue,thick, densely dashed] ( 0 , 1 )-- ( 0.1 , 0.933 )-- ( 0.2 , 0.872 )-- ( 0.3 , 0.817 )-- ( 0.4 , 0.768 )-- ( 0.5 , 0.724 )-- ( 0.6 , 0.683 )-- ( 0.7 , 0.647 )-- ( 0.8 , 0.613 )-- ( 0.9 , 0.583 )-- ( 1 , 0.555 )-- ( 1.1 , 0.529 )-- ( 1.2 , 0.506 )-- ( 1.3 , 0.484 )-- ( 1.4 , 0.464 )-- ( 1.5 , 0.445 )-- ( 1.6 , 0.428 )-- ( 1.7 , 0.412 )-- ( 1.8 , 0.396 )-- ( 1.9 , 0.382 )-- ( 2 , 0.369 )-- ( 2.1 , 0.356 )-- ( 2.2 , 0.344 )-- ( 2.3 , 0.333 )-- ( 2.4 , 0.323 )-- ( 2.5 , 0.313 )-- ( 2.6 , 0.303 )-- ( 2.7 , 0.294 )-- ( 2.8 , 0.285 )-- ( 2.9 , 0.277 )-- ( 3 , 0.269 )-- ( 3.1 , 0.261 )-- ( 3.2 , 0.253 )-- ( 3.3 , 0.246 )-- ( 3.4 , 0.239 )-- ( 3.5 , 0.233 )-- ( 3.6 , 0.226 )-- ( 3.7 , 0.22 )-- ( 3.8 , 0.214 )-- ( 3.9 , 0.208 )-- ( 4 , 0.203 )-- ( 4.1 , 0.197 )-- ( 4.2 , 0.192 )-- ( 4.3 , 0.187 )-- ( 4.4 , 0.182 )-- ( 4.5 , 0.177 )-- ( 4.6 , 0.173 )-- ( 4.7 , 0.168 )-- ( 4.8 , 0.164 )-- ( 4.9 , 0.159 )-- ( 5 , 0.155 )-- ( 5.1 , 0.151 )-- ( 5.2 , 0.147 )-- ( 5.3 , 0.144 )-- ( 5.4 , 0.14 )-- ( 5.5 , 0.136 )-- ( 5.6 , 0.133 )-- ( 5.7 , 0.129 )-- ( 5.8 , 0.126 )-- ( 5.9 , 0.123 )-- ( 6 , 0.12 )-- ( 6.1 , 0.116 )-- ( 6.2 , 0.113 )-- ( 6.3 , 0.111 )-- ( 6.4 , 0.108 )-- ( 6.5 , 0.105 )-- ( 6.6 , 0.102 )-- ( 6.7 , 0.1 )-- ( 6.8 , 0.097 )-- ( 6.9 , 0.095 )-- ( 7 , 0.092 )-- ( 7.1 , 0.09 )-- ( 7.2 , 0.088 )-- ( 7.3 , 0.085 )-- ( 7.4 , 0.083 )-- ( 7.5 , 0.081 )-- ( 7.6 , 0.079 )-- ( 7.7 , 0.077 )-- ( 7.8 , 0.075 )-- ( 7.9 , 0.073 )-- ( 8 , 0.071 )-- ( 8.1 , 0.069 )-- ( 8.2 , 0.068 )-- ( 8.3 , 0.066 )-- ( 8.4 , 0.064 )-- ( 8.5 , 0.062 )-- ( 8.6 , 0.061 )-- ( 8.7 , 0.059 )-- ( 8.8 , 0.058 )-- ( 8.9 , 0.056 )-- ( 9 , 0.055 )-- ( 9.1 , 0.053 )-- ( 9.2 , 0.052 )-- ( 9.3 , 0.051 )-- ( 9.4 , 0.049 )-- ( 9.5 , 0.048 )-- ( 9.6 , 0.047 )-- ( 9.7 , 0.046 )-- ( 9.8 , 0.045 )-- ( 9.9 , 0.043 )-- ( 10 , 0.042 )-- ( 10.1 , 0.041 )-- ( 10.2 , 0.04 )-- ( 10.3 , 0.039 )-- ( 10.4 , 0.038 )-- ( 10.5 , 0.037 )-- ( 10.6 , 0.036 )-- ( 10.7 , 0.035 )-- ( 10.8 , 0.034 )-- ( 10.9 , 0.034 )-- ( 11 , 0.033 )-- ( 11.1 , 0.032 )-- ( 11.2 , 0.031 )-- ( 11.3 , 0.03 )-- ( 11.4 , 0.029 )-- ( 11.5 , 0.029 )-- ( 11.6 , 0.028 )-- ( 11.7 , 0.027 )-- ( 11.8 , 0.027 )-- ( 11.9 , 0.026 )-- ( 12 , 0.025 )-- ( 12.1 , 0.025 )-- ( 12.2 , 0.024 )-- ( 12.3 , 0.023 )-- ( 12.4 , 0.023 )-- ( 12.5 , 0.022 )-- ( 12.6 , 0.022 )-- ( 12.7 , 0.021 )-- ( 12.8 , 0.021 )-- ( 12.9 , 0.02 )-- ( 13 , 0.019 )-- ( 13.1 , 0.019 )-- ( 13.2 , 0.018 )-- ( 13.3 , 0.018 )-- ( 13.4 , 0.018 )-- ( 13.5 , 0.017 )-- ( 13.6 , 0.017 )-- ( 13.7 , 0.016 )-- ( 13.8 , 0.016 )-- ( 13.9 , 0.015 )-- ( 14 , 0.015 )-- ( 14.1 , 0.015 )-- ( 14.2 , 0.014 )-- ( 14.3 , 0.014 )-- ( 14.4 , 0.014 )-- ( 14.5 , 0.013 )-- ( 14.6 , 0.013 )-- ( 14.7 , 0.013 )-- ( 14.8 , 0.012 )-- ( 14.9 , 0.012 )-- ( 15 , 0.012 )-- ( 15.1 , 0.011 )-- ( 15.2 , 0.011 )-- ( 15.3 , 0.011 )-- ( 15.4 , 0.01 )-- ( 15.5 , 0.01 )-- ( 15.6 , 0.01 )-- ( 15.7 , 0.01 )-- ( 15.8 , 0.009 )-- ( 15.9 , 0.009 )-- ( 16 , 0.009 )-- ( 16.1 , 0.009 )-- ( 16.2 , 0.009 )-- ( 16.3 , 0.008 )-- ( 16.4 , 0.008 )-- ( 16.5 , 0.008 )-- ( 16.6 , 0.008 )-- ( 16.7 , 0.007 )-- ( 16.8 , 0.007 )-- ( 16.9 , 0.007 )-- ( 17 , 0.007 )-- ( 17.1 , 0.007 )-- ( 17.2 , 0.007 )-- ( 17.3 , 0.006 )-- ( 17.4 , 0.006 )-- ( 17.5 , 0.006 )-- ( 17.6 , 0.006 )-- ( 17.7 , 0.006 )-- ( 17.8 , 0.006 )-- ( 17.9 , 0.005 )-- ( 18 , 0.005 )-- ( 18.1 , 0.005 )-- ( 18.2 , 0.005 )-- ( 18.3 , 0.005 )-- ( 18.4 , 0.005 )-- ( 18.5 , 0.005 )-- ( 18.6 , 0.005 )-- ( 18.7 , 0.004 )-- ( 18.8 , 0.004 )-- ( 18.9 , 0.004 )-- ( 19 , 0.004 )-- ( 19.1 , 0.004 )-- ( 19.2 , 0.004 )-- ( 19.3 , 0.004 )-- ( 19.4 , 0.004 )-- ( 19.5 , 0.004 )-- ( 19.6 , 0.004 )-- ( 19.7 , 0.003 )-- ( 19.8 , 0.003 )-- ( 19.9 , 0.003 )-- ( 20 , 0.003 );
 
\draw[orange,thick,dashdotted] ( 0 , 1 )-- ( 0.1 , 0.93 )-- ( 0.2 , 0.869 )-- ( 0.3 , 0.814 )-- ( 0.4 , 0.765 )-- ( 0.5 , 0.72 )-- ( 0.6 , 0.681 )-- ( 0.7 , 0.645 )-- ( 0.8 , 0.612 )-- ( 0.9 , 0.582 )-- ( 1 , 0.555 )-- ( 1.1 , 0.53 )-- ( 1.2 , 0.506 )-- ( 1.3 , 0.485 )-- ( 1.4 , 0.465 )-- ( 1.5 , 0.446 )-- ( 1.6 , 0.429 )-- ( 1.7 , 0.413 )-- ( 1.8 , 0.398 )-- ( 1.9 , 0.383 )-- ( 2 , 0.37 )-- ( 2.1 , 0.357 )-- ( 2.2 , 0.345 )-- ( 2.3 , 0.334 )-- ( 2.4 , 0.323 )-- ( 2.5 , 0.313 )-- ( 2.6 , 0.303 )-- ( 2.7 , 0.293 )-- ( 2.8 , 0.285 )-- ( 2.9 , 0.276 )-- ( 3 , 0.268 )-- ( 3.1 , 0.26 )-- ( 3.2 , 0.253 )-- ( 3.3 , 0.245 )-- ( 3.4 , 0.238 )-- ( 3.5 , 0.232 )-- ( 3.6 , 0.226 )-- ( 3.7 , 0.219 )-- ( 3.8 , 0.214 )-- ( 3.9 , 0.208 )-- ( 4 , 0.202 )-- ( 4.1 , 0.197 )-- ( 4.2 , 0.192 )-- ( 4.3 , 0.187 )-- ( 4.4 , 0.182 )-- ( 4.5 , 0.178 )-- ( 4.6 , 0.173 )-- ( 4.7 , 0.169 )-- ( 4.8 , 0.165 )-- ( 4.9 , 0.161 )-- ( 5 , 0.157 )-- ( 5.1 , 0.154 )-- ( 5.2 , 0.15 )-- ( 5.3 , 0.146 )-- ( 5.4 , 0.143 )-- ( 5.5 , 0.14 )-- ( 5.6 , 0.136 )-- ( 5.7 , 0.133 )-- ( 5.8 , 0.13 )-- ( 5.9 , 0.127 )-- ( 6 , 0.125 )-- ( 6.1 , 0.122 )-- ( 6.2 , 0.119 )-- ( 6.3 , 0.116 )-- ( 6.4 , 0.114 )-- ( 6.5 , 0.111 )-- ( 6.6 , 0.109 )-- ( 6.7 , 0.107 )-- ( 6.8 , 0.104 )-- ( 6.9 , 0.102 )-- ( 7 , 0.1 )-- ( 7.1 , 0.098 )-- ( 7.2 , 0.096 )-- ( 7.3 , 0.094 )-- ( 7.4 , 0.092 )-- ( 7.5 , 0.09 )-- ( 7.6 , 0.088 )-- ( 7.7 , 0.086 )-- ( 7.8 , 0.085 )-- ( 7.9 , 0.083 )-- ( 8 , 0.081 )-- ( 8.1 , 0.079 )-- ( 8.2 , 0.078 )-- ( 8.3 , 0.076 )-- ( 8.4 , 0.075 )-- ( 8.5 , 0.073 )-- ( 8.6 , 0.072 )-- ( 8.7 , 0.07 )-- ( 8.8 , 0.069 )-- ( 8.9 , 0.068 )-- ( 9 , 0.066 )-- ( 9.1 , 0.065 )-- ( 9.2 , 0.064 )-- ( 9.3 , 0.062 )-- ( 9.4 , 0.061 )-- ( 9.5 , 0.06 )-- ( 9.6 , 0.059 )-- ( 9.7 , 0.058 )-- ( 9.8 , 0.057 )-- ( 9.9 , 0.055 )-- ( 10 , 0.054 )-- ( 10.1 , 0.053 )-- ( 10.2 , 0.052 )-- ( 10.3 , 0.051 )-- ( 10.4 , 0.05 )-- ( 10.5 , 0.049 )-- ( 10.6 , 0.048 )-- ( 10.7 , 0.047 )-- ( 10.8 , 0.046 )-- ( 10.9 , 0.046 )-- ( 11 , 0.045 )-- ( 11.1 , 0.044 )-- ( 11.2 , 0.043 )-- ( 11.3 , 0.042 )-- ( 11.4 , 0.041 )-- ( 11.5 , 0.041 )-- ( 11.6 , 0.04 )-- ( 11.7 , 0.039 )-- ( 11.8 , 0.038 )-- ( 11.9 , 0.038 )-- ( 12 , 0.037 )-- ( 12.1 , 0.036 )-- ( 12.2 , 0.035 )-- ( 12.3 , 0.035 )-- ( 12.4 , 0.034 )-- ( 12.5 , 0.033 )-- ( 12.6 , 0.033 )-- ( 12.7 , 0.032 )-- ( 12.8 , 0.032 )-- ( 12.9 , 0.031 )-- ( 13 , 0.03 )-- ( 13.1 , 0.03 )-- ( 13.2 , 0.029 )-- ( 13.3 , 0.029 )-- ( 13.4 , 0.028 )-- ( 13.5 , 0.028 )-- ( 13.6 , 0.027 )-- ( 13.7 , 0.027 )-- ( 13.8 , 0.026 )-- ( 13.9 , 0.026 )-- ( 14 , 0.025 )-- ( 14.1 , 0.025 )-- ( 14.2 , 0.024 )-- ( 14.3 , 0.024 )-- ( 14.4 , 0.023 )-- ( 14.5 , 0.023 )-- ( 14.6 , 0.022 )-- ( 14.7 , 0.022 )-- ( 14.8 , 0.022 )-- ( 14.9 , 0.021 )-- ( 15 , 0.021 )-- ( 15.1 , 0.02 )-- ( 15.2 , 0.02 )-- ( 15.3 , 0.02 )-- ( 15.4 , 0.019 )-- ( 15.5 , 0.019 )-- ( 15.6 , 0.019 )-- ( 15.7 , 0.018 )-- ( 15.8 , 0.018 )-- ( 15.9 , 0.018 )-- ( 16 , 0.017 )-- ( 16.1 , 0.017 )-- ( 16.2 , 0.017 )-- ( 16.3 , 0.016 )-- ( 16.4 , 0.016 )-- ( 16.5 , 0.016 )-- ( 16.6 , 0.015 )-- ( 16.7 , 0.015 )-- ( 16.8 , 0.015 )-- ( 16.9 , 0.015 )-- ( 17 , 0.014 )-- ( 17.1 , 0.014 )-- ( 17.2 , 0.014 )-- ( 17.3 , 0.013 )-- ( 17.4 , 0.013 )-- ( 17.5 , 0.013 )-- ( 17.6 , 0.013 )-- ( 17.7 , 0.012 )-- ( 17.8 , 0.012 )-- ( 17.9 , 0.012 )-- ( 18 , 0.012 )-- ( 18.1 , 0.012 )-- ( 18.2 , 0.011 )-- ( 18.3 , 0.011 )-- ( 18.4 , 0.011 )-- ( 18.5 , 0.011 )-- ( 18.6 , 0.011 )-- ( 18.7 , 0.01 )-- ( 18.8 , 0.01 )-- ( 18.9 , 0.01 )-- ( 19 , 0.01 )-- ( 19.1 , 0.01 )-- ( 19.2 , 0.009 )-- ( 19.3 , 0.009 )-- ( 19.4 , 0.009 )-- ( 19.5 , 0.009 )-- ( 19.6 , 0.009 )-- ( 19.7 , 0.009 )-- ( 19.8 , 0.008 )-- ( 19.9 , 0.008 )-- ( 20 , 0.008 );

\draw[olive,loosely dashed,thick] ( 0 , 1 )-- ( 0.1 , 0.921 )-- ( 0.2 , 0.853 )-- ( 0.3 , 0.796 )-- ( 0.4 , 0.747 )-- ( 0.5 , 0.704 )-- ( 0.6 , 0.667 )-- ( 0.7 , 0.634 )-- ( 0.8 , 0.605 )-- ( 0.9 , 0.578 )-- ( 1 , 0.554 )-- ( 1.1 , 0.532 )-- ( 1.2 , 0.511 )-- ( 1.3 , 0.492 )-- ( 1.4 , 0.474 )-- ( 1.5 , 0.457 )-- ( 1.6 , 0.441 )-- ( 1.7 , 0.426 )-- ( 1.8 , 0.411 )-- ( 1.9 , 0.397 )-- ( 2 , 0.384 )-- ( 2.1 , 0.371 )-- ( 2.2 , 0.359 )-- ( 2.3 , 0.347 )-- ( 2.4 , 0.335 )-- ( 2.5 , 0.324 )-- ( 2.6 , 0.313 )-- ( 2.7 , 0.303 )-- ( 2.8 , 0.292 )-- ( 2.9 , 0.283 )-- ( 3 , 0.273 )-- ( 3.1 , 0.264 )-- ( 3.2 , 0.255 )-- ( 3.3 , 0.246 )-- ( 3.4 , 0.238 )-- ( 3.5 , 0.23 )-- ( 3.6 , 0.222 )-- ( 3.7 , 0.215 )-- ( 3.8 , 0.207 )-- ( 3.9 , 0.2 )-- ( 4 , 0.194 )-- ( 4.1 , 0.187 )-- ( 4.2 , 0.181 )-- ( 4.3 , 0.174 )-- ( 4.4 , 0.169 )-- ( 4.5 , 0.163 )-- ( 4.6 , 0.157 )-- ( 4.7 , 0.152 )-- ( 4.8 , 0.147 )-- ( 4.9 , 0.142 )-- ( 5 , 0.137 )-- ( 5.1 , 0.132 )-- ( 5.2 , 0.128 )-- ( 5.3 , 0.123 )-- ( 5.4 , 0.119 )-- ( 5.5 , 0.115 )-- ( 5.6 , 0.111 )-- ( 5.7 , 0.107 )-- ( 5.8 , 0.104 )-- ( 5.9 , 0.1 )-- ( 6 , 0.097 )-- ( 6.1 , 0.093 )-- ( 6.2 , 0.09 )-- ( 6.3 , 0.087 )-- ( 6.4 , 0.084 )-- ( 6.5 , 0.081 )-- ( 6.6 , 0.078 )-- ( 6.7 , 0.076 )-- ( 6.8 , 0.073 )-- ( 6.9 , 0.071 )-- ( 7 , 0.068 )-- ( 7.1 , 0.066 )-- ( 7.2 , 0.064 )-- ( 7.3 , 0.061 )-- ( 7.4 , 0.059 )-- ( 7.5 , 0.057 )-- ( 7.6 , 0.055 )-- ( 7.7 , 0.053 )-- ( 7.8 , 0.052 )-- ( 7.9 , 0.05 )-- ( 8 , 0.048 )-- ( 8.1 , 0.046 )-- ( 8.2 , 0.045 )-- ( 8.3 , 0.043 )-- ( 8.4 , 0.042 )-- ( 8.5 , 0.04 )-- ( 8.6 , 0.039 )-- ( 8.7 , 0.038 )-- ( 8.8 , 0.036 )-- ( 8.9 , 0.035 )-- ( 9 , 0.034 )-- ( 9.1 , 0.033 )-- ( 9.2 , 0.032 )-- ( 9.3 , 0.031 )-- ( 9.4 , 0.029 )-- ( 9.5 , 0.028 )-- ( 9.6 , 0.027 )-- ( 9.7 , 0.027 )-- ( 9.8 , 0.026 )-- ( 9.9 , 0.025 )-- ( 10 , 0.024 )-- ( 10.1 , 0.023 )-- ( 10.2 , 0.022 )-- ( 10.3 , 0.022 )-- ( 10.4 , 0.021 )-- ( 10.5 , 0.02 )-- ( 10.6 , 0.019 )-- ( 10.7 , 0.019 )-- ( 10.8 , 0.018 )-- ( 10.9 , 0.017 )-- ( 11 , 0.017 )-- ( 11.1 , 0.016 )-- ( 11.2 , 0.016 )-- ( 11.3 , 0.015 )-- ( 11.4 , 0.015 )-- ( 11.5 , 0.014 )-- ( 11.6 , 0.014 )-- ( 11.7 , 0.013 )-- ( 11.8 , 0.013 )-- ( 11.9 , 0.012 )-- ( 12 , 0.012 )-- ( 12.1 , 0.011 )-- ( 12.2 , 0.011 )-- ( 12.3 , 0.011 )-- ( 12.4 , 0.01 )-- ( 12.5 , 0.01 )-- ( 12.6 , 0.01 )-- ( 12.7 , 0.009 )-- ( 12.8 , 0.009 )-- ( 12.9 , 0.009 )-- ( 13 , 0.008 )-- ( 13.1 , 0.008 )-- ( 13.2 , 0.008 )-- ( 13.3 , 0.008 )-- ( 13.4 , 0.007 )-- ( 13.5 , 0.007 )-- ( 13.6 , 0.007 )-- ( 13.7 , 0.007 )-- ( 13.8 , 0.006 )-- ( 13.9 , 0.006 )-- ( 14 , 0.006 )-- ( 14.1 , 0.006 )-- ( 14.2 , 0.006 )-- ( 14.3 , 0.005 )-- ( 14.4 , 0.005 )-- ( 14.5 , 0.005 )-- ( 14.6 , 0.005 )-- ( 14.7 , 0.005 )-- ( 14.8 , 0.004 )-- ( 14.9 , 0.004 )-- ( 15 , 0.004 )-- ( 15.1 , 0.004 )-- ( 15.2 , 0.004 )-- ( 15.3 , 0.004 )-- ( 15.4 , 0.004 )-- ( 15.5 , 0.004 )-- ( 15.6 , 0.003 )-- ( 15.7 , 0.003 )-- ( 15.8 , 0.003 )-- ( 15.9 , 0.003 )-- ( 16 , 0.003 )-- ( 16.1 , 0.003 )-- ( 16.2 , 0.003 )-- ( 16.3 , 0.003 )-- ( 16.4 , 0.003 )-- ( 16.5 , 0.002 )-- ( 16.6 , 0.002 )-- ( 16.7 , 0.002 )-- ( 16.8 , 0.002 )-- ( 16.9 , 0.002 )-- ( 17 , 0.002 )-- ( 17.1 , 0.002 )-- ( 17.2 , 0.002 )-- ( 17.3 , 0.002 )-- ( 17.4 , 0.002 )-- ( 17.5 , 0.002 )-- ( 17.6 , 0.002 )-- ( 17.7 , 0.002 )-- ( 17.8 , 0.002 )-- ( 17.9 , 0.002 )-- ( 18 , 0.001 )-- ( 18.1 , 0.001 )-- ( 18.2 , 0.001 )-- ( 18.3 , 0.001 )-- ( 18.4 , 0.001 )-- ( 18.5 , 0.001 )-- ( 18.6 , 0.001 )-- ( 18.7 , 0.001 )-- ( 18.8 , 0.001 )-- ( 18.9 , 0.001 )-- ( 19 , 0.001 )-- ( 19.1 , 0.001 )-- ( 19.2 , 0.001 )-- ( 19.3 , 0.001 )-- ( 19.4 , 0.001 )-- ( 19.5 , 0.001 )-- ( 19.6 , 0.001 )-- ( 19.7 , 0.001 )-- ( 19.8 , 0.001 )-- ( 19.9 , 0.001 )-- ( 20 , 0.001 ); 
 \end{tikzpicture}
\caption{HE$(0.5,0.3,0.15,0.05,0.4,1.5,0.1,1)$}\label{fig:H_MLE_hyper}
\end{subfigure}
\begin{subfigure}{.48\linewidth}
\begin{tikzpicture}[xscale=0.35,yscale=3.5]
\def\xmin{0}
 \def\xmax{20}
 \def\ymin{0}
 \def\ymax{1.05}
 \draw[->] (\xmin,\ymin) -- (\xmax,\ymin) node[right] {$t$} ;
 \draw[->] (\xmin,\ymin) -- (\xmin,\ymax)  ;
 \foreach \x in {0,5,10,20} {
 \node at (\x,\ymin-0.05) [below] {\x};
 \draw[-] (\x,\ymin) -- (\x,{\ymin-0.05});
 }
\foreach \y in {0,0.5,1} {
 \node at (\xmin,\y) [left] {\y};
 \draw[-] (\xmin,\y) -- ({\xmin-0.05},\y);
 }


\draw[black] ( 0 , 1 )-- ( 0.1 , 0.997 )-- ( 0.2 , 0.982 )-- ( 0.3 , 0.956 )-- ( 0.4 , 0.921 )-- ( 0.5 , 0.883 )-- ( 0.6 , 0.844 )-- ( 0.7 , 0.804 )-- ( 0.8 , 0.765 )-- ( 0.9 , 0.727 )-- ( 1 , 0.691 )-- ( 1.1 , 0.657 )-- ( 1.2 , 0.625 )-- ( 1.3 , 0.594 )-- ( 1.4 , 0.565 )-- ( 1.5 , 0.538 )-- ( 1.6 , 0.512 )-- ( 1.7 , 0.488 )-- ( 1.8 , 0.465 )-- ( 1.9 , 0.443 )-- ( 2 , 0.423 )-- ( 2.1 , 0.404 )-- ( 2.2 , 0.386 )-- ( 2.3 , 0.369 )-- ( 2.4 , 0.353 )-- ( 2.5 , 0.338 )-- ( 2.6 , 0.324 )-- ( 2.7 , 0.31 )-- ( 2.8 , 0.298 )-- ( 2.9 , 0.285 )-- ( 3 , 0.274 )-- ( 3.1 , 0.263 )-- ( 3.2 , 0.253 )-- ( 3.3 , 0.243 )-- ( 3.4 , 0.234 )-- ( 3.5 , 0.225 )-- ( 3.6 , 0.217 )-- ( 3.7 , 0.209 )-- ( 3.8 , 0.201 )-- ( 3.9 , 0.194 )-- ( 4 , 0.187 )-- ( 4.1 , 0.181 )-- ( 4.2 , 0.174 )-- ( 4.3 , 0.168 )-- ( 4.4 , 0.163 )-- ( 4.5 , 0.157 )-- ( 4.6 , 0.152 )-- ( 4.7 , 0.147 )-- ( 4.8 , 0.142 )-- ( 4.9 , 0.138 )-- ( 5 , 0.133 )-- ( 5.1 , 0.129 )-- ( 5.2 , 0.125 )-- ( 5.3 , 0.121 )-- ( 5.4 , 0.118 )-- ( 5.5 , 0.114 )-- ( 5.6 , 0.11 )-- ( 5.7 , 0.107 )-- ( 5.8 , 0.104 )-- ( 5.9 , 0.101 )-- ( 6 , 0.098 )-- ( 6.1 , 0.095 )-- ( 6.2 , 0.092 )-- ( 6.3 , 0.09 )-- ( 6.4 , 0.087 )-- ( 6.5 , 0.085 )-- ( 6.6 , 0.082 )-- ( 6.7 , 0.08 )-- ( 6.8 , 0.078 )-- ( 6.9 , 0.076 )-- ( 7 , 0.074 )-- ( 7.1 , 0.072 )-- ( 7.2 , 0.07 )-- ( 7.3 , 0.068 )-- ( 7.4 , 0.066 )-- ( 7.5 , 0.065 )-- ( 7.6 , 0.063 )-- ( 7.7 , 0.061 )-- ( 7.8 , 0.06 )-- ( 7.9 , 0.058 )-- ( 8 , 0.057 )-- ( 8.1 , 0.055 )-- ( 8.2 , 0.054 )-- ( 8.3 , 0.053 )-- ( 8.4 , 0.051 )-- ( 8.5 , 0.05 )-- ( 8.6 , 0.049 )-- ( 8.7 , 0.048 )-- ( 8.8 , 0.047 )-- ( 8.9 , 0.046 )-- ( 9 , 0.045 )-- ( 9.1 , 0.044 )-- ( 9.2 , 0.043 )-- ( 9.3 , 0.042 )-- ( 9.4 , 0.041 )-- ( 9.5 , 0.04 )-- ( 9.6 , 0.039 )-- ( 9.7 , 0.038 )-- ( 9.8 , 0.037 )-- ( 9.9 , 0.036 )-- ( 10 , 0.036 )-- ( 10.1 , 0.035 )-- ( 10.2 , 0.034 )-- ( 10.3 , 0.033 )-- ( 10.4 , 0.033 )-- ( 10.5 , 0.032 )-- ( 10.6 , 0.031 )-- ( 10.7 , 0.031 )-- ( 10.8 , 0.03 )-- ( 10.9 , 0.029 )-- ( 11 , 0.029 )-- ( 11.1 , 0.028 )-- ( 11.2 , 0.028 )-- ( 11.3 , 0.027 )-- ( 11.4 , 0.027 )-- ( 11.5 , 0.026 )-- ( 11.6 , 0.025 )-- ( 11.7 , 0.025 )-- ( 11.8 , 0.025 )-- ( 11.9 , 0.024 )-- ( 12 , 0.024 )-- ( 12.1 , 0.023 )-- ( 12.2 , 0.023 )-- ( 12.3 , 0.022 )-- ( 12.4 , 0.022 )-- ( 12.5 , 0.021 )-- ( 12.6 , 0.021 )-- ( 12.7 , 0.021 )-- ( 12.8 , 0.02 )-- ( 12.9 , 0.02 )-- ( 13 , 0.019 )-- ( 13.1 , 0.019 )-- ( 13.2 , 0.019 )-- ( 13.3 , 0.018 )-- ( 13.4 , 0.018 )-- ( 13.5 , 0.018 )-- ( 13.6 , 0.017 )-- ( 13.7 , 0.017 )-- ( 13.8 , 0.017 )-- ( 13.9 , 0.017 )-- ( 14 , 0.016 )-- ( 14.1 , 0.016 )-- ( 14.2 , 0.016 )-- ( 14.3 , 0.015 )-- ( 14.4 , 0.015 )-- ( 14.5 , 0.015 )-- ( 14.6 , 0.015 )-- ( 14.7 , 0.014 )-- ( 14.8 , 0.014 )-- ( 14.9 , 0.014 )-- ( 15 , 0.014 )-- ( 15.1 , 0.013 )-- ( 15.2 , 0.013 )-- ( 15.3 , 0.013 )-- ( 15.4 , 0.013 )-- ( 15.5 , 0.013 )-- ( 15.6 , 0.012 )-- ( 15.7 , 0.012 )-- ( 15.8 , 0.012 )-- ( 15.9 , 0.012 )-- ( 16 , 0.012 )-- ( 16.1 , 0.011 )-- ( 16.2 , 0.011 )-- ( 16.3 , 0.011 )-- ( 16.4 , 0.011 )-- ( 16.5 , 0.011 )-- ( 16.6 , 0.011 )-- ( 16.7 , 0.01 )-- ( 16.8 , 0.01 )-- ( 16.9 , 0.01 )-- ( 17 , 0.01 )-- ( 17.1 , 0.01 )-- ( 17.2 , 0.01 )-- ( 17.3 , 0.009 )-- ( 17.4 , 0.009 )-- ( 17.5 , 0.009 )-- ( 17.6 , 0.009 )-- ( 17.7 , 0.009 )-- ( 17.8 , 0.009 )-- ( 17.9 , 0.009 )-- ( 18 , 0.008 )-- ( 18.1 , 0.008 )-- ( 18.2 , 0.008 )-- ( 18.3 , 0.008 )-- ( 18.4 , 0.008 )-- ( 18.5 , 0.008 )-- ( 18.6 , 0.008 )-- ( 18.7 , 0.008 )-- ( 18.8 , 0.008 )-- ( 18.9 , 0.007 )-- ( 19 , 0.007 )-- ( 19.1 , 0.007 )-- ( 19.2 , 0.007 )-- ( 19.3 , 0.007 )-- ( 19.4 , 0.007 )-- ( 19.5 , 0.007 )-- ( 19.6 , 0.007 )-- ( 19.7 , 0.007 )-- ( 19.8 , 0.007 )-- ( 19.9 , 0.006 )-- ( 20 , 0.006 ) ;

 \draw[red,dotted,thick]( 0 , 1 )-- ( 0.1 , 0.967 )-- ( 0.2 , 0.935 )-- ( 0.3 , 0.904 )-- ( 0.4 , 0.874 )-- ( 0.5 , 0.845 )-- ( 0.6 , 0.818 )-- ( 0.7 , 0.791 )-- ( 0.8 , 0.764 )-- ( 0.9 , 0.739 )-- ( 1 , 0.715 )-- ( 1.1 , 0.691 )-- ( 1.2 , 0.668 )-- ( 1.3 , 0.646 )-- ( 1.4 , 0.625 )-- ( 1.5 , 0.604 )-- ( 1.6 , 0.584 )-- ( 1.7 , 0.565 )-- ( 1.8 , 0.546 )-- ( 1.9 , 0.528 )-- ( 2 , 0.511 )-- ( 2.1 , 0.494 )-- ( 2.2 , 0.478 )-- ( 2.3 , 0.462 )-- ( 2.4 , 0.447 )-- ( 2.5 , 0.432 )-- ( 2.6 , 0.418 )-- ( 2.7 , 0.404 )-- ( 2.8 , 0.391 )-- ( 2.9 , 0.378 )-- ( 3 , 0.365 )-- ( 3.1 , 0.353 )-- ( 3.2 , 0.341 )-- ( 3.3 , 0.33 )-- ( 3.4 , 0.319 )-- ( 3.5 , 0.309 )-- ( 3.6 , 0.299 )-- ( 3.7 , 0.289 )-- ( 3.8 , 0.279 )-- ( 3.9 , 0.27 )-- ( 4 , 0.261 )-- ( 4.1 , 0.252 )-- ( 4.2 , 0.244 )-- ( 4.3 , 0.236 )-- ( 4.4 , 0.228 )-- ( 4.5 , 0.221 )-- ( 4.6 , 0.213 )-- ( 4.7 , 0.206 )-- ( 4.8 , 0.2 )-- ( 4.9 , 0.193 )-- ( 5 , 0.187 )-- ( 5.1 , 0.18 )-- ( 5.2 , 0.174 )-- ( 5.3 , 0.169 )-- ( 5.4 , 0.163 )-- ( 5.5 , 0.158 )-- ( 5.6 , 0.153 )-- ( 5.7 , 0.148 )-- ( 5.8 , 0.143 )-- ( 5.9 , 0.138 )-- ( 6 , 0.133 )-- ( 6.1 , 0.129 )-- ( 6.2 , 0.125 )-- ( 6.3 , 0.121 )-- ( 6.4 , 0.117 )-- ( 6.5 , 0.113 )-- ( 6.6 , 0.109 )-- ( 6.7 , 0.105 )-- ( 6.8 , 0.102 )-- ( 6.9 , 0.099 )-- ( 7 , 0.095 )-- ( 7.1 , 0.092 )-- ( 7.2 , 0.089 )-- ( 7.3 , 0.086 )-- ( 7.4 , 0.083 )-- ( 7.5 , 0.081 )-- ( 7.6 , 0.078 )-- ( 7.7 , 0.075 )-- ( 7.8 , 0.073 )-- ( 7.9 , 0.07 )-- ( 8 , 0.068 )-- ( 8.1 , 0.066 )-- ( 8.2 , 0.064 )-- ( 8.3 , 0.062 )-- ( 8.4 , 0.06 )-- ( 8.5 , 0.058 )-- ( 8.6 , 0.056 )-- ( 8.7 , 0.054 )-- ( 8.8 , 0.052 )-- ( 8.9 , 0.05 )-- ( 9 , 0.049 )-- ( 9.1 , 0.047 )-- ( 9.2 , 0.046 )-- ( 9.3 , 0.044 )-- ( 9.4 , 0.043 )-- ( 9.5 , 0.041 )-- ( 9.6 , 0.04 )-- ( 9.7 , 0.039 )-- ( 9.8 , 0.037 )-- ( 9.9 , 0.036 )-- ( 10 , 0.035 )-- ( 10.1 , 0.034 )-- ( 10.2 , 0.033 )-- ( 10.3 , 0.031 )-- ( 10.4 , 0.03 )-- ( 10.5 , 0.029 )-- ( 10.6 , 0.028 )-- ( 10.7 , 0.028 )-- ( 10.8 , 0.027 )-- ( 10.9 , 0.026 )-- ( 11 , 0.025 )-- ( 11.1 , 0.024 )-- ( 11.2 , 0.023 )-- ( 11.3 , 0.023 )-- ( 11.4 , 0.022 )-- ( 11.5 , 0.021 )-- ( 11.6 , 0.02 )-- ( 11.7 , 0.02 )-- ( 11.8 , 0.019 )-- ( 11.9 , 0.018 )-- ( 12 , 0.018 )-- ( 12.1 , 0.017 )-- ( 12.2 , 0.017 )-- ( 12.3 , 0.016 )-- ( 12.4 , 0.016 )-- ( 12.5 , 0.015 )-- ( 12.6 , 0.015 )-- ( 12.7 , 0.014 )-- ( 12.8 , 0.014 )-- ( 12.9 , 0.013 )-- ( 13 , 0.013 )-- ( 13.1 , 0.012 )-- ( 13.2 , 0.012 )-- ( 13.3 , 0.011 )-- ( 13.4 , 0.011 )-- ( 13.5 , 0.011 )-- ( 13.6 , 0.01 )-- ( 13.7 , 0.01 )-- ( 13.8 , 0.01 )-- ( 13.9 , 0.009 )-- ( 14 , 0.009 )-- ( 14.1 , 0.009 )-- ( 14.2 , 0.008 )-- ( 14.3 , 0.008 )-- ( 14.4 , 0.008 )-- ( 14.5 , 0.008 )-- ( 14.6 , 0.007 )-- ( 14.7 , 0.007 )-- ( 14.8 , 0.007 )-- ( 14.9 , 0.007 )-- ( 15 , 0.006 )-- ( 15.1 , 0.006 )-- ( 15.2 , 0.006 )-- ( 15.3 , 0.006 )-- ( 15.4 , 0.006 )-- ( 15.5 , 0.005 )-- ( 15.6 , 0.005 )-- ( 15.7 , 0.005 )-- ( 15.8 , 0.005 )-- ( 15.9 , 0.005 )-- ( 16 , 0.005 )-- ( 16.1 , 0.004 )-- ( 16.2 , 0.004 )-- ( 16.3 , 0.004 )-- ( 16.4 , 0.004 )-- ( 16.5 , 0.004 )-- ( 16.6 , 0.004 )-- ( 16.7 , 0.004 )-- ( 16.8 , 0.004 )-- ( 16.9 , 0.003 )-- ( 17 , 0.003 )-- ( 17.1 , 0.003 )-- ( 17.2 , 0.003 )-- ( 17.3 , 0.003 )-- ( 17.4 , 0.003 )-- ( 17.5 , 0.003 )-- ( 17.6 , 0.003 )-- ( 17.7 , 0.003 )-- ( 17.8 , 0.003 )-- ( 17.9 , 0.002 )-- ( 18 , 0.002 )-- ( 18.1 , 0.002 )-- ( 18.2 , 0.002 )-- ( 18.3 , 0.002 )-- ( 18.4 , 0.002 )-- ( 18.5 , 0.002 )-- ( 18.6 , 0.002 )-- ( 18.7 , 0.002 )-- ( 18.8 , 0.002 )-- ( 18.9 , 0.002 )-- ( 19 , 0.002 )-- ( 19.1 , 0.002 )-- ( 19.2 , 0.002 )-- ( 19.3 , 0.002 )-- ( 19.4 , 0.001 )-- ( 19.5 , 0.001 )-- ( 19.6 , 0.001 )-- ( 19.7 , 0.001 )-- ( 19.8 , 0.001 )-- ( 19.9 , 0.001 )-- ( 20 , 0.001 )  ;
 
\draw[blue,thick, densely dashed]  ( 0 , 1 )-- ( 0.1 , 0.95 )-- ( 0.2 , 0.903 )-- ( 0.3 , 0.859 )-- ( 0.4 , 0.817 )-- ( 0.5 , 0.778 )-- ( 0.6 , 0.742 )-- ( 0.7 , 0.707 )-- ( 0.8 , 0.675 )-- ( 0.9 , 0.644 )-- ( 1 , 0.616 )-- ( 1.1 , 0.589 )-- ( 1.2 , 0.563 )-- ( 1.3 , 0.539 )-- ( 1.4 , 0.516 )-- ( 1.5 , 0.494 )-- ( 1.6 , 0.474 )-- ( 1.7 , 0.455 )-- ( 1.8 , 0.437 )-- ( 1.9 , 0.419 )-- ( 2 , 0.403 )-- ( 2.1 , 0.387 )-- ( 2.2 , 0.373 )-- ( 2.3 , 0.358 )-- ( 2.4 , 0.345 )-- ( 2.5 , 0.332 )-- ( 2.6 , 0.32 )-- ( 2.7 , 0.309 )-- ( 2.8 , 0.298 )-- ( 2.9 , 0.287 )-- ( 3 , 0.277 )-- ( 3.1 , 0.268 )-- ( 3.2 , 0.259 )-- ( 3.3 , 0.25 )-- ( 3.4 , 0.242 )-- ( 3.5 , 0.234 )-- ( 3.6 , 0.226 )-- ( 3.7 , 0.219 )-- ( 3.8 , 0.212 )-- ( 3.9 , 0.205 )-- ( 4 , 0.198 )-- ( 4.1 , 0.192 )-- ( 4.2 , 0.186 )-- ( 4.3 , 0.18 )-- ( 4.4 , 0.175 )-- ( 4.5 , 0.17 )-- ( 4.6 , 0.165 )-- ( 4.7 , 0.16 )-- ( 4.8 , 0.155 )-- ( 4.9 , 0.15 )-- ( 5 , 0.146 )-- ( 5.1 , 0.142 )-- ( 5.2 , 0.137 )-- ( 5.3 , 0.133 )-- ( 5.4 , 0.13 )-- ( 5.5 , 0.126 )-- ( 5.6 , 0.122 )-- ( 5.7 , 0.119 )-- ( 5.8 , 0.115 )-- ( 5.9 , 0.112 )-- ( 6 , 0.109 )-- ( 6.1 , 0.106 )-- ( 6.2 , 0.103 )-- ( 6.3 , 0.1 )-- ( 6.4 , 0.097 )-- ( 6.5 , 0.095 )-- ( 6.6 , 0.092 )-- ( 6.7 , 0.09 )-- ( 6.8 , 0.087 )-- ( 6.9 , 0.085 )-- ( 7 , 0.082 )-- ( 7.1 , 0.08 )-- ( 7.2 , 0.078 )-- ( 7.3 , 0.076 )-- ( 7.4 , 0.074 )-- ( 7.5 , 0.072 )-- ( 7.6 , 0.07 )-- ( 7.7 , 0.068 )-- ( 7.8 , 0.066 )-- ( 7.9 , 0.065 )-- ( 8 , 0.063 )-- ( 8.1 , 0.061 )-- ( 8.2 , 0.06 )-- ( 8.3 , 0.058 )-- ( 8.4 , 0.056 )-- ( 8.5 , 0.055 )-- ( 8.6 , 0.053 )-- ( 8.7 , 0.052 )-- ( 8.8 , 0.051 )-- ( 8.9 , 0.049 )-- ( 9 , 0.048 )-- ( 9.1 , 0.047 )-- ( 9.2 , 0.046 )-- ( 9.3 , 0.044 )-- ( 9.4 , 0.043 )-- ( 9.5 , 0.042 )-- ( 9.6 , 0.041 )-- ( 9.7 , 0.04 )-- ( 9.8 , 0.039 )-- ( 9.9 , 0.038 )-- ( 10 , 0.037 )-- ( 10.1 , 0.036 )-- ( 10.2 , 0.035 )-- ( 10.3 , 0.034 )-- ( 10.4 , 0.033 )-- ( 10.5 , 0.032 )-- ( 10.6 , 0.031 )-- ( 10.7 , 0.031 )-- ( 10.8 , 0.03 )-- ( 10.9 , 0.029 )-- ( 11 , 0.028 )-- ( 11.1 , 0.028 )-- ( 11.2 , 0.027 )-- ( 11.3 , 0.026 )-- ( 11.4 , 0.025 )-- ( 11.5 , 0.025 )-- ( 11.6 , 0.024 )-- ( 11.7 , 0.024 )-- ( 11.8 , 0.023 )-- ( 11.9 , 0.022 )-- ( 12 , 0.022 )-- ( 12.1 , 0.021 )-- ( 12.2 , 0.021 )-- ( 12.3 , 0.02 )-- ( 12.4 , 0.02 )-- ( 12.5 , 0.019 )-- ( 12.6 , 0.019 )-- ( 12.7 , 0.018 )-- ( 12.8 , 0.018 )-- ( 12.9 , 0.017 )-- ( 13 , 0.017 )-- ( 13.1 , 0.016 )-- ( 13.2 , 0.016 )-- ( 13.3 , 0.015 )-- ( 13.4 , 0.015 )-- ( 13.5 , 0.015 )-- ( 13.6 , 0.014 )-- ( 13.7 , 0.014 )-- ( 13.8 , 0.014 )-- ( 13.9 , 0.013 )-- ( 14 , 0.013 )-- ( 14.1 , 0.013 )-- ( 14.2 , 0.012 )-- ( 14.3 , 0.012 )-- ( 14.4 , 0.012 )-- ( 14.5 , 0.011 )-- ( 14.6 , 0.011 )-- ( 14.7 , 0.011 )-- ( 14.8 , 0.01 )-- ( 14.9 , 0.01 )-- ( 15 , 0.01 )-- ( 15.1 , 0.01 )-- ( 15.2 , 0.009 )-- ( 15.3 , 0.009 )-- ( 15.4 , 0.009 )-- ( 15.5 , 0.009 )-- ( 15.6 , 0.008 )-- ( 15.7 , 0.008 )-- ( 15.8 , 0.008 )-- ( 15.9 , 0.008 )-- ( 16 , 0.008 )-- ( 16.1 , 0.007 )-- ( 16.2 , 0.007 )-- ( 16.3 , 0.007 )-- ( 16.4 , 0.007 )-- ( 16.5 , 0.007 )-- ( 16.6 , 0.007 )-- ( 16.7 , 0.006 )-- ( 16.8 , 0.006 )-- ( 16.9 , 0.006 )-- ( 17 , 0.006 )-- ( 17.1 , 0.006 )-- ( 17.2 , 0.006 )-- ( 17.3 , 0.005 )-- ( 17.4 , 0.005 )-- ( 17.5 , 0.005 )-- ( 17.6 , 0.005 )-- ( 17.7 , 0.005 )-- ( 17.8 , 0.005 )-- ( 17.9 , 0.005 )-- ( 18 , 0.005 )-- ( 18.1 , 0.004 )-- ( 18.2 , 0.004 )-- ( 18.3 , 0.004 )-- ( 18.4 , 0.004 )-- ( 18.5 , 0.004 )-- ( 18.6 , 0.004 )-- ( 18.7 , 0.004 )-- ( 18.8 , 0.004 )-- ( 18.9 , 0.004 )-- ( 19 , 0.003 )-- ( 19.1 , 0.003 )-- ( 19.2 , 0.003 )-- ( 19.3 , 0.003 )-- ( 19.4 , 0.003 )-- ( 19.5 , 0.003 )-- ( 19.6 , 0.003 )-- ( 19.7 , 0.003 )-- ( 19.8 , 0.003 )-- ( 19.9 , 0.003 )-- ( 20 , 0.003 );
 
\draw[orange,thick,dashdotted] ( 0 , 1 )-- ( 0.1 , 0.955 )-- ( 0.2 , 0.911 )-- ( 0.3 , 0.871 )-- ( 0.4 , 0.832 )-- ( 0.5 , 0.795 )-- ( 0.6 , 0.76 )-- ( 0.7 , 0.726 )-- ( 0.8 , 0.694 )-- ( 0.9 , 0.664 )-- ( 1 , 0.636 )-- ( 1.1 , 0.608 )-- ( 1.2 , 0.582 )-- ( 1.3 , 0.558 )-- ( 1.4 , 0.534 )-- ( 1.5 , 0.512 )-- ( 1.6 , 0.49 )-- ( 1.7 , 0.47 )-- ( 1.8 , 0.451 )-- ( 1.9 , 0.432 )-- ( 2 , 0.415 )-- ( 2.1 , 0.398 )-- ( 2.2 , 0.382 )-- ( 2.3 , 0.367 )-- ( 2.4 , 0.353 )-- ( 2.5 , 0.339 )-- ( 2.6 , 0.326 )-- ( 2.7 , 0.313 )-- ( 2.8 , 0.301 )-- ( 2.9 , 0.29 )-- ( 3 , 0.279 )-- ( 3.1 , 0.268 )-- ( 3.2 , 0.258 )-- ( 3.3 , 0.249 )-- ( 3.4 , 0.24 )-- ( 3.5 , 0.231 )-- ( 3.6 , 0.223 )-- ( 3.7 , 0.215 )-- ( 3.8 , 0.207 )-- ( 3.9 , 0.2 )-- ( 4 , 0.193 )-- ( 4.1 , 0.186 )-- ( 4.2 , 0.18 )-- ( 4.3 , 0.174 )-- ( 4.4 , 0.168 )-- ( 4.5 , 0.163 )-- ( 4.6 , 0.157 )-- ( 4.7 , 0.152 )-- ( 4.8 , 0.147 )-- ( 4.9 , 0.143 )-- ( 5 , 0.138 )-- ( 5.1 , 0.134 )-- ( 5.2 , 0.13 )-- ( 5.3 , 0.126 )-- ( 5.4 , 0.122 )-- ( 5.5 , 0.118 )-- ( 5.6 , 0.115 )-- ( 5.7 , 0.111 )-- ( 5.8 , 0.108 )-- ( 5.9 , 0.105 )-- ( 6 , 0.102 )-- ( 6.1 , 0.099 )-- ( 6.2 , 0.096 )-- ( 6.3 , 0.093 )-- ( 6.4 , 0.091 )-- ( 6.5 , 0.088 )-- ( 6.6 , 0.086 )-- ( 6.7 , 0.084 )-- ( 6.8 , 0.081 )-- ( 6.9 , 0.079 )-- ( 7 , 0.077 )-- ( 7.1 , 0.075 )-- ( 7.2 , 0.073 )-- ( 7.3 , 0.071 )-- ( 7.4 , 0.069 )-- ( 7.5 , 0.068 )-- ( 7.6 , 0.066 )-- ( 7.7 , 0.064 )-- ( 7.8 , 0.063 )-- ( 7.9 , 0.061 )-- ( 8 , 0.06 )-- ( 8.1 , 0.058 )-- ( 8.2 , 0.057 )-- ( 8.3 , 0.056 )-- ( 8.4 , 0.055 )-- ( 8.5 , 0.053 )-- ( 8.6 , 0.052 )-- ( 8.7 , 0.051 )-- ( 8.8 , 0.05 )-- ( 8.9 , 0.049 )-- ( 9 , 0.048 )-- ( 9.1 , 0.047 )-- ( 9.2 , 0.046 )-- ( 9.3 , 0.045 )-- ( 9.4 , 0.044 )-- ( 9.5 , 0.043 )-- ( 9.6 , 0.042 )-- ( 9.7 , 0.041 )-- ( 9.8 , 0.04 )-- ( 9.9 , 0.039 )-- ( 10 , 0.038 )-- ( 10.1 , 0.038 )-- ( 10.2 , 0.037 )-- ( 10.3 , 0.036 )-- ( 10.4 , 0.035 )-- ( 10.5 , 0.035 )-- ( 10.6 , 0.034 )-- ( 10.7 , 0.033 )-- ( 10.8 , 0.033 )-- ( 10.9 , 0.032 )-- ( 11 , 0.032 )-- ( 11.1 , 0.031 )-- ( 11.2 , 0.03 )-- ( 11.3 , 0.03 )-- ( 11.4 , 0.029 )-- ( 11.5 , 0.029 )-- ( 11.6 , 0.028 )-- ( 11.7 , 0.028 )-- ( 11.8 , 0.027 )-- ( 11.9 , 0.027 )-- ( 12 , 0.026 )-- ( 12.1 , 0.026 )-- ( 12.2 , 0.025 )-- ( 12.3 , 0.025 )-- ( 12.4 , 0.024 )-- ( 12.5 , 0.024 )-- ( 12.6 , 0.023 )-- ( 12.7 , 0.023 )-- ( 12.8 , 0.023 )-- ( 12.9 , 0.022 )-- ( 13 , 0.022 )-- ( 13.1 , 0.021 )-- ( 13.2 , 0.021 )-- ( 13.3 , 0.021 )-- ( 13.4 , 0.02 )-- ( 13.5 , 0.02 )-- ( 13.6 , 0.02 )-- ( 13.7 , 0.019 )-- ( 13.8 , 0.019 )-- ( 13.9 , 0.019 )-- ( 14 , 0.018 )-- ( 14.1 , 0.018 )-- ( 14.2 , 0.018 )-- ( 14.3 , 0.017 )-- ( 14.4 , 0.017 )-- ( 14.5 , 0.017 )-- ( 14.6 , 0.017 )-- ( 14.7 , 0.016 )-- ( 14.8 , 0.016 )-- ( 14.9 , 0.016 )-- ( 15 , 0.016 )-- ( 15.1 , 0.015 )-- ( 15.2 , 0.015 )-- ( 15.3 , 0.015 )-- ( 15.4 , 0.015 )-- ( 15.5 , 0.014 )-- ( 15.6 , 0.014 )-- ( 15.7 , 0.014 )-- ( 15.8 , 0.014 )-- ( 15.9 , 0.013 )-- ( 16 , 0.013 )-- ( 16.1 , 0.013 )-- ( 16.2 , 0.013 )-- ( 16.3 , 0.013 )-- ( 16.4 , 0.012 )-- ( 16.5 , 0.012 )-- ( 16.6 , 0.012 )-- ( 16.7 , 0.012 )-- ( 16.8 , 0.012 )-- ( 16.9 , 0.011 )-- ( 17 , 0.011 )-- ( 17.1 , 0.011 )-- ( 17.2 , 0.011 )-- ( 17.3 , 0.011 )-- ( 17.4 , 0.011 )-- ( 17.5 , 0.01 )-- ( 17.6 , 0.01 )-- ( 17.7 , 0.01 )-- ( 17.8 , 0.01 )-- ( 17.9 , 0.01 )-- ( 18 , 0.01 )-- ( 18.1 , 0.009 )-- ( 18.2 , 0.009 )-- ( 18.3 , 0.009 )-- ( 18.4 , 0.009 )-- ( 18.5 , 0.009 )-- ( 18.6 , 0.009 )-- ( 18.7 , 0.009 )-- ( 18.8 , 0.008 )-- ( 18.9 , 0.008 )-- ( 19 , 0.008 )-- ( 19.1 , 0.008 )-- ( 19.2 , 0.008 )-- ( 19.3 , 0.008 )-- ( 19.4 , 0.008 )-- ( 19.5 , 0.008 )-- ( 19.6 , 0.007 )-- ( 19.7 , 0.007 )-- ( 19.8 , 0.007 )-- ( 19.9 , 0.007 )-- ( 20 , 0.007 );

\draw[olive,loosely dashed,thick] ( 0 , 1 )-- ( 0.1 , 0.978 )-- ( 0.2 , 0.949 )-- ( 0.3 , 0.916 )-- ( 0.4 , 0.88 )-- ( 0.5 , 0.842 )-- ( 0.6 , 0.804 )-- ( 0.7 , 0.766 )-- ( 0.8 , 0.729 )-- ( 0.9 , 0.693 )-- ( 1 , 0.659 )-- ( 1.1 , 0.627 )-- ( 1.2 , 0.596 )-- ( 1.3 , 0.567 )-- ( 1.4 , 0.539 )-- ( 1.5 , 0.513 )-- ( 1.6 , 0.489 )-- ( 1.7 , 0.466 )-- ( 1.8 , 0.445 )-- ( 1.9 , 0.425 )-- ( 2 , 0.407 )-- ( 2.1 , 0.389 )-- ( 2.2 , 0.373 )-- ( 2.3 , 0.358 )-- ( 2.4 , 0.343 )-- ( 2.5 , 0.33 )-- ( 2.6 , 0.317 )-- ( 2.7 , 0.305 )-- ( 2.8 , 0.294 )-- ( 2.9 , 0.284 )-- ( 3 , 0.273 )-- ( 3.1 , 0.264 )-- ( 3.2 , 0.255 )-- ( 3.3 , 0.246 )-- ( 3.4 , 0.238 )-- ( 3.5 , 0.23 )-- ( 3.6 , 0.223 )-- ( 3.7 , 0.216 )-- ( 3.8 , 0.209 )-- ( 3.9 , 0.203 )-- ( 4 , 0.196 )-- ( 4.1 , 0.19 )-- ( 4.2 , 0.185 )-- ( 4.3 , 0.179 )-- ( 4.4 , 0.174 )-- ( 4.5 , 0.169 )-- ( 4.6 , 0.164 )-- ( 4.7 , 0.159 )-- ( 4.8 , 0.154 )-- ( 4.9 , 0.15 )-- ( 5 , 0.146 )-- ( 5.1 , 0.142 )-- ( 5.2 , 0.138 )-- ( 5.3 , 0.134 )-- ( 5.4 , 0.13 )-- ( 5.5 , 0.126 )-- ( 5.6 , 0.123 )-- ( 5.7 , 0.119 )-- ( 5.8 , 0.116 )-- ( 5.9 , 0.113 )-- ( 6 , 0.11 )-- ( 6.1 , 0.107 )-- ( 6.2 , 0.104 )-- ( 6.3 , 0.101 )-- ( 6.4 , 0.098 )-- ( 6.5 , 0.095 )-- ( 6.6 , 0.093 )-- ( 6.7 , 0.09 )-- ( 6.8 , 0.088 )-- ( 6.9 , 0.085 )-- ( 7 , 0.083 )-- ( 7.1 , 0.081 )-- ( 7.2 , 0.079 )-- ( 7.3 , 0.076 )-- ( 7.4 , 0.074 )-- ( 7.5 , 0.072 )-- ( 7.6 , 0.07 )-- ( 7.7 , 0.068 )-- ( 7.8 , 0.067 )-- ( 7.9 , 0.065 )-- ( 8 , 0.063 )-- ( 8.1 , 0.061 )-- ( 8.2 , 0.06 )-- ( 8.3 , 0.058 )-- ( 8.4 , 0.056 )-- ( 8.5 , 0.055 )-- ( 8.6 , 0.053 )-- ( 8.7 , 0.052 )-- ( 8.8 , 0.05 )-- ( 8.9 , 0.049 )-- ( 9 , 0.048 )-- ( 9.1 , 0.046 )-- ( 9.2 , 0.045 )-- ( 9.3 , 0.044 )-- ( 9.4 , 0.043 )-- ( 9.5 , 0.042 )-- ( 9.6 , 0.041 )-- ( 9.7 , 0.039 )-- ( 9.8 , 0.038 )-- ( 9.9 , 0.037 )-- ( 10 , 0.036 )-- ( 10.1 , 0.035 )-- ( 10.2 , 0.034 )-- ( 10.3 , 0.033 )-- ( 10.4 , 0.032 )-- ( 10.5 , 0.032 )-- ( 10.6 , 0.031 )-- ( 10.7 , 0.03 )-- ( 10.8 , 0.029 )-- ( 10.9 , 0.028 )-- ( 11 , 0.028 )-- ( 11.1 , 0.027 )-- ( 11.2 , 0.026 )-- ( 11.3 , 0.025 )-- ( 11.4 , 0.025 )-- ( 11.5 , 0.024 )-- ( 11.6 , 0.023 )-- ( 11.7 , 0.023 )-- ( 11.8 , 0.022 )-- ( 11.9 , 0.021 )-- ( 12 , 0.021 )-- ( 12.1 , 0.02 )-- ( 12.2 , 0.02 )-- ( 12.3 , 0.019 )-- ( 12.4 , 0.019 )-- ( 12.5 , 0.018 )-- ( 12.6 , 0.018 )-- ( 12.7 , 0.017 )-- ( 12.8 , 0.017 )-- ( 12.9 , 0.016 )-- ( 13 , 0.016 )-- ( 13.1 , 0.015 )-- ( 13.2 , 0.015 )-- ( 13.3 , 0.015 )-- ( 13.4 , 0.014 )-- ( 13.5 , 0.014 )-- ( 13.6 , 0.013 )-- ( 13.7 , 0.013 )-- ( 13.8 , 0.013 )-- ( 13.9 , 0.012 )-- ( 14 , 0.012 )-- ( 14.1 , 0.012 )-- ( 14.2 , 0.011 )-- ( 14.3 , 0.011 )-- ( 14.4 , 0.011 )-- ( 14.5 , 0.011 )-- ( 14.6 , 0.01 )-- ( 14.7 , 0.01 )-- ( 14.8 , 0.01 )-- ( 14.9 , 0.009 )-- ( 15 , 0.009 )-- ( 15.1 , 0.009 )-- ( 15.2 , 0.009 )-- ( 15.3 , 0.008 )-- ( 15.4 , 0.008 )-- ( 15.5 , 0.008 )-- ( 15.6 , 0.008 )-- ( 15.7 , 0.008 )-- ( 15.8 , 0.007 )-- ( 15.9 , 0.007 )-- ( 16 , 0.007 )-- ( 16.1 , 0.007 )-- ( 16.2 , 0.007 )-- ( 16.3 , 0.006 )-- ( 16.4 , 0.006 )-- ( 16.5 , 0.006 )-- ( 16.6 , 0.006 )-- ( 16.7 , 0.006 )-- ( 16.8 , 0.006 )-- ( 16.9 , 0.005 )-- ( 17 , 0.005 )-- ( 17.1 , 0.005 )-- ( 17.2 , 0.005 )-- ( 17.3 , 0.005 )-- ( 17.4 , 0.005 )-- ( 17.5 , 0.005 )-- ( 17.6 , 0.004 )-- ( 17.7 , 0.004 )-- ( 17.8 , 0.004 )-- ( 17.9 , 0.004 )-- ( 18 , 0.004 )-- ( 18.1 , 0.004 )-- ( 18.2 , 0.004 )-- ( 18.3 , 0.004 )-- ( 18.4 , 0.004 )-- ( 18.5 , 0.003 )-- ( 18.6 , 0.003 )-- ( 18.7 , 0.003 )-- ( 18.8 , 0.003 )-- ( 18.9 , 0.003 )-- ( 19 , 0.003 )-- ( 19.1 , 0.003 )-- ( 19.2 , 0.003 )-- ( 19.3 , 0.003 )-- ( 19.4 , 0.003 )-- ( 19.5 , 0.003 )-- ( 19.6 , 0.003 )-- ( 19.7 , 0.003 )-- ( 19.8 , 0.002 )-- ( 19.9 , 0.002 )-- ( 20 , 0.002 ) ; 
\end{tikzpicture}

\caption{Lognormal$(0.5,1)$}\label{fig:H_MLE_lognormal}
\end{subfigure}

\centering
\begin{subfigure}{.48\linewidth}
\begin{tikzpicture}[xscale=0.35,yscale=3.5]
  \def\xmin{0}
 \def\xmax{20}
 \def\ymin{0}
 \def\ymax{1.05}
 \draw[->] (\xmin,\ymin) -- (\xmax,\ymin) node[right] {$t$} ;
 \draw[->] (\xmin,\ymin) -- (\xmin,\ymax)  ;
 \foreach \x in {0,5,10,20} {
 \node at (\x,\ymin-0.05) [below] {\x};
 \draw[-] (\x,\ymin) -- (\x,{\ymin-0.05});
 }
\foreach \y in {0,0.5,1} {
 \node at (\xmin,\y) [left] {\y};
 \draw[-] (\xmin,\y) -- ({\xmin-0.05},\y);
 }
 
\draw[black]  ( 0 , 1 )-- ( 0.1 , 0.992 )-- ( 0.2 , 0.977 )-- ( 0.3 , 0.96 )-- ( 0.4 , 0.94 )-- ( 0.5 , 0.919 )-- ( 0.6 , 0.896 )-- ( 0.7 , 0.873 )-- ( 0.8 , 0.849 )-- ( 0.9 , 0.826 )-- ( 1 , 0.802 )-- ( 1.1 , 0.777 )-- ( 1.2 , 0.753 )-- ( 1.3 , 0.73 )-- ( 1.4 , 0.706 )-- ( 1.5 , 0.683 )-- ( 1.6 , 0.66 )-- ( 1.7 , 0.637 )-- ( 1.8 , 0.615 )-- ( 1.9 , 0.593 )-- ( 2 , 0.572 )-- ( 2.1 , 0.552 )-- ( 2.2 , 0.532 )-- ( 2.3 , 0.512 )-- ( 2.4 , 0.494 )-- ( 2.5 , 0.475 )-- ( 2.6 , 0.458 )-- ( 2.7 , 0.44 )-- ( 2.8 , 0.424 )-- ( 2.9 , 0.407 )-- ( 3 , 0.392 )-- ( 3.1 , 0.376 )-- ( 3.2 , 0.362 )-- ( 3.3 , 0.348 )-- ( 3.4 , 0.334 )-- ( 3.5 , 0.321 )-- ( 3.6 , 0.308 )-- ( 3.7 , 0.296 )-- ( 3.8 , 0.284 )-- ( 3.9 , 0.273 )-- ( 4 , 0.262 )-- ( 4.1 , 0.251 )-- ( 4.2 , 0.241 )-- ( 4.3 , 0.231 )-- ( 4.4 , 0.222 )-- ( 4.5 , 0.213 )-- ( 4.6 , 0.204 )-- ( 4.7 , 0.195 )-- ( 4.8 , 0.187 )-- ( 4.9 , 0.179 )-- ( 5 , 0.172 )-- ( 5.1 , 0.165 )-- ( 5.2 , 0.158 )-- ( 5.3 , 0.151 )-- ( 5.4 , 0.145 )-- ( 5.5 , 0.139 )-- ( 5.6 , 0.133 )-- ( 5.7 , 0.127 )-- ( 5.8 , 0.122 )-- ( 5.9 , 0.117 )-- ( 6 , 0.112 )-- ( 6.1 , 0.107 )-- ( 6.2 , 0.102 )-- ( 6.3 , 0.098 )-- ( 6.4 , 0.094 )-- ( 6.5 , 0.09 )-- ( 6.6 , 0.086 )-- ( 6.7 , 0.082 )-- ( 6.8 , 0.079 )-- ( 6.9 , 0.075 )-- ( 7 , 0.072 )-- ( 7.1 , 0.069 )-- ( 7.2 , 0.066 )-- ( 7.3 , 0.063 )-- ( 7.4 , 0.06 )-- ( 7.5 , 0.057 )-- ( 7.6 , 0.055 )-- ( 7.7 , 0.052 )-- ( 7.8 , 0.05 )-- ( 7.9 , 0.048 )-- ( 8 , 0.046 )-- ( 8.1 , 0.044 )-- ( 8.2 , 0.042 )-- ( 8.3 , 0.04 )-- ( 8.4 , 0.038 )-- ( 8.5 , 0.036 )-- ( 8.6 , 0.035 )-- ( 8.7 , 0.033 )-- ( 8.8 , 0.032 )-- ( 8.9 , 0.031 )-- ( 9 , 0.029 )-- ( 9.1 , 0.028 )-- ( 9.2 , 0.027 )-- ( 9.3 , 0.026 )-- ( 9.4 , 0.024 )-- ( 9.5 , 0.023 )-- ( 9.6 , 0.022 )-- ( 9.7 , 0.021 )-- ( 9.8 , 0.02 )-- ( 9.9 , 0.019 )-- ( 10 , 0.018 )-- ( 10.1 , 0.018 )-- ( 10.2 , 0.017 )-- ( 10.3 , 0.016 )-- ( 10.4 , 0.015 )-- ( 10.5 , 0.015 )-- ( 10.6 , 0.014 )-- ( 10.7 , 0.013 )-- ( 10.8 , 0.013 )-- ( 10.9 , 0.012 )-- ( 11 , 0.012 )-- ( 11.1 , 0.011 )-- ( 11.2 , 0.011 )-- ( 11.3 , 0.01 )-- ( 11.4 , 0.01 )-- ( 11.5 , 0.009 )-- ( 11.6 , 0.009 )-- ( 11.7 , 0.008 )-- ( 11.8 , 0.008 )-- ( 11.9 , 0.008 )-- ( 12 , 0.007 )-- ( 12.1 , 0.007 )-- ( 12.2 , 0.007 )-- ( 12.3 , 0.006 )-- ( 12.4 , 0.006 )-- ( 12.5 , 0.006 )-- ( 12.6 , 0.006 )-- ( 12.7 , 0.005 )-- ( 12.8 , 0.005 )-- ( 12.9 , 0.005 )-- ( 13 , 0.005 )-- ( 13.1 , 0.004 )-- ( 13.2 , 0.004 )-- ( 13.3 , 0.004 )-- ( 13.4 , 0.004 )-- ( 13.5 , 0.004 )-- ( 13.6 , 0.003 )-- ( 13.7 , 0.003 )-- ( 13.8 , 0.003 )-- ( 13.9 , 0.003 )-- ( 14 , 0.003 )-- ( 14.1 , 0.003 )-- ( 14.2 , 0.003 )-- ( 14.3 , 0.003 )-- ( 14.4 , 0.002 )-- ( 14.5 , 0.002 )-- ( 14.6 , 0.002 )-- ( 14.7 , 0.002 )-- ( 14.8 , 0.002 )-- ( 14.9 , 0.002 )-- ( 15 , 0.002 )-- ( 15.1 , 0.002 )-- ( 15.2 , 0.002 )-- ( 15.3 , 0.002 )-- ( 15.4 , 0.002 )-- ( 15.5 , 0.001 )-- ( 15.6 , 0.001 )-- ( 15.7 , 0.001 )-- ( 15.8 , 0.001 )-- ( 15.9 , 0.001 )-- ( 16 , 0.001 )-- ( 16.1 , 0.001 )-- ( 16.2 , 0.001 )-- ( 16.3 , 0.001 )-- ( 16.4 , 0.001 )-- ( 16.5 , 0.001 )-- ( 16.6 , 0.001 )-- ( 16.7 , 0.001 )-- ( 16.8 , 0.001 )-- ( 16.9 , 0.001 )-- ( 17 , 0.001 )-- ( 17.1 , 0.001 )-- ( 17.2 , 0.001 )-- ( 17.3 , 0.001 )-- ( 17.4 , 0.001 )-- ( 17.5 , 0.001 )-- ( 17.6 , 0.001 )-- ( 17.7 , 0.001 )-- ( 17.8 , 0.001 )-- ( 17.9 , 0.001 )-- ( 18 , 0 )-- ( 18.1 , 0 )-- ( 18.2 , 0 )-- ( 18.3 , 0 )-- ( 18.4 , 0 )-- ( 18.5 , 0 )-- ( 18.6 , 0 )-- ( 18.7 , 0 )-- ( 18.8 , 0 )-- ( 18.9 , 0 )-- ( 19 , 0 )-- ( 19.1 , 0 )-- ( 19.2 , 0 )-- ( 19.3 , 0 )-- ( 19.4 , 0 )-- ( 19.5 , 0 )-- ( 19.6 , 0 )-- ( 19.7 , 0 )-- ( 19.8 , 0 )-- ( 19.9 , 0 )-- ( 20 , 0 );

 \draw[red,dotted,thick] ( 0 , 1 )-- ( 0.1 , 0.961 )-- ( 0.2 , 0.924 )-- ( 0.3 , 0.888 )-- ( 0.4 , 0.853 )-- ( 0.5 , 0.82 )-- ( 0.6 , 0.788 )-- ( 0.7 , 0.758 )-- ( 0.8 , 0.728 )-- ( 0.9 , 0.7 )-- ( 1 , 0.673 )-- ( 1.1 , 0.647 )-- ( 1.2 , 0.622 )-- ( 1.3 , 0.598 )-- ( 1.4 , 0.574 )-- ( 1.5 , 0.552 )-- ( 1.6 , 0.531 )-- ( 1.7 , 0.51 )-- ( 1.8 , 0.49 )-- ( 1.9 , 0.471 )-- ( 2 , 0.453 )-- ( 2.1 , 0.435 )-- ( 2.2 , 0.418 )-- ( 2.3 , 0.402 )-- ( 2.4 , 0.386 )-- ( 2.5 , 0.371 )-- ( 2.6 , 0.357 )-- ( 2.7 , 0.343 )-- ( 2.8 , 0.33 )-- ( 2.9 , 0.317 )-- ( 3 , 0.305 )-- ( 3.1 , 0.293 )-- ( 3.2 , 0.281 )-- ( 3.3 , 0.271 )-- ( 3.4 , 0.26 )-- ( 3.5 , 0.25 )-- ( 3.6 , 0.24 )-- ( 3.7 , 0.231 )-- ( 3.8 , 0.222 )-- ( 3.9 , 0.213 )-- ( 4 , 0.205 )-- ( 4.1 , 0.197 )-- ( 4.2 , 0.189 )-- ( 4.3 , 0.182 )-- ( 4.4 , 0.175 )-- ( 4.5 , 0.168 )-- ( 4.6 , 0.162 )-- ( 4.7 , 0.155 )-- ( 4.8 , 0.149 )-- ( 4.9 , 0.144 )-- ( 5 , 0.138 )-- ( 5.1 , 0.133 )-- ( 5.2 , 0.127 )-- ( 5.3 , 0.123 )-- ( 5.4 , 0.118 )-- ( 5.5 , 0.113 )-- ( 5.6 , 0.109 )-- ( 5.7 , 0.105 )-- ( 5.8 , 0.1 )-- ( 5.9 , 0.097 )-- ( 6 , 0.093 )-- ( 6.1 , 0.089 )-- ( 6.2 , 0.086 )-- ( 6.3 , 0.082 )-- ( 6.4 , 0.079 )-- ( 6.5 , 0.076 )-- ( 6.6 , 0.073 )-- ( 6.7 , 0.07 )-- ( 6.8 , 0.068 )-- ( 6.9 , 0.065 )-- ( 7 , 0.062 )-- ( 7.1 , 0.06 )-- ( 7.2 , 0.058 )-- ( 7.3 , 0.055 )-- ( 7.4 , 0.053 )-- ( 7.5 , 0.051 )-- ( 7.6 , 0.049 )-- ( 7.7 , 0.047 )-- ( 7.8 , 0.046 )-- ( 7.9 , 0.044 )-- ( 8 , 0.042 )-- ( 8.1 , 0.04 )-- ( 8.2 , 0.039 )-- ( 8.3 , 0.037 )-- ( 8.4 , 0.036 )-- ( 8.5 , 0.034 )-- ( 8.6 , 0.033 )-- ( 8.7 , 0.032 )-- ( 8.8 , 0.031 )-- ( 8.9 , 0.029 )-- ( 9 , 0.028 )-- ( 9.1 , 0.027 )-- ( 9.2 , 0.026 )-- ( 9.3 , 0.025 )-- ( 9.4 , 0.024 )-- ( 9.5 , 0.023 )-- ( 9.6 , 0.022 )-- ( 9.7 , 0.021 )-- ( 9.8 , 0.021 )-- ( 9.9 , 0.02 )-- ( 10 , 0.019 )-- ( 10.1 , 0.018 )-- ( 10.2 , 0.018 )-- ( 10.3 , 0.017 )-- ( 10.4 , 0.016 )-- ( 10.5 , 0.016 )-- ( 10.6 , 0.015 )-- ( 10.7 , 0.014 )-- ( 10.8 , 0.014 )-- ( 10.9 , 0.013 )-- ( 11 , 0.013 )-- ( 11.1 , 0.012 )-- ( 11.2 , 0.012 )-- ( 11.3 , 0.011 )-- ( 11.4 , 0.011 )-- ( 11.5 , 0.011 )-- ( 11.6 , 0.01 )-- ( 11.7 , 0.01 )-- ( 11.8 , 0.009 )-- ( 11.9 , 0.009 )-- ( 12 , 0.009 )-- ( 12.1 , 0.008 )-- ( 12.2 , 0.008 )-- ( 12.3 , 0.008 )-- ( 12.4 , 0.007 )-- ( 12.5 , 0.007 )-- ( 12.6 , 0.007 )-- ( 12.7 , 0.007 )-- ( 12.8 , 0.006 )-- ( 12.9 , 0.006 )-- ( 13 , 0.006 )-- ( 13.1 , 0.006 )-- ( 13.2 , 0.005 )-- ( 13.3 , 0.005 )-- ( 13.4 , 0.005 )-- ( 13.5 , 0.005 )-- ( 13.6 , 0.005 )-- ( 13.7 , 0.004 )-- ( 13.8 , 0.004 )-- ( 13.9 , 0.004 )-- ( 14 , 0.004 )-- ( 14.1 , 0.004 )-- ( 14.2 , 0.004 )-- ( 14.3 , 0.003 )-- ( 14.4 , 0.003 )-- ( 14.5 , 0.003 )-- ( 14.6 , 0.003 )-- ( 14.7 , 0.003 )-- ( 14.8 , 0.003 )-- ( 14.9 , 0.003 )-- ( 15 , 0.003 )-- ( 15.1 , 0.003 )-- ( 15.2 , 0.002 )-- ( 15.3 , 0.002 )-- ( 15.4 , 0.002 )-- ( 15.5 , 0.002 )-- ( 15.6 , 0.002 )-- ( 15.7 , 0.002 )-- ( 15.8 , 0.002 )-- ( 15.9 , 0.002 )-- ( 16 , 0.002 )-- ( 16.1 , 0.002 )-- ( 16.2 , 0.002 )-- ( 16.3 , 0.002 )-- ( 16.4 , 0.002 )-- ( 16.5 , 0.001 )-- ( 16.6 , 0.001 )-- ( 16.7 , 0.001 )-- ( 16.8 , 0.001 )-- ( 16.9 , 0.001 )-- ( 17 , 0.001 )-- ( 17.1 , 0.001 )-- ( 17.2 , 0.001 )-- ( 17.3 , 0.001 )-- ( 17.4 , 0.001 )-- ( 17.5 , 0.001 )-- ( 17.6 , 0.001 )-- ( 17.7 , 0.001 )-- ( 17.8 , 0.001 )-- ( 17.9 , 0.001 )-- ( 18 , 0.001 )-- ( 18.1 , 0.001 )-- ( 18.2 , 0.001 )-- ( 18.3 , 0.001 )-- ( 18.4 , 0.001 )-- ( 18.5 , 0.001 )-- ( 18.6 , 0.001 )-- ( 18.7 , 0.001 )-- ( 18.8 , 0.001 )-- ( 18.9 , 0.001 )-- ( 19 , 0.001 )-- ( 19.1 , 0.001 )-- ( 19.2 , 0 )-- ( 19.3 , 0 )-- ( 19.4 , 0 )-- ( 19.5 , 0 )-- ( 19.6 , 0 )-- ( 19.7 , 0 )-- ( 19.8 , 0 )-- ( 19.9 , 0 )-- ( 20 , 0 );
 
\draw[blue,thick,densely dashed] ( 0 , 1 )-- ( 0.1 , 0.964 )-- ( 0.2 , 0.93 )-- ( 0.3 , 0.897 )-- ( 0.4 , 0.865 )-- ( 0.5 , 0.834 )-- ( 0.6 , 0.804 )-- ( 0.7 , 0.776 )-- ( 0.8 , 0.748 )-- ( 0.9 , 0.721 )-- ( 1 , 0.696 )-- ( 1.1 , 0.671 )-- ( 1.2 , 0.647 )-- ( 1.3 , 0.624 )-- ( 1.4 , 0.602 )-- ( 1.5 , 0.58 )-- ( 1.6 , 0.559 )-- ( 1.7 , 0.539 )-- ( 1.8 , 0.52 )-- ( 1.9 , 0.502 )-- ( 2 , 0.484 )-- ( 2.1 , 0.467 )-- ( 2.2 , 0.45 )-- ( 2.3 , 0.434 )-- ( 2.4 , 0.418 )-- ( 2.5 , 0.404 )-- ( 2.6 , 0.389 )-- ( 2.7 , 0.375 )-- ( 2.8 , 0.362 )-- ( 2.9 , 0.349 )-- ( 3 , 0.337 )-- ( 3.1 , 0.325 )-- ( 3.2 , 0.313 )-- ( 3.3 , 0.302 )-- ( 3.4 , 0.291 )-- ( 3.5 , 0.281 )-- ( 3.6 , 0.271 )-- ( 3.7 , 0.261 )-- ( 3.8 , 0.252 )-- ( 3.9 , 0.243 )-- ( 4 , 0.234 )-- ( 4.1 , 0.226 )-- ( 4.2 , 0.218 )-- ( 4.3 , 0.21 )-- ( 4.4 , 0.202 )-- ( 4.5 , 0.195 )-- ( 4.6 , 0.188 )-- ( 4.7 , 0.182 )-- ( 4.8 , 0.175 )-- ( 4.9 , 0.169 )-- ( 5 , 0.163 )-- ( 5.1 , 0.157 )-- ( 5.2 , 0.151 )-- ( 5.3 , 0.146 )-- ( 5.4 , 0.141 )-- ( 5.5 , 0.136 )-- ( 5.6 , 0.131 )-- ( 5.7 , 0.126 )-- ( 5.8 , 0.122 )-- ( 5.9 , 0.117 )-- ( 6 , 0.113 )-- ( 6.1 , 0.109 )-- ( 6.2 , 0.105 )-- ( 6.3 , 0.102 )-- ( 6.4 , 0.098 )-- ( 6.5 , 0.094 )-- ( 6.6 , 0.091 )-- ( 6.7 , 0.088 )-- ( 6.8 , 0.085 )-- ( 6.9 , 0.082 )-- ( 7 , 0.079 )-- ( 7.1 , 0.076 )-- ( 7.2 , 0.073 )-- ( 7.3 , 0.071 )-- ( 7.4 , 0.068 )-- ( 7.5 , 0.066 )-- ( 7.6 , 0.063 )-- ( 7.7 , 0.061 )-- ( 7.8 , 0.059 )-- ( 7.9 , 0.057 )-- ( 8 , 0.055 )-- ( 8.1 , 0.053 )-- ( 8.2 , 0.051 )-- ( 8.3 , 0.049 )-- ( 8.4 , 0.047 )-- ( 8.5 , 0.046 )-- ( 8.6 , 0.044 )-- ( 8.7 , 0.043 )-- ( 8.8 , 0.041 )-- ( 8.9 , 0.04 )-- ( 9 , 0.038 )-- ( 9.1 , 0.037 )-- ( 9.2 , 0.035 )-- ( 9.3 , 0.034 )-- ( 9.4 , 0.033 )-- ( 9.5 , 0.032 )-- ( 9.6 , 0.031 )-- ( 9.7 , 0.03 )-- ( 9.8 , 0.029 )-- ( 9.9 , 0.027 )-- ( 10 , 0.027 )-- ( 10.1 , 0.026 )-- ( 10.2 , 0.025 )-- ( 10.3 , 0.024 )-- ( 10.4 , 0.023 )-- ( 10.5 , 0.022 )-- ( 10.6 , 0.021 )-- ( 10.7 , 0.021 )-- ( 10.8 , 0.02 )-- ( 10.9 , 0.019 )-- ( 11 , 0.018 )-- ( 11.1 , 0.018 )-- ( 11.2 , 0.017 )-- ( 11.3 , 0.017 )-- ( 11.4 , 0.016 )-- ( 11.5 , 0.015 )-- ( 11.6 , 0.015 )-- ( 11.7 , 0.014 )-- ( 11.8 , 0.014 )-- ( 11.9 , 0.013 )-- ( 12 , 0.013 )-- ( 12.1 , 0.012 )-- ( 12.2 , 0.012 )-- ( 12.3 , 0.012 )-- ( 12.4 , 0.011 )-- ( 12.5 , 0.011 )-- ( 12.6 , 0.01 )-- ( 12.7 , 0.01 )-- ( 12.8 , 0.01 )-- ( 12.9 , 0.009 )-- ( 13 , 0.009 )-- ( 13.1 , 0.009 )-- ( 13.2 , 0.008 )-- ( 13.3 , 0.008 )-- ( 13.4 , 0.008 )-- ( 13.5 , 0.007 )-- ( 13.6 , 0.007 )-- ( 13.7 , 0.007 )-- ( 13.8 , 0.007 )-- ( 13.9 , 0.006 )-- ( 14 , 0.006 )-- ( 14.1 , 0.006 )-- ( 14.2 , 0.006 )-- ( 14.3 , 0.006 )-- ( 14.4 , 0.005 )-- ( 14.5 , 0.005 )-- ( 14.6 , 0.005 )-- ( 14.7 , 0.005 )-- ( 14.8 , 0.005 )-- ( 14.9 , 0.004 )-- ( 15 , 0.004 )-- ( 15.1 , 0.004 )-- ( 15.2 , 0.004 )-- ( 15.3 , 0.004 )-- ( 15.4 , 0.004 )-- ( 15.5 , 0.004 )-- ( 15.6 , 0.003 )-- ( 15.7 , 0.003 )-- ( 15.8 , 0.003 )-- ( 15.9 , 0.003 )-- ( 16 , 0.003 )-- ( 16.1 , 0.003 )-- ( 16.2 , 0.003 )-- ( 16.3 , 0.003 )-- ( 16.4 , 0.003 )-- ( 16.5 , 0.003 )-- ( 16.6 , 0.002 )-- ( 16.7 , 0.002 )-- ( 16.8 , 0.002 )-- ( 16.9 , 0.002 )-- ( 17 , 0.002 )-- ( 17.1 , 0.002 )-- ( 17.2 , 0.002 )-- ( 17.3 , 0.002 )-- ( 17.4 , 0.002 )-- ( 17.5 , 0.002 )-- ( 17.6 , 0.002 )-- ( 17.7 , 0.002 )-- ( 17.8 , 0.002 )-- ( 17.9 , 0.002 )-- ( 18 , 0.001 )-- ( 18.1 , 0.001 )-- ( 18.2 , 0.001 )-- ( 18.3 , 0.001 )-- ( 18.4 , 0.001 )-- ( 18.5 , 0.001 )-- ( 18.6 , 0.001 )-- ( 18.7 , 0.001 )-- ( 18.8 , 0.001 )-- ( 18.9 , 0.001 )-- ( 19 , 0.001 )-- ( 19.1 , 0.001 )-- ( 19.2 , 0.001 )-- ( 19.3 , 0.001 )-- ( 19.4 , 0.001 )-- ( 19.5 , 0.001 )-- ( 19.6 , 0.001 )-- ( 19.7 , 0.001 )-- ( 19.8 , 0.001 )-- ( 19.9 , 0.001 )-- ( 20 , 0.001 ) ;
 
\draw[orange,thick,dashdotted]  ( 0 , 1 )-- ( 0.1 , 0.964 )-- ( 0.2 , 0.93 )-- ( 0.3 , 0.897 )-- ( 0.4 , 0.865 )-- ( 0.5 , 0.834 )-- ( 0.6 , 0.805 )-- ( 0.7 , 0.776 )-- ( 0.8 , 0.748 )-- ( 0.9 , 0.722 )-- ( 1 , 0.696 )-- ( 1.1 , 0.671 )-- ( 1.2 , 0.647 )-- ( 1.3 , 0.624 )-- ( 1.4 , 0.602 )-- ( 1.5 , 0.581 )-- ( 1.6 , 0.56 )-- ( 1.7 , 0.54 )-- ( 1.8 , 0.521 )-- ( 1.9 , 0.502 )-- ( 2 , 0.484 )-- ( 2.1 , 0.467 )-- ( 2.2 , 0.451 )-- ( 2.3 , 0.434 )-- ( 2.4 , 0.419 )-- ( 2.5 , 0.404 )-- ( 2.6 , 0.39 )-- ( 2.7 , 0.376 )-- ( 2.8 , 0.362 )-- ( 2.9 , 0.35 )-- ( 3 , 0.337 )-- ( 3.1 , 0.325 )-- ( 3.2 , 0.314 )-- ( 3.3 , 0.302 )-- ( 3.4 , 0.292 )-- ( 3.5 , 0.281 )-- ( 3.6 , 0.271 )-- ( 3.7 , 0.262 )-- ( 3.8 , 0.252 )-- ( 3.9 , 0.243 )-- ( 4 , 0.235 )-- ( 4.1 , 0.226 )-- ( 4.2 , 0.218 )-- ( 4.3 , 0.21 )-- ( 4.4 , 0.203 )-- ( 4.5 , 0.196 )-- ( 4.6 , 0.189 )-- ( 4.7 , 0.182 )-- ( 4.8 , 0.176 )-- ( 4.9 , 0.169 )-- ( 5 , 0.163 )-- ( 5.1 , 0.158 )-- ( 5.2 , 0.152 )-- ( 5.3 , 0.147 )-- ( 5.4 , 0.141 )-- ( 5.5 , 0.136 )-- ( 5.6 , 0.131 )-- ( 5.7 , 0.127 )-- ( 5.8 , 0.122 )-- ( 5.9 , 0.118 )-- ( 6 , 0.114 )-- ( 6.1 , 0.11 )-- ( 6.2 , 0.106 )-- ( 6.3 , 0.102 )-- ( 6.4 , 0.098 )-- ( 6.5 , 0.095 )-- ( 6.6 , 0.091 )-- ( 6.7 , 0.088 )-- ( 6.8 , 0.085 )-- ( 6.9 , 0.082 )-- ( 7 , 0.079 )-- ( 7.1 , 0.076 )-- ( 7.2 , 0.074 )-- ( 7.3 , 0.071 )-- ( 7.4 , 0.068 )-- ( 7.5 , 0.066 )-- ( 7.6 , 0.064 )-- ( 7.7 , 0.061 )-- ( 7.8 , 0.059 )-- ( 7.9 , 0.057 )-- ( 8 , 0.055 )-- ( 8.1 , 0.053 )-- ( 8.2 , 0.051 )-- ( 8.3 , 0.049 )-- ( 8.4 , 0.048 )-- ( 8.5 , 0.046 )-- ( 8.6 , 0.044 )-- ( 8.7 , 0.043 )-- ( 8.8 , 0.041 )-- ( 8.9 , 0.04 )-- ( 9 , 0.038 )-- ( 9.1 , 0.037 )-- ( 9.2 , 0.036 )-- ( 9.3 , 0.034 )-- ( 9.4 , 0.033 )-- ( 9.5 , 0.032 )-- ( 9.6 , 0.031 )-- ( 9.7 , 0.03 )-- ( 9.8 , 0.029 )-- ( 9.9 , 0.028 )-- ( 10 , 0.027 )-- ( 10.1 , 0.026 )-- ( 10.2 , 0.025 )-- ( 10.3 , 0.024 )-- ( 10.4 , 0.023 )-- ( 10.5 , 0.022 )-- ( 10.6 , 0.021 )-- ( 10.7 , 0.021 )-- ( 10.8 , 0.02 )-- ( 10.9 , 0.019 )-- ( 11 , 0.019 )-- ( 11.1 , 0.018 )-- ( 11.2 , 0.017 )-- ( 11.3 , 0.017 )-- ( 11.4 , 0.016 )-- ( 11.5 , 0.016 )-- ( 11.6 , 0.015 )-- ( 11.7 , 0.014 )-- ( 11.8 , 0.014 )-- ( 11.9 , 0.013 )-- ( 12 , 0.013 )-- ( 12.1 , 0.012 )-- ( 12.2 , 0.012 )-- ( 12.3 , 0.012 )-- ( 12.4 , 0.011 )-- ( 12.5 , 0.011 )-- ( 12.6 , 0.01 )-- ( 12.7 , 0.01 )-- ( 12.8 , 0.01 )-- ( 12.9 , 0.009 )-- ( 13 , 0.009 )-- ( 13.1 , 0.009 )-- ( 13.2 , 0.008 )-- ( 13.3 , 0.008 )-- ( 13.4 , 0.008 )-- ( 13.5 , 0.008 )-- ( 13.6 , 0.007 )-- ( 13.7 , 0.007 )-- ( 13.8 , 0.007 )-- ( 13.9 , 0.007 )-- ( 14 , 0.006 )-- ( 14.1 , 0.006 )-- ( 14.2 , 0.006 )-- ( 14.3 , 0.006 )-- ( 14.4 , 0.005 )-- ( 14.5 , 0.005 )-- ( 14.6 , 0.005 )-- ( 14.7 , 0.005 )-- ( 14.8 , 0.005 )-- ( 14.9 , 0.005 )-- ( 15 , 0.004 )-- ( 15.1 , 0.004 )-- ( 15.2 , 0.004 )-- ( 15.3 , 0.004 )-- ( 15.4 , 0.004 )-- ( 15.5 , 0.004 )-- ( 15.6 , 0.004 )-- ( 15.7 , 0.003 )-- ( 15.8 , 0.003 )-- ( 15.9 , 0.003 )-- ( 16 , 0.003 )-- ( 16.1 , 0.003 )-- ( 16.2 , 0.003 )-- ( 16.3 , 0.003 )-- ( 16.4 , 0.003 )-- ( 16.5 , 0.003 )-- ( 16.6 , 0.002 )-- ( 16.7 , 0.002 )-- ( 16.8 , 0.002 )-- ( 16.9 , 0.002 )-- ( 17 , 0.002 )-- ( 17.1 , 0.002 )-- ( 17.2 , 0.002 )-- ( 17.3 , 0.002 )-- ( 17.4 , 0.002 )-- ( 17.5 , 0.002 )-- ( 17.6 , 0.002 )-- ( 17.7 , 0.002 )-- ( 17.8 , 0.002 )-- ( 17.9 , 0.002 )-- ( 18 , 0.001 )-- ( 18.1 , 0.001 )-- ( 18.2 , 0.001 )-- ( 18.3 , 0.001 )-- ( 18.4 , 0.001 )-- ( 18.5 , 0.001 )-- ( 18.6 , 0.001 )-- ( 18.7 , 0.001 )-- ( 18.8 , 0.001 )-- ( 18.9 , 0.001 )-- ( 19 , 0.001 )-- ( 19.1 , 0.001 )-- ( 19.2 , 0.001 )-- ( 19.3 , 0.001 )-- ( 19.4 , 0.001 )-- ( 19.5 , 0.001 )-- ( 19.6 , 0.001 )-- ( 19.7 , 0.001 )-- ( 19.8 , 0.001 )-- ( 19.9 , 0.001 )-- ( 20 , 0.001 ); 

\draw[olive,loosely dashed,thick] ( 0 , 1 )-- ( 0.1 , 0.952 )-- ( 0.2 , 0.93 )-- ( 0.3 , 0.919 )-- ( 0.4 , 0.909 )-- ( 0.5 , 0.898 )-- ( 0.6 , 0.883 )-- ( 0.7 , 0.865 )-- ( 0.8 , 0.845 )-- ( 0.9 , 0.823 )-- ( 1 , 0.799 )-- ( 1.1 , 0.774 )-- ( 1.2 , 0.749 )-- ( 1.3 , 0.724 )-- ( 1.4 , 0.699 )-- ( 1.5 , 0.675 )-- ( 1.6 , 0.651 )-- ( 1.7 , 0.628 )-- ( 1.8 , 0.606 )-- ( 1.9 , 0.584 )-- ( 2 , 0.563 )-- ( 2.1 , 0.543 )-- ( 2.2 , 0.523 )-- ( 2.3 , 0.504 )-- ( 2.4 , 0.485 )-- ( 2.5 , 0.467 )-- ( 2.6 , 0.45 )-- ( 2.7 , 0.433 )-- ( 2.8 , 0.417 )-- ( 2.9 , 0.401 )-- ( 3 , 0.386 )-- ( 3.1 , 0.371 )-- ( 3.2 , 0.357 )-- ( 3.3 , 0.343 )-- ( 3.4 , 0.329 )-- ( 3.5 , 0.316 )-- ( 3.6 , 0.304 )-- ( 3.7 , 0.292 )-- ( 3.8 , 0.28 )-- ( 3.9 , 0.269 )-- ( 4 , 0.258 )-- ( 4.1 , 0.248 )-- ( 4.2 , 0.238 )-- ( 4.3 , 0.228 )-- ( 4.4 , 0.219 )-- ( 4.5 , 0.21 )-- ( 4.6 , 0.201 )-- ( 4.7 , 0.193 )-- ( 4.8 , 0.185 )-- ( 4.9 , 0.177 )-- ( 5 , 0.17 )-- ( 5.1 , 0.163 )-- ( 5.2 , 0.156 )-- ( 5.3 , 0.149 )-- ( 5.4 , 0.143 )-- ( 5.5 , 0.137 )-- ( 5.6 , 0.131 )-- ( 5.7 , 0.126 )-- ( 5.8 , 0.12 )-- ( 5.9 , 0.115 )-- ( 6 , 0.11 )-- ( 6.1 , 0.106 )-- ( 6.2 , 0.101 )-- ( 6.3 , 0.097 )-- ( 6.4 , 0.093 )-- ( 6.5 , 0.089 )-- ( 6.6 , 0.085 )-- ( 6.7 , 0.081 )-- ( 6.8 , 0.078 )-- ( 6.9 , 0.074 )-- ( 7 , 0.071 )-- ( 7.1 , 0.068 )-- ( 7.2 , 0.065 )-- ( 7.3 , 0.062 )-- ( 7.4 , 0.06 )-- ( 7.5 , 0.057 )-- ( 7.6 , 0.055 )-- ( 7.7 , 0.052 )-- ( 7.8 , 0.05 )-- ( 7.9 , 0.048 )-- ( 8 , 0.046 )-- ( 8.1 , 0.044 )-- ( 8.2 , 0.042 )-- ( 8.3 , 0.04 )-- ( 8.4 , 0.038 )-- ( 8.5 , 0.037 )-- ( 8.6 , 0.035 )-- ( 8.7 , 0.034 )-- ( 8.8 , 0.032 )-- ( 8.9 , 0.031 )-- ( 9 , 0.029 )-- ( 9.1 , 0.028 )-- ( 9.2 , 0.027 )-- ( 9.3 , 0.026 )-- ( 9.4 , 0.025 )-- ( 9.5 , 0.024 )-- ( 9.6 , 0.023 )-- ( 9.7 , 0.022 )-- ( 9.8 , 0.021 )-- ( 9.9 , 0.02 )-- ( 10 , 0.019 )-- ( 10.1 , 0.018 )-- ( 10.2 , 0.017 )-- ( 10.3 , 0.017 )-- ( 10.4 , 0.016 )-- ( 10.5 , 0.015 )-- ( 10.6 , 0.014 )-- ( 10.7 , 0.014 )-- ( 10.8 , 0.013 )-- ( 10.9 , 0.013 )-- ( 11 , 0.012 )-- ( 11.1 , 0.012 )-- ( 11.2 , 0.011 )-- ( 11.3 , 0.011 )-- ( 11.4 , 0.01 )-- ( 11.5 , 0.01 )-- ( 11.6 , 0.009 )-- ( 11.7 , 0.009 )-- ( 11.8 , 0.009 )-- ( 11.9 , 0.008 )-- ( 12 , 0.008 )-- ( 12.1 , 0.007 )-- ( 12.2 , 0.007 )-- ( 12.3 , 0.007 )-- ( 12.4 , 0.007 )-- ( 12.5 , 0.006 )-- ( 12.6 , 0.006 )-- ( 12.7 , 0.006 )-- ( 12.8 , 0.005 )-- ( 12.9 , 0.005 )-- ( 13 , 0.005 )-- ( 13.1 , 0.005 )-- ( 13.2 , 0.005 )-- ( 13.3 , 0.004 )-- ( 13.4 , 0.004 )-- ( 13.5 , 0.004 )-- ( 13.6 , 0.004 )-- ( 13.7 , 0.004 )-- ( 13.8 , 0.004 )-- ( 13.9 , 0.003 )-- ( 14 , 0.003 )-- ( 14.1 , 0.003 )-- ( 14.2 , 0.003 )-- ( 14.3 , 0.003 )-- ( 14.4 , 0.003 )-- ( 14.5 , 0.003 )-- ( 14.6 , 0.002 )-- ( 14.7 , 0.002 )-- ( 14.8 , 0.002 )-- ( 14.9 , 0.002 )-- ( 15 , 0.002 )-- ( 15.1 , 0.002 )-- ( 15.2 , 0.002 )-- ( 15.3 , 0.002 )-- ( 15.4 , 0.002 )-- ( 15.5 , 0.002 )-- ( 15.6 , 0.002 )-- ( 15.7 , 0.002 )-- ( 15.8 , 0.001 )-- ( 15.9 , 0.001 )-- ( 16 , 0.001 )-- ( 16.1 , 0.001 )-- ( 16.2 , 0.001 )-- ( 16.3 , 0.001 )-- ( 16.4 , 0.001 )-- ( 16.5 , 0.001 )-- ( 16.6 , 0.001 )-- ( 16.7 , 0.001 )-- ( 16.8 , 0.001 )-- ( 16.9 , 0.001 )-- ( 17 , 0.001 )-- ( 17.1 , 0.001 )-- ( 17.2 , 0.001 )-- ( 17.3 , 0.001 )-- ( 17.4 , 0.001 )-- ( 17.5 , 0.001 )-- ( 17.6 , 0.001 )-- ( 17.7 , 0.001 )-- ( 17.8 , 0.001 )-- ( 17.9 , 0.001 )-- ( 18 , 0.001 )-- ( 18.1 , 0.001 )-- ( 18.2 , 0.001 )-- ( 18.3 , 0 )-- ( 18.4 , 0 )-- ( 18.5 , 0 )-- ( 18.6 , 0 )-- ( 18.7 , 0 )-- ( 18.8 , 0 )-- ( 18.9 , 0 )-- ( 19 , 0 )-- ( 19.1 , 0 )-- ( 19.2 , 0 )-- ( 19.3 , 0 )-- ( 19.4 , 0 )-- ( 19.5 , 0 )-- ( 19.6 , 0 )-- ( 19.7 , 0 )-- ( 19.8 , 0 )-- ( 19.9 , 0 )-- ( 20 , 0 ); 
 \end{tikzpicture}
\caption{Gamma$(1.5,0.5)$}\label{fig:H_gamma}
\end{subfigure}

\vspace{0.5cm}
 \begin{tikzpicture}
    \begin{customlegend}
    [legend entries={ $\overline H_{\theta_0}(t)$,$\overline H_{\hat{\theta}_n^{(1)}}(t)$,$\overline H_{\hat{\theta}_n^{(2)}}(t)$,$\overline H_{\hat{\theta}_n^{(4)}}(t)$,$\overline H_{\hat{\theta}_n^{({\rm GHE})}}(t)$},legend columns=-1,legend style={/tikz/every even column/.append style={column sep=0.8cm}}]   
    \addlegendimage{black} 
    \addlegendimage{red,densely dotted}    
    \addlegendimage{blue,densely dashed}    
    \addlegendimage{orange,dashdotted}    
    \addlegendimage{olive, loosely dashed}  
    \end{customlegend}
\end{tikzpicture}
\caption{Tail distribution of the customer patience, compared with its estimated counterparts. The true patience distributions $H_{\theta_0}(t)$ are hyperexponential, lognormal and Gamma. The fitted distributions $H_{\hat{\theta}_n^{(p)}}(t)$ are based on a sample of $n=10\,000$ observations with the HE MLE for $p\in\{1,2,4\}$, and the GHE heuristic with $p=10$. The arrival rate is $\lambda=1$, and the service requirements follow a Gamma distribution with parameters $(3,2)$.} \label{fig:H_MLE_tails}
\end{figure}

{The above observations are replicated for a multi-server system with $s=10$. Figures \ref{fig:H_MLE_lognormal_multi} and~\ref{fig:H_MLE_gamma_multi} plot the true and estimated patience distribution for lognormal and gamma distributions, respectively,  using the misspecified hyperexponential MLE.  As in the single-server case, the GHE approximation yields an accurate fit. }

\begin{figure}[h]
\centering
\begin{subfigure}{.48\linewidth}
\begin{tikzpicture}[xscale=0.35,yscale=3.5]
 \def\xmin{0}
 \def\xmax{20}
 \def\ymin{0}
 \def\ymax{1.05}
 \draw[->] (\xmin,\ymin) -- (\xmax,\ymin) node[right] {$t$} ;
 \draw[->] (\xmin,\ymin) -- (\xmin,\ymax)  ;
 \foreach \x in {0,5,10,20} {
 \node at (\x,\ymin-0.05) [below] {\x};
 \draw[-] (\x,\ymin) -- (\x,{\ymin-0.05});
 }
\foreach \y in {0,0.5,1} {
 \node at (\xmin,\y) [left] {\y};
 \draw[-] (\xmin,\y) -- ({\xmin-0.05},\y);
 }


\draw[black]( 0 , 1 )-- ( 0.1 , 0.997 )-- ( 0.2 , 0.983 )-- ( 0.3 , 0.956 )-- ( 0.4 , 0.922 )-- ( 0.5 , 0.883 )-- ( 0.6 , 0.844 )-- ( 0.7 , 0.804 )-- ( 0.8 , 0.765 )-- ( 0.9 , 0.728 )-- ( 1 , 0.692 )-- ( 1.1 , 0.658 )-- ( 1.2 , 0.625 )-- ( 1.3 , 0.594 )-- ( 1.4 , 0.565 )-- ( 1.5 , 0.538 )-- ( 1.6 , 0.512 )-- ( 1.7 , 0.488 )-- ( 1.8 , 0.465 )-- ( 1.9 , 0.444 )-- ( 2 , 0.424 )-- ( 2.1 , 0.405 )-- ( 2.2 , 0.387 )-- ( 2.3 , 0.37 )-- ( 2.4 , 0.354 )-- ( 2.5 , 0.339 )-- ( 2.6 , 0.325 )-- ( 2.7 , 0.311 )-- ( 2.8 , 0.298 )-- ( 2.9 , 0.286 )-- ( 3 , 0.275 )-- ( 3.1 , 0.264 )-- ( 3.2 , 0.254 )-- ( 3.3 , 0.244 )-- ( 3.4 , 0.235 )-- ( 3.5 , 0.226 )-- ( 3.6 , 0.218 )-- ( 3.7 , 0.21 )-- ( 3.8 , 0.202 )-- ( 3.9 , 0.195 )-- ( 4 , 0.188 )-- ( 4.1 , 0.182 )-- ( 4.2 , 0.175 )-- ( 4.3 , 0.169 )-- ( 4.4 , 0.163 )-- ( 4.5 , 0.158 )-- ( 4.6 , 0.153 )-- ( 4.7 , 0.148 )-- ( 4.8 , 0.143 )-- ( 4.9 , 0.138 )-- ( 5 , 0.134 )-- ( 5.1 , 0.13 )-- ( 5.2 , 0.125 )-- ( 5.3 , 0.122 )-- ( 5.4 , 0.118 )-- ( 5.5 , 0.114 )-- ( 5.6 , 0.111 )-- ( 5.7 , 0.108 )-- ( 5.8 , 0.104 )-- ( 5.9 , 0.101 )-- ( 6 , 0.098 )-- ( 6.1 , 0.095 )-- ( 6.2 , 0.093 )-- ( 6.3 , 0.09 )-- ( 6.4 , 0.088 )-- ( 6.5 , 0.085 )-- ( 6.6 , 0.083 )-- ( 6.7 , 0.08 )-- ( 6.8 , 0.078 )-- ( 6.9 , 0.076 )-- ( 7 , 0.074 )-- ( 7.1 , 0.072 )-- ( 7.2 , 0.07 )-- ( 7.3 , 0.068 )-- ( 7.4 , 0.066 )-- ( 7.5 , 0.065 )-- ( 7.6 , 0.063 )-- ( 7.7 , 0.061 )-- ( 7.8 , 0.06 )-- ( 7.9 , 0.058 )-- ( 8 , 0.057 )-- ( 8.1 , 0.056 )-- ( 8.2 , 0.054 )-- ( 8.3 , 0.053 )-- ( 8.4 , 0.052 )-- ( 8.5 , 0.05 )-- ( 8.6 , 0.049 )-- ( 8.7 , 0.048 )-- ( 8.8 , 0.047 )-- ( 8.9 , 0.046 )-- ( 9 , 0.045 )-- ( 9.1 , 0.044 )-- ( 9.2 , 0.043 )-- ( 9.3 , 0.042 )-- ( 9.4 , 0.041 )-- ( 9.5 , 0.04 )-- ( 9.6 , 0.039 )-- ( 9.7 , 0.038 )-- ( 9.8 , 0.037 )-- ( 9.9 , 0.036 )-- ( 10 , 0.036 )-- ( 10.1 , 0.035 )-- ( 10.2 , 0.034 )-- ( 10.3 , 0.033 )-- ( 10.4 , 0.033 )-- ( 10.5 , 0.032 )-- ( 10.6 , 0.031 )-- ( 10.7 , 0.031 )-- ( 10.8 , 0.03 )-- ( 10.9 , 0.029 )-- ( 11 , 0.029 )-- ( 11.1 , 0.028 )-- ( 11.2 , 0.028 )-- ( 11.3 , 0.027 )-- ( 11.4 , 0.026 )-- ( 11.5 , 0.026 )-- ( 11.6 , 0.025 )-- ( 11.7 , 0.025 )-- ( 11.8 , 0.024 )-- ( 11.9 , 0.024 )-- ( 12 , 0.024 )-- ( 12.1 , 0.023 )-- ( 12.2 , 0.023 )-- ( 12.3 , 0.022 )-- ( 12.4 , 0.022 )-- ( 12.5 , 0.021 )-- ( 12.6 , 0.021 )-- ( 12.7 , 0.021 )-- ( 12.8 , 0.02 )-- ( 12.9 , 0.02 )-- ( 13 , 0.019 )-- ( 13.1 , 0.019 )-- ( 13.2 , 0.019 )-- ( 13.3 , 0.018 )-- ( 13.4 , 0.018 )-- ( 13.5 , 0.018 )-- ( 13.6 , 0.017 )-- ( 13.7 , 0.017 )-- ( 13.8 , 0.017 )-- ( 13.9 , 0.016 )-- ( 14 , 0.016 )-- ( 14.1 , 0.016 )-- ( 14.2 , 0.016 )-- ( 14.3 , 0.015 )-- ( 14.4 , 0.015 )-- ( 14.5 , 0.015 )-- ( 14.6 , 0.014 )-- ( 14.7 , 0.014 )-- ( 14.8 , 0.014 )-- ( 14.9 , 0.014 )-- ( 15 , 0.013 )-- ( 15.1 , 0.013 )-- ( 15.2 , 0.013 )-- ( 15.3 , 0.013 )-- ( 15.4 , 0.013 )-- ( 15.5 , 0.012 )-- ( 15.6 , 0.012 )-- ( 15.7 , 0.012 )-- ( 15.8 , 0.012 )-- ( 15.9 , 0.012 )-- ( 16 , 0.011 )-- ( 16.1 , 0.011 )-- ( 16.2 , 0.011 )-- ( 16.3 , 0.011 )-- ( 16.4 , 0.011 )-- ( 16.5 , 0.011 )-- ( 16.6 , 0.01 )-- ( 16.7 , 0.01 )-- ( 16.8 , 0.01 )-- ( 16.9 , 0.01 )-- ( 17 , 0.01 )-- ( 17.1 , 0.01 )-- ( 17.2 , 0.009 )-- ( 17.3 , 0.009 )-- ( 17.4 , 0.009 )-- ( 17.5 , 0.009 )-- ( 17.6 , 0.009 )-- ( 17.7 , 0.009 )-- ( 17.8 , 0.009 )-- ( 17.9 , 0.008 )-- ( 18 , 0.008 )-- ( 18.1 , 0.008 )-- ( 18.2 , 0.008 )-- ( 18.3 , 0.008 )-- ( 18.4 , 0.008 )-- ( 18.5 , 0.008 )-- ( 18.6 , 0.008 )-- ( 18.7 , 0.007 )-- ( 18.8 , 0.007 )-- ( 18.9 , 0.007 )-- ( 19 , 0.007 )-- ( 19.1 , 0.007 )-- ( 19.2 , 0.007 )-- ( 19.3 , 0.007 )-- ( 19.4 , 0.007 )-- ( 19.5 , 0.007 )-- ( 19.6 , 0.007 )-- ( 19.7 , 0.006 )-- ( 19.8 , 0.006 )-- ( 19.9 , 0.006 )-- ( 20 , 0.006 );

 \draw[red,dotted,thick] ( 0 , 1 )-- ( 0.1 , 0.96 )-- ( 0.2 , 0.922 )-- ( 0.3 , 0.885 )-- ( 0.4 , 0.85 )-- ( 0.5 , 0.816 )-- ( 0.6 , 0.783 )-- ( 0.7 , 0.752 )-- ( 0.8 , 0.722 )-- ( 0.9 , 0.693 )-- ( 1 , 0.666 )-- ( 1.1 , 0.639 )-- ( 1.2 , 0.614 )-- ( 1.3 , 0.589 )-- ( 1.4 , 0.566 )-- ( 1.5 , 0.543 )-- ( 1.6 , 0.521 )-- ( 1.7 , 0.501 )-- ( 1.8 , 0.481 )-- ( 1.9 , 0.461 )-- ( 2 , 0.443 )-- ( 2.1 , 0.425 )-- ( 2.2 , 0.408 )-- ( 2.3 , 0.392 )-- ( 2.4 , 0.376 )-- ( 2.5 , 0.361 )-- ( 2.6 , 0.347 )-- ( 2.7 , 0.333 )-- ( 2.8 , 0.32 )-- ( 2.9 , 0.307 )-- ( 3 , 0.295 )-- ( 3.1 , 0.283 )-- ( 3.2 , 0.272 )-- ( 3.3 , 0.261 )-- ( 3.4 , 0.251 )-- ( 3.5 , 0.241 )-- ( 3.6 , 0.231 )-- ( 3.7 , 0.222 )-- ( 3.8 , 0.213 )-- ( 3.9 , 0.204 )-- ( 4 , 0.196 )-- ( 4.1 , 0.188 )-- ( 4.2 , 0.181 )-- ( 4.3 , 0.174 )-- ( 4.4 , 0.167 )-- ( 4.5 , 0.16 )-- ( 4.6 , 0.154 )-- ( 4.7 , 0.148 )-- ( 4.8 , 0.142 )-- ( 4.9 , 0.136 )-- ( 5 , 0.131 )-- ( 5.1 , 0.125 )-- ( 5.2 , 0.12 )-- ( 5.3 , 0.116 )-- ( 5.4 , 0.111 )-- ( 5.5 , 0.107 )-- ( 5.6 , 0.102 )-- ( 5.7 , 0.098 )-- ( 5.8 , 0.094 )-- ( 5.9 , 0.091 )-- ( 6 , 0.087 )-- ( 6.1 , 0.083 )-- ( 6.2 , 0.08 )-- ( 6.3 , 0.077 )-- ( 6.4 , 0.074 )-- ( 6.5 , 0.071 )-- ( 6.6 , 0.068 )-- ( 6.7 , 0.065 )-- ( 6.8 , 0.063 )-- ( 6.9 , 0.06 )-- ( 7 , 0.058 )-- ( 7.1 , 0.056 )-- ( 7.2 , 0.053 )-- ( 7.3 , 0.051 )-- ( 7.4 , 0.049 )-- ( 7.5 , 0.047 )-- ( 7.6 , 0.045 )-- ( 7.7 , 0.044 )-- ( 7.8 , 0.042 )-- ( 7.9 , 0.04 )-- ( 8 , 0.039 )-- ( 8.1 , 0.037 )-- ( 8.2 , 0.036 )-- ( 8.3 , 0.034 )-- ( 8.4 , 0.033 )-- ( 8.5 , 0.031 )-- ( 8.6 , 0.03 )-- ( 8.7 , 0.029 )-- ( 8.8 , 0.028 )-- ( 8.9 , 0.027 )-- ( 9 , 0.026 )-- ( 9.1 , 0.025 )-- ( 9.2 , 0.024 )-- ( 9.3 , 0.023 )-- ( 9.4 , 0.022 )-- ( 9.5 , 0.021 )-- ( 9.6 , 0.02 )-- ( 9.7 , 0.019 )-- ( 9.8 , 0.019 )-- ( 9.9 , 0.018 )-- ( 10 , 0.017 )-- ( 10.1 , 0.016 )-- ( 10.2 , 0.016 )-- ( 10.3 , 0.015 )-- ( 10.4 , 0.015 )-- ( 10.5 , 0.014 )-- ( 10.6 , 0.013 )-- ( 10.7 , 0.013 )-- ( 10.8 , 0.012 )-- ( 10.9 , 0.012 )-- ( 11 , 0.011 )-- ( 11.1 , 0.011 )-- ( 11.2 , 0.01 )-- ( 11.3 , 0.01 )-- ( 11.4 , 0.01 )-- ( 11.5 , 0.009 )-- ( 11.6 , 0.009 )-- ( 11.7 , 0.009 )-- ( 11.8 , 0.008 )-- ( 11.9 , 0.008 )-- ( 12 , 0.008 )-- ( 12.1 , 0.007 )-- ( 12.2 , 0.007 )-- ( 12.3 , 0.007 )-- ( 12.4 , 0.006 )-- ( 12.5 , 0.006 )-- ( 12.6 , 0.006 )-- ( 12.7 , 0.006 )-- ( 12.8 , 0.005 )-- ( 12.9 , 0.005 )-- ( 13 , 0.005 )-- ( 13.1 , 0.005 )-- ( 13.2 , 0.005 )-- ( 13.3 , 0.004 )-- ( 13.4 , 0.004 )-- ( 13.5 , 0.004 )-- ( 13.6 , 0.004 )-- ( 13.7 , 0.004 )-- ( 13.8 , 0.004 )-- ( 13.9 , 0.003 )-- ( 14 , 0.003 )-- ( 14.1 , 0.003 )-- ( 14.2 , 0.003 )-- ( 14.3 , 0.003 )-- ( 14.4 , 0.003 )-- ( 14.5 , 0.003 )-- ( 14.6 , 0.003 )-- ( 14.7 , 0.003 )-- ( 14.8 , 0.002 )-- ( 14.9 , 0.002 )-- ( 15 , 0.002 )-- ( 15.1 , 0.002 )-- ( 15.2 , 0.002 )-- ( 15.3 , 0.002 )-- ( 15.4 , 0.002 )-- ( 15.5 , 0.002 )-- ( 15.6 , 0.002 )-- ( 15.7 , 0.002 )-- ( 15.8 , 0.002 )-- ( 15.9 , 0.002 )-- ( 16 , 0.001 )-- ( 16.1 , 0.001 )-- ( 16.2 , 0.001 )-- ( 16.3 , 0.001 )-- ( 16.4 , 0.001 )-- ( 16.5 , 0.001 )-- ( 16.6 , 0.001 )-- ( 16.7 , 0.001 )-- ( 16.8 , 0.001 )-- ( 16.9 , 0.001 )-- ( 17 , 0.001 )-- ( 17.1 , 0.001 )-- ( 17.2 , 0.001 )-- ( 17.3 , 0.001 )-- ( 17.4 , 0.001 )-- ( 17.5 , 0.001 )-- ( 17.6 , 0.001 )-- ( 17.7 , 0.001 )-- ( 17.8 , 0.001 )-- ( 17.9 , 0.001 )-- ( 18 , 0.001 )-- ( 18.1 , 0.001 )-- ( 18.2 , 0.001 )-- ( 18.3 , 0.001 )-- ( 18.4 , 0.001 )-- ( 18.5 , 0.001 )-- ( 18.6 , 0.001 )-- ( 18.7 , 0 )-- ( 18.8 , 0 )-- ( 18.9 , 0 )-- ( 19 , 0 )-- ( 19.1 , 0 )-- ( 19.2 , 0 )-- ( 19.3 , 0 )-- ( 19.4 , 0 )-- ( 19.5 , 0 )-- ( 19.6 , 0 )-- ( 19.7 , 0 )-- ( 19.8 , 0 )-- ( 19.9 , 0 )-- ( 20 , 0 );

\draw[blue,densely dashed,thick] ( 0 , 1 )-- ( 0.1 , 0.951 )-- ( 0.2 , 0.905 )-- ( 0.3 , 0.862 )-- ( 0.4 , 0.821 )-- ( 0.5 , 0.782 )-- ( 0.6 , 0.745 )-- ( 0.7 , 0.711 )-- ( 0.8 , 0.678 )-- ( 0.9 , 0.647 )-- ( 1 , 0.618 )-- ( 1.1 , 0.591 )-- ( 1.2 , 0.565 )-- ( 1.3 , 0.54 )-- ( 1.4 , 0.517 )-- ( 1.5 , 0.495 )-- ( 1.6 , 0.474 )-- ( 1.7 , 0.454 )-- ( 1.8 , 0.435 )-- ( 1.9 , 0.417 )-- ( 2 , 0.4 )-- ( 2.1 , 0.384 )-- ( 2.2 , 0.368 )-- ( 2.3 , 0.354 )-- ( 2.4 , 0.34 )-- ( 2.5 , 0.327 )-- ( 2.6 , 0.314 )-- ( 2.7 , 0.302 )-- ( 2.8 , 0.291 )-- ( 2.9 , 0.28 )-- ( 3 , 0.269 )-- ( 3.1 , 0.26 )-- ( 3.2 , 0.25 )-- ( 3.3 , 0.241 )-- ( 3.4 , 0.232 )-- ( 3.5 , 0.224 )-- ( 3.6 , 0.216 )-- ( 3.7 , 0.209 )-- ( 3.8 , 0.202 )-- ( 3.9 , 0.195 )-- ( 4 , 0.188 )-- ( 4.1 , 0.182 )-- ( 4.2 , 0.176 )-- ( 4.3 , 0.17 )-- ( 4.4 , 0.164 )-- ( 4.5 , 0.159 )-- ( 4.6 , 0.154 )-- ( 4.7 , 0.149 )-- ( 4.8 , 0.144 )-- ( 4.9 , 0.14 )-- ( 5 , 0.135 )-- ( 5.1 , 0.131 )-- ( 5.2 , 0.127 )-- ( 5.3 , 0.123 )-- ( 5.4 , 0.119 )-- ( 5.5 , 0.116 )-- ( 5.6 , 0.112 )-- ( 5.7 , 0.109 )-- ( 5.8 , 0.106 )-- ( 5.9 , 0.102 )-- ( 6 , 0.099 )-- ( 6.1 , 0.096 )-- ( 6.2 , 0.094 )-- ( 6.3 , 0.091 )-- ( 6.4 , 0.088 )-- ( 6.5 , 0.086 )-- ( 6.6 , 0.083 )-- ( 6.7 , 0.081 )-- ( 6.8 , 0.079 )-- ( 6.9 , 0.076 )-- ( 7 , 0.074 )-- ( 7.1 , 0.072 )-- ( 7.2 , 0.07 )-- ( 7.3 , 0.068 )-- ( 7.4 , 0.066 )-- ( 7.5 , 0.064 )-- ( 7.6 , 0.063 )-- ( 7.7 , 0.061 )-- ( 7.8 , 0.059 )-- ( 7.9 , 0.058 )-- ( 8 , 0.056 )-- ( 8.1 , 0.055 )-- ( 8.2 , 0.053 )-- ( 8.3 , 0.052 )-- ( 8.4 , 0.05 )-- ( 8.5 , 0.049 )-- ( 8.6 , 0.048 )-- ( 8.7 , 0.046 )-- ( 8.8 , 0.045 )-- ( 8.9 , 0.044 )-- ( 9 , 0.043 )-- ( 9.1 , 0.042 )-- ( 9.2 , 0.04 )-- ( 9.3 , 0.039 )-- ( 9.4 , 0.038 )-- ( 9.5 , 0.037 )-- ( 9.6 , 0.036 )-- ( 9.7 , 0.035 )-- ( 9.8 , 0.034 )-- ( 9.9 , 0.034 )-- ( 10 , 0.033 )-- ( 10.1 , 0.032 )-- ( 10.2 , 0.031 )-- ( 10.3 , 0.03 )-- ( 10.4 , 0.029 )-- ( 10.5 , 0.029 )-- ( 10.6 , 0.028 )-- ( 10.7 , 0.027 )-- ( 10.8 , 0.027 )-- ( 10.9 , 0.026 )-- ( 11 , 0.025 )-- ( 11.1 , 0.025 )-- ( 11.2 , 0.024 )-- ( 11.3 , 0.023 )-- ( 11.4 , 0.023 )-- ( 11.5 , 0.022 )-- ( 11.6 , 0.022 )-- ( 11.7 , 0.021 )-- ( 11.8 , 0.02 )-- ( 11.9 , 0.02 )-- ( 12 , 0.019 )-- ( 12.1 , 0.019 )-- ( 12.2 , 0.018 )-- ( 12.3 , 0.018 )-- ( 12.4 , 0.018 )-- ( 12.5 , 0.017 )-- ( 12.6 , 0.017 )-- ( 12.7 , 0.016 )-- ( 12.8 , 0.016 )-- ( 12.9 , 0.015 )-- ( 13 , 0.015 )-- ( 13.1 , 0.015 )-- ( 13.2 , 0.014 )-- ( 13.3 , 0.014 )-- ( 13.4 , 0.014 )-- ( 13.5 , 0.013 )-- ( 13.6 , 0.013 )-- ( 13.7 , 0.013 )-- ( 13.8 , 0.012 )-- ( 13.9 , 0.012 )-- ( 14 , 0.012 )-- ( 14.1 , 0.011 )-- ( 14.2 , 0.011 )-- ( 14.3 , 0.011 )-- ( 14.4 , 0.01 )-- ( 14.5 , 0.01 )-- ( 14.6 , 0.01 )-- ( 14.7 , 0.01 )-- ( 14.8 , 0.009 )-- ( 14.9 , 0.009 )-- ( 15 , 0.009 )-- ( 15.1 , 0.009 )-- ( 15.2 , 0.009 )-- ( 15.3 , 0.008 )-- ( 15.4 , 0.008 )-- ( 15.5 , 0.008 )-- ( 15.6 , 0.008 )-- ( 15.7 , 0.008 )-- ( 15.8 , 0.007 )-- ( 15.9 , 0.007 )-- ( 16 , 0.007 )-- ( 16.1 , 0.007 )-- ( 16.2 , 0.007 )-- ( 16.3 , 0.006 )-- ( 16.4 , 0.006 )-- ( 16.5 , 0.006 )-- ( 16.6 , 0.006 )-- ( 16.7 , 0.006 )-- ( 16.8 , 0.006 )-- ( 16.9 , 0.006 )-- ( 17 , 0.005 )-- ( 17.1 , 0.005 )-- ( 17.2 , 0.005 )-- ( 17.3 , 0.005 )-- ( 17.4 , 0.005 )-- ( 17.5 , 0.005 )-- ( 17.6 , 0.005 )-- ( 17.7 , 0.005 )-- ( 17.8 , 0.004 )-- ( 17.9 , 0.004 )-- ( 18 , 0.004 )-- ( 18.1 , 0.004 )-- ( 18.2 , 0.004 )-- ( 18.3 , 0.004 )-- ( 18.4 , 0.004 )-- ( 18.5 , 0.004 )-- ( 18.6 , 0.004 )-- ( 18.7 , 0.004 )-- ( 18.8 , 0.003 )-- ( 18.9 , 0.003 )-- ( 19 , 0.003 )-- ( 19.1 , 0.003 )-- ( 19.2 , 0.003 )-- ( 19.3 , 0.003 )-- ( 19.4 , 0.003 )-- ( 19.5 , 0.003 )-- ( 19.6 , 0.003 )-- ( 19.7 , 0.003 )-- ( 19.8 , 0.003 )-- ( 19.9 , 0.003 )-- ( 20 , 0.003 );

\draw[olive,dashdotted,thick] ( 0 , 1 )-- ( 0.1 , 0.96 )-- ( 0.2 , 0.921 )-- ( 0.3 , 0.884 )-- ( 0.4 , 0.848 )-- ( 0.5 , 0.814 )-- ( 0.6 , 0.78 )-- ( 0.7 , 0.748 )-- ( 0.8 , 0.716 )-- ( 0.9 , 0.686 )-- ( 1 , 0.656 )-- ( 1.1 , 0.628 )-- ( 1.2 , 0.601 )-- ( 1.3 , 0.575 )-- ( 1.4 , 0.55 )-- ( 1.5 , 0.526 )-- ( 1.6 , 0.502 )-- ( 1.7 , 0.48 )-- ( 1.8 , 0.459 )-- ( 1.9 , 0.439 )-- ( 2 , 0.42 )-- ( 2.1 , 0.402 )-- ( 2.2 , 0.385 )-- ( 2.3 , 0.368 )-- ( 2.4 , 0.352 )-- ( 2.5 , 0.337 )-- ( 2.6 , 0.323 )-- ( 2.7 , 0.31 )-- ( 2.8 , 0.297 )-- ( 2.9 , 0.285 )-- ( 3 , 0.274 )-- ( 3.1 , 0.263 )-- ( 3.2 , 0.253 )-- ( 3.3 , 0.243 )-- ( 3.4 , 0.233 )-- ( 3.5 , 0.225 )-- ( 3.6 , 0.216 )-- ( 3.7 , 0.208 )-- ( 3.8 , 0.201 )-- ( 3.9 , 0.193 )-- ( 4 , 0.187 )-- ( 4.1 , 0.18 )-- ( 4.2 , 0.174 )-- ( 4.3 , 0.168 )-- ( 4.4 , 0.162 )-- ( 4.5 , 0.157 )-- ( 4.6 , 0.152 )-- ( 4.7 , 0.147 )-- ( 4.8 , 0.142 )-- ( 4.9 , 0.138 )-- ( 5 , 0.133 )-- ( 5.1 , 0.129 )-- ( 5.2 , 0.125 )-- ( 5.3 , 0.121 )-- ( 5.4 , 0.118 )-- ( 5.5 , 0.114 )-- ( 5.6 , 0.111 )-- ( 5.7 , 0.107 )-- ( 5.8 , 0.104 )-- ( 5.9 , 0.101 )-- ( 6 , 0.098 )-- ( 6.1 , 0.096 )-- ( 6.2 , 0.093 )-- ( 6.3 , 0.09 )-- ( 6.4 , 0.088 )-- ( 6.5 , 0.085 )-- ( 6.6 , 0.083 )-- ( 6.7 , 0.081 )-- ( 6.8 , 0.078 )-- ( 6.9 , 0.076 )-- ( 7 , 0.074 )-- ( 7.1 , 0.072 )-- ( 7.2 , 0.07 )-- ( 7.3 , 0.068 )-- ( 7.4 , 0.067 )-- ( 7.5 , 0.065 )-- ( 7.6 , 0.063 )-- ( 7.7 , 0.061 )-- ( 7.8 , 0.06 )-- ( 7.9 , 0.058 )-- ( 8 , 0.057 )-- ( 8.1 , 0.055 )-- ( 8.2 , 0.054 )-- ( 8.3 , 0.052 )-- ( 8.4 , 0.051 )-- ( 8.5 , 0.05 )-- ( 8.6 , 0.048 )-- ( 8.7 , 0.047 )-- ( 8.8 , 0.046 )-- ( 8.9 , 0.045 )-- ( 9 , 0.044 )-- ( 9.1 , 0.042 )-- ( 9.2 , 0.041 )-- ( 9.3 , 0.04 )-- ( 9.4 , 0.039 )-- ( 9.5 , 0.038 )-- ( 9.6 , 0.037 )-- ( 9.7 , 0.036 )-- ( 9.8 , 0.035 )-- ( 9.9 , 0.034 )-- ( 10 , 0.034 )-- ( 10.1 , 0.033 )-- ( 10.2 , 0.032 )-- ( 10.3 , 0.031 )-- ( 10.4 , 0.03 )-- ( 10.5 , 0.03 )-- ( 10.6 , 0.029 )-- ( 10.7 , 0.028 )-- ( 10.8 , 0.027 )-- ( 10.9 , 0.027 )-- ( 11 , 0.026 )-- ( 11.1 , 0.025 )-- ( 11.2 , 0.025 )-- ( 11.3 , 0.024 )-- ( 11.4 , 0.023 )-- ( 11.5 , 0.023 )-- ( 11.6 , 0.022 )-- ( 11.7 , 0.022 )-- ( 11.8 , 0.021 )-- ( 11.9 , 0.021 )-- ( 12 , 0.02 )-- ( 12.1 , 0.02 )-- ( 12.2 , 0.019 )-- ( 12.3 , 0.019 )-- ( 12.4 , 0.018 )-- ( 12.5 , 0.018 )-- ( 12.6 , 0.017 )-- ( 12.7 , 0.017 )-- ( 12.8 , 0.016 )-- ( 12.9 , 0.016 )-- ( 13 , 0.016 )-- ( 13.1 , 0.015 )-- ( 13.2 , 0.015 )-- ( 13.3 , 0.014 )-- ( 13.4 , 0.014 )-- ( 13.5 , 0.014 )-- ( 13.6 , 0.013 )-- ( 13.7 , 0.013 )-- ( 13.8 , 0.013 )-- ( 13.9 , 0.012 )-- ( 14 , 0.012 )-- ( 14.1 , 0.012 )-- ( 14.2 , 0.011 )-- ( 14.3 , 0.011 )-- ( 14.4 , 0.011 )-- ( 14.5 , 0.011 )-- ( 14.6 , 0.01 )-- ( 14.7 , 0.01 )-- ( 14.8 , 0.01 )-- ( 14.9 , 0.01 )-- ( 15 , 0.009 )-- ( 15.1 , 0.009 )-- ( 15.2 , 0.009 )-- ( 15.3 , 0.009 )-- ( 15.4 , 0.008 )-- ( 15.5 , 0.008 )-- ( 15.6 , 0.008 )-- ( 15.7 , 0.008 )-- ( 15.8 , 0.008 )-- ( 15.9 , 0.007 )-- ( 16 , 0.007 )-- ( 16.1 , 0.007 )-- ( 16.2 , 0.007 )-- ( 16.3 , 0.007 )-- ( 16.4 , 0.006 )-- ( 16.5 , 0.006 )-- ( 16.6 , 0.006 )-- ( 16.7 , 0.006 )-- ( 16.8 , 0.006 )-- ( 16.9 , 0.006 )-- ( 17 , 0.006 )-- ( 17.1 , 0.005 )-- ( 17.2 , 0.005 )-- ( 17.3 , 0.005 )-- ( 17.4 , 0.005 )-- ( 17.5 , 0.005 )-- ( 17.6 , 0.005 )-- ( 17.7 , 0.005 )-- ( 17.8 , 0.005 )-- ( 17.9 , 0.004 )-- ( 18 , 0.004 )-- ( 18.1 , 0.004 )-- ( 18.2 , 0.004 )-- ( 18.3 , 0.004 )-- ( 18.4 , 0.004 )-- ( 18.5 , 0.004 )-- ( 18.6 , 0.004 )-- ( 18.7 , 0.004 )-- ( 18.8 , 0.004 )-- ( 18.9 , 0.003 )-- ( 19 , 0.003 )-- ( 19.1 , 0.003 )-- ( 19.2 , 0.003 )-- ( 19.3 , 0.003 )-- ( 19.4 , 0.003 )-- ( 19.5 , 0.003 )-- ( 19.6 , 0.003 )-- ( 19.7 , 0.003 )-- ( 19.8 , 0.003 )-- ( 19.9 , 0.003 )-- ( 20 , 0.003 ); 
\end{tikzpicture}
\caption{lognormal$(0.5,1)$}\label{fig:H_MLE_lognormal_multi}
\end{subfigure}
\begin{subfigure}{.48\linewidth}
\begin{tikzpicture}[xscale=0.35,yscale=3.5]
\def\xmin{0}
 \def\xmax{20}
 \def\ymin{0}
 \def\ymax{1.05}
 \draw[->] (\xmin,\ymin) -- (\xmax,\ymin) node[right] {$t$} ;
 \draw[->] (\xmin,\ymin) -- (\xmin,\ymax)  ;
 \foreach \x in {0,5,10,20} {
 \node at (\x,\ymin-0.05) [below] {\x};
 \draw[-] (\x,\ymin) -- (\x,{\ymin-0.05});
 }
\foreach \y in {0,0.5,1} {
 \node at (\xmin,\y) [left] {\y};
 \draw[-] (\xmin,\y) -- ({\xmin-0.05},\y);
 }


\draw[black] ( 0 , 1 )-- ( 0.1 , 0.992 )-- ( 0.2 , 0.977 )-- ( 0.3 , 0.96 )-- ( 0.4 , 0.94 )-- ( 0.5 , 0.919 )-- ( 0.6 , 0.896 )-- ( 0.7 , 0.873 )-- ( 0.8 , 0.849 )-- ( 0.9 , 0.825 )-- ( 1 , 0.801 )-- ( 1.1 , 0.777 )-- ( 1.2 , 0.753 )-- ( 1.3 , 0.729 )-- ( 1.4 , 0.706 )-- ( 1.5 , 0.682 )-- ( 1.6 , 0.659 )-- ( 1.7 , 0.637 )-- ( 1.8 , 0.615 )-- ( 1.9 , 0.593 )-- ( 2 , 0.572 )-- ( 2.1 , 0.552 )-- ( 2.2 , 0.532 )-- ( 2.3 , 0.512 )-- ( 2.4 , 0.493 )-- ( 2.5 , 0.475 )-- ( 2.6 , 0.457 )-- ( 2.7 , 0.44 )-- ( 2.8 , 0.423 )-- ( 2.9 , 0.407 )-- ( 3 , 0.391 )-- ( 3.1 , 0.376 )-- ( 3.2 , 0.361 )-- ( 3.3 , 0.347 )-- ( 3.4 , 0.333 )-- ( 3.5 , 0.32 )-- ( 3.6 , 0.307 )-- ( 3.7 , 0.295 )-- ( 3.8 , 0.283 )-- ( 3.9 , 0.272 )-- ( 4 , 0.261 )-- ( 4.1 , 0.25 )-- ( 4.2 , 0.24 )-- ( 4.3 , 0.231 )-- ( 4.4 , 0.221 )-- ( 4.5 , 0.212 )-- ( 4.6 , 0.203 )-- ( 4.7 , 0.195 )-- ( 4.8 , 0.187 )-- ( 4.9 , 0.179 )-- ( 5 , 0.171 )-- ( 5.1 , 0.164 )-- ( 5.2 , 0.157 )-- ( 5.3 , 0.151 )-- ( 5.4 , 0.144 )-- ( 5.5 , 0.138 )-- ( 5.6 , 0.132 )-- ( 5.7 , 0.127 )-- ( 5.8 , 0.121 )-- ( 5.9 , 0.116 )-- ( 6 , 0.111 )-- ( 6.1 , 0.107 )-- ( 6.2 , 0.102 )-- ( 6.3 , 0.098 )-- ( 6.4 , 0.093 )-- ( 6.5 , 0.09 )-- ( 6.6 , 0.086 )-- ( 6.7 , 0.082 )-- ( 6.8 , 0.079 )-- ( 6.9 , 0.075 )-- ( 7 , 0.072 )-- ( 7.1 , 0.069 )-- ( 7.2 , 0.066 )-- ( 7.3 , 0.063 )-- ( 7.4 , 0.06 )-- ( 7.5 , 0.057 )-- ( 7.6 , 0.055 )-- ( 7.7 , 0.053 )-- ( 7.8 , 0.05 )-- ( 7.9 , 0.048 )-- ( 8 , 0.046 )-- ( 8.1 , 0.044 )-- ( 8.2 , 0.042 )-- ( 8.3 , 0.04 )-- ( 8.4 , 0.038 )-- ( 8.5 , 0.037 )-- ( 8.6 , 0.035 )-- ( 8.7 , 0.034 )-- ( 8.8 , 0.032 )-- ( 8.9 , 0.031 )-- ( 9 , 0.029 )-- ( 9.1 , 0.028 )-- ( 9.2 , 0.027 )-- ( 9.3 , 0.026 )-- ( 9.4 , 0.025 )-- ( 9.5 , 0.023 )-- ( 9.6 , 0.022 )-- ( 9.7 , 0.021 )-- ( 9.8 , 0.02 )-- ( 9.9 , 0.019 )-- ( 10 , 0.019 )-- ( 10.1 , 0.018 )-- ( 10.2 , 0.017 )-- ( 10.3 , 0.016 )-- ( 10.4 , 0.015 )-- ( 10.5 , 0.015 )-- ( 10.6 , 0.014 )-- ( 10.7 , 0.014 )-- ( 10.8 , 0.013 )-- ( 10.9 , 0.012 )-- ( 11 , 0.012 )-- ( 11.1 , 0.011 )-- ( 11.2 , 0.011 )-- ( 11.3 , 0.01 )-- ( 11.4 , 0.01 )-- ( 11.5 , 0.009 )-- ( 11.6 , 0.009 )-- ( 11.7 , 0.009 )-- ( 11.8 , 0.008 )-- ( 11.9 , 0.008 )-- ( 12 , 0.008 )-- ( 12.1 , 0.007 )-- ( 12.2 , 0.007 )-- ( 12.3 , 0.007 )-- ( 12.4 , 0.006 )-- ( 12.5 , 0.006 )-- ( 12.6 , 0.006 )-- ( 12.7 , 0.005 )-- ( 12.8 , 0.005 )-- ( 12.9 , 0.005 )-- ( 13 , 0.005 )-- ( 13.1 , 0.005 )-- ( 13.2 , 0.004 )-- ( 13.3 , 0.004 )-- ( 13.4 , 0.004 )-- ( 13.5 , 0.004 )-- ( 13.6 , 0.004 )-- ( 13.7 , 0.003 )-- ( 13.8 , 0.003 )-- ( 13.9 , 0.003 )-- ( 14 , 0.003 )-- ( 14.1 , 0.003 )-- ( 14.2 , 0.003 )-- ( 14.3 , 0.002 )-- ( 14.4 , 0.002 )-- ( 14.5 , 0.002 )-- ( 14.6 , 0.002 )-- ( 14.7 , 0.002 )-- ( 14.8 , 0.002 )-- ( 14.9 , 0.002 )-- ( 15 , 0.002 )-- ( 15.1 , 0.002 )-- ( 15.2 , 0.002 )-- ( 15.3 , 0.002 )-- ( 15.4 , 0.001 )-- ( 15.5 , 0.001 )-- ( 15.6 , 0.001 )-- ( 15.7 , 0.001 )-- ( 15.8 , 0.001 )-- ( 15.9 , 0.001 )-- ( 16 , 0.001 )-- ( 16.1 , 0.001 )-- ( 16.2 , 0.001 )-- ( 16.3 , 0.001 )-- ( 16.4 , 0.001 )-- ( 16.5 , 0.001 )-- ( 16.6 , 0.001 )-- ( 16.7 , 0.001 )-- ( 16.8 , 0.001 )-- ( 16.9 , 0.001 )-- ( 17 , 0.001 )-- ( 17.1 , 0.001 )-- ( 17.2 , 0.001 )-- ( 17.3 , 0.001 )-- ( 17.4 , 0.001 )-- ( 17.5 , 0.001 )-- ( 17.6 , 0.001 )-- ( 17.7 , 0.001 )-- ( 17.8 , 0 )-- ( 17.9 , 0 )-- ( 18 , 0 )-- ( 18.1 , 0 )-- ( 18.2 , 0 )-- ( 18.3 , 0 )-- ( 18.4 , 0 )-- ( 18.5 , 0 )-- ( 18.6 , 0 )-- ( 18.7 , 0 )-- ( 18.8 , 0 )-- ( 18.9 , 0 )-- ( 19 , 0 )-- ( 19.1 , 0 )-- ( 19.2 , 0 )-- ( 19.3 , 0 )-- ( 19.4 , 0 )-- ( 19.5 , 0 )-- ( 19.6 , 0 )-- ( 19.7 , 0 )-- ( 19.8 , 0 )-- ( 19.9 , 0 )-- ( 20 , 0 );

 \draw[red,dotted,thick] ( 0 , 1 )-- ( 0.1 , 0.974 )-- ( 0.2 , 0.948 )-- ( 0.3 , 0.923 )-- ( 0.4 , 0.898 )-- ( 0.5 , 0.875 )-- ( 0.6 , 0.852 )-- ( 0.7 , 0.829 )-- ( 0.8 , 0.807 )-- ( 0.9 , 0.786 )-- ( 1 , 0.765 )-- ( 1.1 , 0.745 )-- ( 1.2 , 0.725 )-- ( 1.3 , 0.706 )-- ( 1.4 , 0.688 )-- ( 1.5 , 0.669 )-- ( 1.6 , 0.652 )-- ( 1.7 , 0.635 )-- ( 1.8 , 0.618 )-- ( 1.9 , 0.601 )-- ( 2 , 0.586 )-- ( 2.1 , 0.57 )-- ( 2.2 , 0.555 )-- ( 2.3 , 0.54 )-- ( 2.4 , 0.526 )-- ( 2.5 , 0.512 )-- ( 2.6 , 0.499 )-- ( 2.7 , 0.486 )-- ( 2.8 , 0.473 )-- ( 2.9 , 0.46 )-- ( 3 , 0.448 )-- ( 3.1 , 0.436 )-- ( 3.2 , 0.425 )-- ( 3.3 , 0.414 )-- ( 3.4 , 0.403 )-- ( 3.5 , 0.392 )-- ( 3.6 , 0.382 )-- ( 3.7 , 0.372 )-- ( 3.8 , 0.362 )-- ( 3.9 , 0.352 )-- ( 4 , 0.343 )-- ( 4.1 , 0.334 )-- ( 4.2 , 0.325 )-- ( 4.3 , 0.316 )-- ( 4.4 , 0.308 )-- ( 4.5 , 0.3 )-- ( 4.6 , 0.292 )-- ( 4.7 , 0.284 )-- ( 4.8 , 0.277 )-- ( 4.9 , 0.27 )-- ( 5 , 0.262 )-- ( 5.1 , 0.255 )-- ( 5.2 , 0.249 )-- ( 5.3 , 0.242 )-- ( 5.4 , 0.236 )-- ( 5.5 , 0.23 )-- ( 5.6 , 0.223 )-- ( 5.7 , 0.218 )-- ( 5.8 , 0.212 )-- ( 5.9 , 0.206 )-- ( 6 , 0.201 )-- ( 6.1 , 0.195 )-- ( 6.2 , 0.19 )-- ( 6.3 , 0.185 )-- ( 6.4 , 0.18 )-- ( 6.5 , 0.176 )-- ( 6.6 , 0.171 )-- ( 6.7 , 0.166 )-- ( 6.8 , 0.162 )-- ( 6.9 , 0.158 )-- ( 7 , 0.154 )-- ( 7.1 , 0.15 )-- ( 7.2 , 0.146 )-- ( 7.3 , 0.142 )-- ( 7.4 , 0.138 )-- ( 7.5 , 0.134 )-- ( 7.6 , 0.131 )-- ( 7.7 , 0.127 )-- ( 7.8 , 0.124 )-- ( 7.9 , 0.121 )-- ( 8 , 0.118 )-- ( 8.1 , 0.114 )-- ( 8.2 , 0.111 )-- ( 8.3 , 0.109 )-- ( 8.4 , 0.106 )-- ( 8.5 , 0.103 )-- ( 8.6 , 0.1 )-- ( 8.7 , 0.097 )-- ( 8.8 , 0.095 )-- ( 8.9 , 0.092 )-- ( 9 , 0.09 )-- ( 9.1 , 0.088 )-- ( 9.2 , 0.085 )-- ( 9.3 , 0.083 )-- ( 9.4 , 0.081 )-- ( 9.5 , 0.079 )-- ( 9.6 , 0.077 )-- ( 9.7 , 0.075 )-- ( 9.8 , 0.073 )-- ( 9.9 , 0.071 )-- ( 10 , 0.069 )-- ( 10.1 , 0.067 )-- ( 10.2 , 0.065 )-- ( 10.3 , 0.064 )-- ( 10.4 , 0.062 )-- ( 10.5 , 0.06 )-- ( 10.6 , 0.059 )-- ( 10.7 , 0.057 )-- ( 10.8 , 0.056 )-- ( 10.9 , 0.054 )-- ( 11 , 0.053 )-- ( 11.1 , 0.051 )-- ( 11.2 , 0.05 )-- ( 11.3 , 0.049 )-- ( 11.4 , 0.047 )-- ( 11.5 , 0.046 )-- ( 11.6 , 0.045 )-- ( 11.7 , 0.044 )-- ( 11.8 , 0.043 )-- ( 11.9 , 0.041 )-- ( 12 , 0.04 )-- ( 12.1 , 0.039 )-- ( 12.2 , 0.038 )-- ( 12.3 , 0.037 )-- ( 12.4 , 0.036 )-- ( 12.5 , 0.035 )-- ( 12.6 , 0.034 )-- ( 12.7 , 0.033 )-- ( 12.8 , 0.033 )-- ( 12.9 , 0.032 )-- ( 13 , 0.031 )-- ( 13.1 , 0.03 )-- ( 13.2 , 0.029 )-- ( 13.3 , 0.028 )-- ( 13.4 , 0.028 )-- ( 13.5 , 0.027 )-- ( 13.6 , 0.026 )-- ( 13.7 , 0.026 )-- ( 13.8 , 0.025 )-- ( 13.9 , 0.024 )-- ( 14 , 0.024 )-- ( 14.1 , 0.023 )-- ( 14.2 , 0.022 )-- ( 14.3 , 0.022 )-- ( 14.4 , 0.021 )-- ( 14.5 , 0.021 )-- ( 14.6 , 0.02 )-- ( 14.7 , 0.02 )-- ( 14.8 , 0.019 )-- ( 14.9 , 0.019 )-- ( 15 , 0.018 )-- ( 15.1 , 0.018 )-- ( 15.2 , 0.017 )-- ( 15.3 , 0.017 )-- ( 15.4 , 0.016 )-- ( 15.5 , 0.016 )-- ( 15.6 , 0.015 )-- ( 15.7 , 0.015 )-- ( 15.8 , 0.015 )-- ( 15.9 , 0.014 )-- ( 16 , 0.014 )-- ( 16.1 , 0.013 )-- ( 16.2 , 0.013 )-- ( 16.3 , 0.013 )-- ( 16.4 , 0.012 )-- ( 16.5 , 0.012 )-- ( 16.6 , 0.012 )-- ( 16.7 , 0.011 )-- ( 16.8 , 0.011 )-- ( 16.9 , 0.011 )-- ( 17 , 0.011 )-- ( 17.1 , 0.01 )-- ( 17.2 , 0.01 )-- ( 17.3 , 0.01 )-- ( 17.4 , 0.01 )-- ( 17.5 , 0.009 )-- ( 17.6 , 0.009 )-- ( 17.7 , 0.009 )-- ( 17.8 , 0.009 )-- ( 17.9 , 0.008 )-- ( 18 , 0.008 )-- ( 18.1 , 0.008 )-- ( 18.2 , 0.008 )-- ( 18.3 , 0.007 )-- ( 18.4 , 0.007 )-- ( 18.5 , 0.007 )-- ( 18.6 , 0.007 )-- ( 18.7 , 0.007 )-- ( 18.8 , 0.007 )-- ( 18.9 , 0.006 )-- ( 19 , 0.006 )-- ( 19.1 , 0.006 )-- ( 19.2 , 0.006 )-- ( 19.3 , 0.006 )-- ( 19.4 , 0.006 )-- ( 19.5 , 0.005 )-- ( 19.6 , 0.005 )-- ( 19.7 , 0.005 )-- ( 19.8 , 0.005 )-- ( 19.9 , 0.005 )-- ( 20 , 0.005 );

\draw[blue,densely dashed,thick] ( 0 , 1 )-- ( 0.1 , 0.973 )-- ( 0.2 , 0.947 )-- ( 0.3 , 0.922 )-- ( 0.4 , 0.898 )-- ( 0.5 , 0.874 )-- ( 0.6 , 0.851 )-- ( 0.7 , 0.828 )-- ( 0.8 , 0.806 )-- ( 0.9 , 0.785 )-- ( 1 , 0.765 )-- ( 1.1 , 0.745 )-- ( 1.2 , 0.725 )-- ( 1.3 , 0.706 )-- ( 1.4 , 0.688 )-- ( 1.5 , 0.67 )-- ( 1.6 , 0.653 )-- ( 1.7 , 0.636 )-- ( 1.8 , 0.619 )-- ( 1.9 , 0.603 )-- ( 2 , 0.588 )-- ( 2.1 , 0.573 )-- ( 2.2 , 0.558 )-- ( 2.3 , 0.544 )-- ( 2.4 , 0.53 )-- ( 2.5 , 0.517 )-- ( 2.6 , 0.504 )-- ( 2.7 , 0.491 )-- ( 2.8 , 0.478 )-- ( 2.9 , 0.466 )-- ( 3 , 0.455 )-- ( 3.1 , 0.443 )-- ( 3.2 , 0.432 )-- ( 3.3 , 0.421 )-- ( 3.4 , 0.411 )-- ( 3.5 , 0.4 )-- ( 3.6 , 0.39 )-- ( 3.7 , 0.381 )-- ( 3.8 , 0.371 )-- ( 3.9 , 0.362 )-- ( 4 , 0.353 )-- ( 4.1 , 0.344 )-- ( 4.2 , 0.336 )-- ( 4.3 , 0.328 )-- ( 4.4 , 0.32 )-- ( 4.5 , 0.312 )-- ( 4.6 , 0.304 )-- ( 4.7 , 0.297 )-- ( 4.8 , 0.29 )-- ( 4.9 , 0.283 )-- ( 5 , 0.276 )-- ( 5.1 , 0.269 )-- ( 5.2 , 0.263 )-- ( 5.3 , 0.256 )-- ( 5.4 , 0.25 )-- ( 5.5 , 0.244 )-- ( 5.6 , 0.238 )-- ( 5.7 , 0.232 )-- ( 5.8 , 0.227 )-- ( 5.9 , 0.221 )-- ( 6 , 0.216 )-- ( 6.1 , 0.211 )-- ( 6.2 , 0.206 )-- ( 6.3 , 0.201 )-- ( 6.4 , 0.196 )-- ( 6.5 , 0.192 )-- ( 6.6 , 0.187 )-- ( 6.7 , 0.183 )-- ( 6.8 , 0.178 )-- ( 6.9 , 0.174 )-- ( 7 , 0.17 )-- ( 7.1 , 0.166 )-- ( 7.2 , 0.162 )-- ( 7.3 , 0.158 )-- ( 7.4 , 0.155 )-- ( 7.5 , 0.151 )-- ( 7.6 , 0.148 )-- ( 7.7 , 0.144 )-- ( 7.8 , 0.141 )-- ( 7.9 , 0.138 )-- ( 8 , 0.134 )-- ( 8.1 , 0.131 )-- ( 8.2 , 0.128 )-- ( 8.3 , 0.125 )-- ( 8.4 , 0.122 )-- ( 8.5 , 0.12 )-- ( 8.6 , 0.117 )-- ( 8.7 , 0.114 )-- ( 8.8 , 0.112 )-- ( 8.9 , 0.109 )-- ( 9 , 0.107 )-- ( 9.1 , 0.104 )-- ( 9.2 , 0.102 )-- ( 9.3 , 0.099 )-- ( 9.4 , 0.097 )-- ( 9.5 , 0.095 )-- ( 9.6 , 0.093 )-- ( 9.7 , 0.091 )-- ( 9.8 , 0.089 )-- ( 9.9 , 0.087 )-- ( 10 , 0.085 )-- ( 10.1 , 0.083 )-- ( 10.2 , 0.081 )-- ( 10.3 , 0.079 )-- ( 10.4 , 0.077 )-- ( 10.5 , 0.076 )-- ( 10.6 , 0.074 )-- ( 10.7 , 0.072 )-- ( 10.8 , 0.071 )-- ( 10.9 , 0.069 )-- ( 11 , 0.068 )-- ( 11.1 , 0.066 )-- ( 11.2 , 0.065 )-- ( 11.3 , 0.063 )-- ( 11.4 , 0.062 )-- ( 11.5 , 0.06 )-- ( 11.6 , 0.059 )-- ( 11.7 , 0.058 )-- ( 11.8 , 0.056 )-- ( 11.9 , 0.055 )-- ( 12 , 0.054 )-- ( 12.1 , 0.053 )-- ( 12.2 , 0.052 )-- ( 12.3 , 0.051 )-- ( 12.4 , 0.049 )-- ( 12.5 , 0.048 )-- ( 12.6 , 0.047 )-- ( 12.7 , 0.046 )-- ( 12.8 , 0.045 )-- ( 12.9 , 0.044 )-- ( 13 , 0.043 )-- ( 13.1 , 0.042 )-- ( 13.2 , 0.041 )-- ( 13.3 , 0.04 )-- ( 13.4 , 0.04 )-- ( 13.5 , 0.039 )-- ( 13.6 , 0.038 )-- ( 13.7 , 0.037 )-- ( 13.8 , 0.036 )-- ( 13.9 , 0.035 )-- ( 14 , 0.035 )-- ( 14.1 , 0.034 )-- ( 14.2 , 0.033 )-- ( 14.3 , 0.033 )-- ( 14.4 , 0.032 )-- ( 14.5 , 0.031 )-- ( 14.6 , 0.03 )-- ( 14.7 , 0.03 )-- ( 14.8 , 0.029 )-- ( 14.9 , 0.029 )-- ( 15 , 0.028 )-- ( 15.1 , 0.027 )-- ( 15.2 , 0.027 )-- ( 15.3 , 0.026 )-- ( 15.4 , 0.026 )-- ( 15.5 , 0.025 )-- ( 15.6 , 0.025 )-- ( 15.7 , 0.024 )-- ( 15.8 , 0.023 )-- ( 15.9 , 0.023 )-- ( 16 , 0.023 )-- ( 16.1 , 0.022 )-- ( 16.2 , 0.022 )-- ( 16.3 , 0.021 )-- ( 16.4 , 0.021 )-- ( 16.5 , 0.02 )-- ( 16.6 , 0.02 )-- ( 16.7 , 0.019 )-- ( 16.8 , 0.019 )-- ( 16.9 , 0.019 )-- ( 17 , 0.018 )-- ( 17.1 , 0.018 )-- ( 17.2 , 0.017 )-- ( 17.3 , 0.017 )-- ( 17.4 , 0.017 )-- ( 17.5 , 0.016 )-- ( 17.6 , 0.016 )-- ( 17.7 , 0.016 )-- ( 17.8 , 0.015 )-- ( 17.9 , 0.015 )-- ( 18 , 0.015 )-- ( 18.1 , 0.014 )-- ( 18.2 , 0.014 )-- ( 18.3 , 0.014 )-- ( 18.4 , 0.013 )-- ( 18.5 , 0.013 )-- ( 18.6 , 0.013 )-- ( 18.7 , 0.013 )-- ( 18.8 , 0.012 )-- ( 18.9 , 0.012 )-- ( 19 , 0.012 )-- ( 19.1 , 0.012 )-- ( 19.2 , 0.011 )-- ( 19.3 , 0.011 )-- ( 19.4 , 0.011 )-- ( 19.5 , 0.011 )-- ( 19.6 , 0.01 )-- ( 19.7 , 0.01 )-- ( 19.8 , 0.01 )-- ( 19.9 , 0.01 )-- ( 20 , 0.01 );

\draw[olive,dashdotted,thick]  ( 0 , 1 )-- ( 0.1 , 0.99 )-- ( 0.2 , 0.978 )-- ( 0.3 , 0.962 )-- ( 0.4 , 0.944 )-- ( 0.5 , 0.925 )-- ( 0.6 , 0.904 )-- ( 0.7 , 0.882 )-- ( 0.8 , 0.859 )-- ( 0.9 , 0.836 )-- ( 1 , 0.811 )-- ( 1.1 , 0.787 )-- ( 1.2 , 0.763 )-- ( 1.3 , 0.738 )-- ( 1.4 , 0.713 )-- ( 1.5 , 0.689 )-- ( 1.6 , 0.665 )-- ( 1.7 , 0.642 )-- ( 1.8 , 0.618 )-- ( 1.9 , 0.595 )-- ( 2 , 0.573 )-- ( 2.1 , 0.551 )-- ( 2.2 , 0.53 )-- ( 2.3 , 0.509 )-- ( 2.4 , 0.489 )-- ( 2.5 , 0.47 )-- ( 2.6 , 0.451 )-- ( 2.7 , 0.433 )-- ( 2.8 , 0.415 )-- ( 2.9 , 0.398 )-- ( 3 , 0.381 )-- ( 3.1 , 0.365 )-- ( 3.2 , 0.35 )-- ( 3.3 , 0.335 )-- ( 3.4 , 0.321 )-- ( 3.5 , 0.307 )-- ( 3.6 , 0.294 )-- ( 3.7 , 0.281 )-- ( 3.8 , 0.269 )-- ( 3.9 , 0.257 )-- ( 4 , 0.246 )-- ( 4.1 , 0.235 )-- ( 4.2 , 0.225 )-- ( 4.3 , 0.215 )-- ( 4.4 , 0.205 )-- ( 4.5 , 0.196 )-- ( 4.6 , 0.188 )-- ( 4.7 , 0.179 )-- ( 4.8 , 0.171 )-- ( 4.9 , 0.164 )-- ( 5 , 0.156 )-- ( 5.1 , 0.149 )-- ( 5.2 , 0.143 )-- ( 5.3 , 0.136 )-- ( 5.4 , 0.13 )-- ( 5.5 , 0.124 )-- ( 5.6 , 0.118 )-- ( 5.7 , 0.113 )-- ( 5.8 , 0.108 )-- ( 5.9 , 0.103 )-- ( 6 , 0.098 )-- ( 6.1 , 0.094 )-- ( 6.2 , 0.09 )-- ( 6.3 , 0.086 )-- ( 6.4 , 0.082 )-- ( 6.5 , 0.078 )-- ( 6.6 , 0.074 )-- ( 6.7 , 0.071 )-- ( 6.8 , 0.068 )-- ( 6.9 , 0.065 )-- ( 7 , 0.062 )-- ( 7.1 , 0.059 )-- ( 7.2 , 0.056 )-- ( 7.3 , 0.054 )-- ( 7.4 , 0.051 )-- ( 7.5 , 0.049 )-- ( 7.6 , 0.046 )-- ( 7.7 , 0.044 )-- ( 7.8 , 0.042 )-- ( 7.9 , 0.04 )-- ( 8 , 0.039 )-- ( 8.1 , 0.037 )-- ( 8.2 , 0.035 )-- ( 8.3 , 0.033 )-- ( 8.4 , 0.032 )-- ( 8.5 , 0.03 )-- ( 8.6 , 0.029 )-- ( 8.7 , 0.028 )-- ( 8.8 , 0.026 )-- ( 8.9 , 0.025 )-- ( 9 , 0.024 )-- ( 9.1 , 0.023 )-- ( 9.2 , 0.022 )-- ( 9.3 , 0.021 )-- ( 9.4 , 0.02 )-- ( 9.5 , 0.019 )-- ( 9.6 , 0.018 )-- ( 9.7 , 0.017 )-- ( 9.8 , 0.016 )-- ( 9.9 , 0.016 )-- ( 10 , 0.015 )-- ( 10.1 , 0.014 )-- ( 10.2 , 0.014 )-- ( 10.3 , 0.013 )-- ( 10.4 , 0.012 )-- ( 10.5 , 0.012 )-- ( 10.6 , 0.011 )-- ( 10.7 , 0.011 )-- ( 10.8 , 0.01 )-- ( 10.9 , 0.01 )-- ( 11 , 0.009 )-- ( 11.1 , 0.009 )-- ( 11.2 , 0.008 )-- ( 11.3 , 0.008 )-- ( 11.4 , 0.008 )-- ( 11.5 , 0.007 )-- ( 11.6 , 0.007 )-- ( 11.7 , 0.007 )-- ( 11.8 , 0.006 )-- ( 11.9 , 0.006 )-- ( 12 , 0.006 )-- ( 12.1 , 0.006 )-- ( 12.2 , 0.005 )-- ( 12.3 , 0.005 )-- ( 12.4 , 0.005 )-- ( 12.5 , 0.005 )-- ( 12.6 , 0.004 )-- ( 12.7 , 0.004 )-- ( 12.8 , 0.004 )-- ( 12.9 , 0.004 )-- ( 13 , 0.004 )-- ( 13.1 , 0.003 )-- ( 13.2 , 0.003 )-- ( 13.3 , 0.003 )-- ( 13.4 , 0.003 )-- ( 13.5 , 0.003 )-- ( 13.6 , 0.003 )-- ( 13.7 , 0.003 )-- ( 13.8 , 0.002 )-- ( 13.9 , 0.002 )-- ( 14 , 0.002 )-- ( 14.1 , 0.002 )-- ( 14.2 , 0.002 )-- ( 14.3 , 0.002 )-- ( 14.4 , 0.002 )-- ( 14.5 , 0.002 )-- ( 14.6 , 0.002 )-- ( 14.7 , 0.002 )-- ( 14.8 , 0.002 )-- ( 14.9 , 0.001 )-- ( 15 , 0.001 )-- ( 15.1 , 0.001 )-- ( 15.2 , 0.001 )-- ( 15.3 , 0.001 )-- ( 15.4 , 0.001 )-- ( 15.5 , 0.001 )-- ( 15.6 , 0.001 )-- ( 15.7 , 0.001 )-- ( 15.8 , 0.001 )-- ( 15.9 , 0.001 )-- ( 16 , 0.001 )-- ( 16.1 , 0.001 )-- ( 16.2 , 0.001 )-- ( 16.3 , 0.001 )-- ( 16.4 , 0.001 )-- ( 16.5 , 0.001 )-- ( 16.6 , 0.001 )-- ( 16.7 , 0.001 )-- ( 16.8 , 0.001 )-- ( 16.9 , 0.001 )-- ( 17 , 0.001 )-- ( 17.1 , 0.001 )-- ( 17.2 , 0 )-- ( 17.3 , 0 )-- ( 17.4 , 0 )-- ( 17.5 , 0 )-- ( 17.6 , 0 )-- ( 17.7 , 0 )-- ( 17.8 , 0 )-- ( 17.9 , 0 )-- ( 18 , 0 )-- ( 18.1 , 0 )-- ( 18.2 , 0 )-- ( 18.3 , 0 )-- ( 18.4 , 0 )-- ( 18.5 , 0 )-- ( 18.6 , 0 )-- ( 18.7 , 0 )-- ( 18.8 , 0 )-- ( 18.9 , 0 )-- ( 19 , 0 )-- ( 19.1 , 0 )-- ( 19.2 , 0 )-- ( 19.3 , 0 )-- ( 19.4 , 0 )-- ( 19.5 , 0 )-- ( 19.6 , 0 )-- ( 19.7 , 0 )-- ( 19.8 , 0 )-- ( 19.9 , 0 )-- ( 20 , 0 ); 
 \end{tikzpicture}
\caption{Gamma$(1.5,0.5)$}\label{fig:H_MLE_gamma_multi}
\end{subfigure}

\vspace{0.5cm}
 \begin{tikzpicture}
    \begin{customlegend}
    [legend entries={ $\overline H_{\theta_0}(t)$,$\overline H_{\hat{\theta}_n^{(1)}}(t)$,$\overline H_{\hat{\theta}_n^{(3)}}(t)$,$\overline H_{\hat{\theta}_n^{({\rm GHE})}}(t)$},legend columns=-1,legend style={/tikz/every even column/.append style={column sep=0.8cm}}]   
    \addlegendimage{black} 
    \addlegendimage{red,densely dotted, thick}    
    \addlegendimage{blue,densely dashed,thick}    
    \addlegendimage{olive,dashdotted,thick}    
    \end{customlegend}
\end{tikzpicture}
\caption{Tail distribution of the customer patience, compared with its estimated counterparts. The true patience distributions $H_{\theta_0}$ are lognormal and gamma. The fitted distributions $H_{\hat{\theta}_n^{(p)}}(t)$ are based on a sample of $n=30\,000$ observations with the HE MLE for $p\in\{1,3\}$, and the GHE heuristic.  The system has $s=10$ has servers with an arrival rate of $\lambda=10$, and the service requirements follow a Gamma distribution with parameters $(3,2)$.} \label{fig:H_MLE_multi_tails}
\end{figure}


\section{Applications}\label{sec:app}

In this section we discuss a series of applications in which our estimation methodology can be directly used.

\subsection{Unknown proportion of impatient customers}\label{sec:mle_types}
In this subsection we consider a system where an unknown proportion $\theta\in(0,1)$ of the arriving customers has a known deterministic patience threshold $w>0$. The other customers are patient and always join the queue. Suppose one wishes to estimate the fraction $\theta$ from data. 

To this end, first observe that
\[
H_\theta(x)=\theta\,\mathbf{1}_{\{x \geqslant w\}}.
\] 
Assumption (A4) is not satisfied because $H_\theta(x)$ is not continuous (in $x$) at $x=w$. However, with respect to $\theta$ we {\it do} have that $H_\theta(x)$ is continuous. In addition, $\ell_n(\theta)$ is concave, and therefore the MLE is given by the first order condition $\ell_n\apost(\hat{\theta}_n)=0$. In this case we can apply results from \cite{HM1986} to establish strong consistency. In particular, the smoothness of $\ell_n\apost(\theta)$ enables direct verification of the sufficient condition \eqref{eq:cont_mle_condition2}. Asymptotic normality then follows by verifying the conditions of Theorem \ref{thm:cont_mle_normal} directly.

\subsection{Noisy delay messages}\label{sec:mle_noisy}

Suppose that the patience threshold is a constant $\theta$ but the customers do not observe their exact waiting time but rather some noisy estimate $W_e=e(W)$ for some random function $e$ such that $\E[W_e\,|\,W]=W$. We assume customers join based on this `perturbed delay' $W_e$, i.e., if $W_e\leqslant \theta$. Suppose we are in the context that the parameters underlying the noise distribution are known, but the threshold $\theta$ is not. 

The probability of joining at virtual waiting time level $W$ can be computed given the specific noise distribution.
A few examples are:
\begin{enumerate}
\item[$\circ$] {\it Additive perturbations}. In this case $W_e=W+\epsilon$. One could for instance consider normally distributed perturbations: $\epsilon\sim\mathrm{N}(0,\sigma^2)$, independent of $W$, with $\sigma>0$. Let customers facing $W_e<\theta$ join the system. We thus have, with $\Phi(\cdot)$ the cdf of the standard normal distribution,
\[
\overline{H}_\theta(v)=\P(W_e\leqslant \theta\,|\,W=v)=\P(\epsilon\leqslant \theta-v)=\Phi\left(\frac{\theta-v}{\sigma}\right).
\]
The asymptotic properties of Section~\ref{sec:mle_continuous} hold in this case because $\Phi(\cdot)$ satisfies the regularity conditions.
\item[$\circ$] {\it Multiplicative perturbations.} Now $W_e=WG$, with $G$ non-negative unit-mean, and independent of $W$.
In this case
\[
\overline{H}_\theta(v)=\P(W_e\leqslant \theta\,|\,W=v)=\P\left(G\leqslant \frac{\theta}{v}\right).
\]
Now if the random variable $G$ is such that the regularity conditions of Section~\ref{sec:mle_continuous} are met, then it follows that the asymptotic properties hold in this case as well (with an asymptotic variance that can be expressed in terms of $G$).
\end{enumerate}

\subsection{Admission pricing}\label{sec:mle_price}
Our estimation procedure can be exploited in the context of various problems rooted in operations research. Evidently, when having estimates of the arrival rate and the patience-level distribution, one could consider the option of increasing the service rate so as to potentially raise profits. Thus, the operational decision to be made is whether the increased revenues outweigh the cost of speeding up the service rate. In this subsection we consider another operations research related problem. 

Suppose that the queue has an admission price of $p$ and that customers are homogenous with a utility function as featuring in the Naor model \cite{NA}. More concretely, let there be constants $r$ and $c$ such that a customer will join the queue only if the virtual waiting time $w$ upon their arrival satisfies
\[
r-p-c w\geqslant 0,\]
or, alternatively, $w$ is smaller than the threshold value $(r-p)/c.$
Now observe that for any fixed price $p$ the threshold $\theta(p)=(r-p)/c$ can be estimated using the MLE procedure presented in Section \ref{sec:mle_constant}. If one of the cost function parameters $r$ and $c$ is known, then the other parameter can be estimated directly. If both are unknown, then their estimation can be performed by an exploration procedure: set two prices, say $p_1$ and $p_2$, and observe the system for each price. Supposing each of these two experiments is done with $n$ clients, we let $\hat{\theta}_n(p_i)$ be the MLE for price $i=1,2$. Then the estimators for the cost function parameters $r$ and $c$ are given by, respectively,
\[
\hat{r}_n=\frac{\hat{\theta}_n(p_1)\,p_2-\hat{\theta}_n(p_2)\,p_1}{\hat{\theta}_n(p_1)-\hat{\theta}_n(p_2)} , \quad \hat{c}_n=\frac{\hat{r}_n-p_1}{\hat{\theta}_n(p_1)} .
\]
Theorem \ref{thm:mle_constant}, in combination with the continuous mapping theorem, implies that, as $n\to\infty$, both $\hat{r}_n\asarrow r$ and $\hat{c}_n\asarrow c$. Furthermore, for large $n$ confidence intervals for the stationary average revenue per unit of time can be approximated using \eqref{eq:hat-theta-scaled-limit}. In particular, observing that the loss probability depends on the price, in self-evident notation,
\[
n(\theta(p_i)-\hat{\theta}_n(p_i)) \darrow \mathrm{Exp}\left[\frac{v_i(\theta(p_i))}{1-P_{\ell}(p_i)}\right], \ i=1,2,
\]
as $n\to\infty.$ 
This opens the possibility of approximating, for $n$ large, the (joint) distribution of $n\left(\hat{r}_n-r_0,\hat{c}_n-c_0\right))$. Note that the loss probability $P_{\ell}$ needs to be estimated as well. For the single-server case this can be done by estimating the idle probability upon arrival for a given price and using \eqref{eq:P_loss}.


\section{Concluding remarks}\label{sec:Conc}

This paper has considered a service system in which clients
potentially balk based on the virtual waiting time level they face at arrival. The
main objective concerned the development of a framework for estimating the
arrival rate and patience-level distribution. Our approach resolves the
complication that in our setup only non-balking clients are observed.
Distinguishing between the case of a continuous patience-level distribution
and constant patience, we developed MLE estimators and quantified
their asymptotic properties. Through a sequence of examples and
ramifications we have illustrated the performance and broad
applicability of our findings.

An important next step could concern the extension to a non-stationary arrival process.  For example, one may assume that the potential arrival process is a non-homogeneous Poisson process with an arrival rate function that depends on time, i.e., $\{\lambda(t)\!: t\geqslant 0\}$.  This means that, when the virtual waiting time is $v$, the effective arrival rate at time $t$  is given by $\lambda(t)(1-H(v))$, and that the likelihood function in Section~\ref{sec:model} can be updated accordingly.  If the arrival rate function is known, then our  estimation procedure for $\theta$ essentially carries over, including its performance guarantees. If the arrival rate function is unknown and parametric assumptions about it are made, then the joint estimation of the arrival rate and patience parameters is possible. Furthermore, in the practically relevant case that the time-dependent arrival rate is cyclic (i.e.,  $\lambda(t)=\lambda(t+s)$ for all $t\geqslant 0$ and some cycle length $s\geqslant 0$; think of daily or weekly patterns), then the queueing process will still have regenerative dynamics which can be exploited for asymptotic analysis. 
It is thus anticipated that the framework laid out in this paper can serve as a basis for developing estimation techniques for a wide class of non-stationary models.

While we have developed estimation procedures assuming the
independence between the customers' service requirements and patience levels, one could consider an extension that allows dependence. Such an extension could, for example,
model the situation in which customers with larger service requirements can be assumed to have more
patience. Multi-class queueing models can be used to
capture such dependence in service requirements and patience levels~\cite{S2017}. More
specifically, suppose that there are two customer classes $1$ and $2$,
where class $k$ ($k=1,2$) customers arrive with rate $\lambda_k$ and
have service requirements (resp.\ patience levels), i.i.d.\ according to a
class-dependent cdf $G_k(\cdot)$ (resp.\ $H_k(\cdot)$). It is natural to assume that
the customer class is unobservable by the estimator, thus yielding a
sequence of {\it dependent} service requirements and patience levels. In this setting,
we have to estimate the service-requirement cdf $G_k(\cdot)$ jointly with the
patience-level cdf $H_k(\cdot)$ (and the arrival rate $\lambda_k$) to
compute the likelihood function corresponding to an observed sequence 
$({\boldsymbol A}, {\boldsymbol W}, {\boldsymbol X})$. 

Another possible direction for future work concerns the situation in which 
each customer decides to balk based on her {\it sojourn time}, i.e., the virtual waiting time
seen by this customer just {\it after} (instead of before) her arrival. If 
the customer precisely knows her service requirement, this problem is easily reduced
to the one considered in this paper. 
The other obvious option is that she balks if the virtual waiting time 
just before her arrival increased by a `guess' of her service requirement exceeds
her patience threshold. 

An alternative to the approach followed in the paper, would be to pursue non-parametric estimation; cf.\ the results in \cite{HP2006} for the conventional case without balking. In the single-server case it may be of help that  we
can write, with $W(x)$ and $D(x)$ denoting the cdf\,s of the waiting times and sojourn times of non-balking customers respectively,
\begin{equation}
\label{INV}
\overline{H}(x) = \displaystyle
\frac{w(x)}{\lambda(W(x)-D(x))}
\end{equation}
if $w(x)>0$ and $0$ else; the validity of \eqref{INV} follows from combining \eqref{eq:w-general} with a level-crossing identity. While it is clear how to estimate $W(x)$ and $D(x)$ using the empirical distribution, estimation of $w(x)$ is more challenging; kernel-based techniques may turn out useful in this context.

A final issue to consider is that in many systems customers do not observe exact waiting times, but rather queue lengths.   Specifically,  this can be modelled as an M/G/$s$\,+\,H system with customers that join or balk after having observed the number of customers in the system.  In this case the patience level $Y$ is a discrete random variable indicating at what queue length a customer is willing to join. Therefore, the inter-arrival process at any given queue length $q\geqslant 0$ is a Poisson process with rate $\lambda_q=\lambda \bar{H}_\theta(q)$. As was done in this paper, and MLE for the arrival rates (and the corresponding parameter $\theta$) can be derived from a sample of inter-arrival times and the respective queue lengths.

\bibliographystyle{plain}

\section*{Appendix: Proofs}\label{sec:appen}

\begin{proof}[Proof of Lemma~\ref{lemma:unifrom_ll}]
Assuming (A1)--(A3) and (A5), we verify the conditions for \eqref{eq:cont_uniform_ell_theta} as given in \cite[Thm. 6.4]{R1962}. These are: (i) $\Theta$ is compact, (ii) $(A_n,W_n,X_n)_{n\geqslant 0}$ is a stationary ergodic sequence, (iii) the function $\frac{1}{n}\ell_n(\theta)$ converges almost surely pointwise to $\ell(\theta)$, 
(iv) the set $\{\ell_i(\theta):\ \theta\in\Theta\}$ is equicontinuous with respect to $(A,X,W)$,
(v) there exists a function $K:[0,\infty)^4\to\mathbb{R}$ such that 
\begin{equation}\label{eq:K_uniform_condition}
\ell_1(\theta;a,z,y,x)\leqslant K(a,z,y,x), \ \forall \theta\in\Theta; \quad \E[K(A_1,W_0,W_1,X_0)]<\infty.
\end{equation}
Condition (i) is assumed in (A1). Assumption (A2) states that $W_n$ is a stationary ergodic sequence, which also implies that the same is true for $A_n$ and $X_n$, thus yielding the validity of (ii). Note that conditional on $W_{i}=z$ the distribution of $A_i$ is determined by $z$ and the external arrival process. Moreover, the upward-jump size is determined by the new service requirement together with the vector of residual service requirements in the system, which are assumed to be stationary and ergodic. Condition (iii) is \eqref{eq:ell_theta_as}. The equicontinuity assumption (A5) on $H_\theta$ implies that the same holds for the set of log-likelihood functions because both the log and the integral terms in \eqref{eq:V_log_likelihood} are continuous functions with respect to the observations $({\bs A},{\bs X},{\bs W})$, so that we have established (iv). In the remainder of the proof we verify (v).
 
Recall, from (A3), the definition $\overline{H}_{\inf}(y)=\inf_{\theta\in\Theta}\overline{H}_\theta(y)$ for any $y\geqslant 0$. As $\Theta$ is compact and all functions $H_\theta$ are continuous, we have that $\overline{H}_{\inf}(\cdot)$ is a continuous monotone non-increasing function such that $\overline{H}_{\inf}(0)\in(0,1]$. Defining
\[
K(a,y):=|\log \lambda\,|+|\log
\overline{H}_{\inf}(y)|{\boldsymbol 1}_{\{\overline{H}_{\inf}(y)>0\}}+\lambda a,
\]
by recalling that $H_\theta(y)\leqslant 1$ for any $y\geqslant 0$ we conclude from \eqref{eq:V_log_likelihood} and the triangle inequality that
\[
|\ell_1(\theta;a,z,y,x)|\leqslant K(a,y),\quad \forall a,z,y,x\geqslant 0.
\]
Clearly $\E A_1<\infty$ and so to complete the proof we need to show that $\E\,|\log \overline{H}_{\inf}(W_1)|<\infty$. 
Recall that $W_1\sim W_0$ is the stationary virtual waiting time seen by an effective arrival, whose pdf $w(\cdot)$ is given by (cf.\ (\ref{eq:w-general}))
\[
w(y) = \frac{v(y)\overline{H}_{\theta_0}(y)}{1-P_{\ell}}.
\]
As mentioned in Section \ref{sec:cont_mle_assumptions}, (A5)
ensures that the support of the distribution $H_{\theta}$ does not
depend on $\theta$. The above equation thus implies $W_0 < h_{\sup}$
with probability one.

Note that we have
\[
\E\,|\log \overline{H}_{\inf}(W_0)|
=
\P(W_0 = 0) (-\log (\overline{H}_{\inf}(0)))
+
\P(W_0 > 0) \E[-\log(\overline{H}_{\inf}(W_0)) \,|\, W_0 > 0],
\]
where the first term on the right-hand side is always finite.
We thus provide a proof of $\E[-\log(\overline{H}_{\inf}(W_0)) \,|\, W_0 > 0] < \infty$ below.

Regardless of whether $h_{\sup} < \infty$ or $h_{\sup} = \infty$,
it follows from (A3) that for any $ \epsilon \in(0, \infty)$,
there exist $0 < y_1^\star  < h_{\sup}$ and $0 < y_2^\star  < h_{\sup}$ such that 
\[
\frac{e^{-f(y)}}{\overline{H}_{\inf}(y)} 
< 
c_1 + \epsilon,
\;\;
y \geqslant y_1^\star ,
\qquad
0 < 
f(y) \overline{H}_{\theta_0}(y)
<
c_2 + \epsilon,
\;\;
y \geqslant y_2^\star .
\]
Using these inequalities, with $y^\star  := \max\{y_1^\star , y_2^\star \}$, we obtain
\begin{align*}
\lefteqn{
(1-P_{\ell})
\E[-\log(\overline{H}_{\inf}(W_0)) \,|\, W_0 > 0}]\quad
\\
&=
(1-P_{\ell})
\int_0^{h_{\sup}} (-\log(\overline{H}_{\inf}(y))) w(y) \diff y
\\
&=
\int_0^{h_{\sup}}
(-\log(\overline{H}_{\inf}(y)))
\overline{H}_{\theta_0}(y) v(y) \diff y
\\
&<
\int_0^{y^\star }
(-\log(\overline{H}_{\inf}(y)))
\overline{H}_{\theta_0}(y) v(y) \diff y
+
\int_{y^\star }^{h_{\sup}}
\left(-\log\left(\frac{e^{-f(y)}}{c_1+\epsilon}\right)\right)
\frac{c_2+\epsilon}{f(y)}
\cdot
v(y) \diff y
\\
&\leq
-\log(\overline{H}_{\inf}(y^\star ))
\int_0^{y^\star }
v(y) \diff y
+
\int_{y^\star }^{h_{\sup}}
\left(f(y) + |\log(c_1+\epsilon)|\right)
\cdot
\frac{c_2+\epsilon}{f(y)} 
\cdot
v(y) \diff y
\\
&\leq
-\log(\overline{H}_{\inf}(y^\star ))
+
(c_2+\epsilon) 
\int_{y^\star }^{h_{\sup}}
v(y) \diff y
+
\frac{|\log(c_1+\epsilon)| (c_2+\epsilon)}{f(y^\star )}
\int_{y^\star }^{h_{\sup}}
v(y) \diff y
\\
&\leq
-\log(\overline{H}_{\inf}(y^\star ))
+
(c_2+\epsilon) 
\left(
1 + \frac{|\log(c_1+\epsilon)|}{f(y^\star )}
\right)
< \infty,
\end{align*}
which implies $\E[-\log(\overline{H}_{\inf}(W_0)) \,|\, W_0 > 0] < \infty$. 
\end{proof}

\begin{proof}[Proof of Lemma~\ref{lemma:cont_mle_clt}]
(a) With Lemma \ref{lemma:unifrom_ll} and Theorem \ref{thm:cont_mle_consistency} at our disposal, the proof essentially follows standard arguments. For instance, see the proof of \cite[Thm.\ 18]{book_F1996} for i.i.d.\ observations, or the proof of \cite[Lemma 14]{RBM2019} that applies a martingale CLT for an MLE based on dependent workload observations in the M/G/1-queue context. \\
(b) We construct a martingale CLT for stationary sequences \cite[Thm.\ 18.3]{book_Billingsley1999} and verify that the corresponding conditions are satisfied. Let $Z_i:=\dot{\ell}_i(\theta_0)-\dot{\ell}_{i-1}(\theta_0)$ for $i\in\{1,2,\ldots\}$, where $\dot{\ell}_0(\theta_0):=0$. The identifiability condition \eqref{eq:cont_MLE_unique} implies that the smooth function $\ell(\theta)$ is maximized at $\theta_0$, hence $\dot{\ell}(\theta_0):=\nabla\ell(\theta_0)=0$, i.e., all coordinates equal zero. Therefore, if $W_{i-1}$ is stationary we have that 
\[
\E[(Z_i)_k\,|\,W_{i-1}]=\E[(\dot{\ell}_1(\theta_0))_k\,|\,W_{0}]=(\dot{\ell}(\theta_0))_k=0, \ \forall k\in\{1,\ldots,p \},
\]
and
\[
Z_{ni}:=K_n^{-1}Z_i, \  i\in \{1,\ldots, n\},
\]
is a martingale difference, where $K_nK_n^\top=\Var[\dot{\ell}_n(\theta_0)]$ is the Cholesky decomposition of the covariance matrix of $\dot{\ell}_n(\theta_0)$. Furthermore,
\[
\frac{1}{n}\Var\left[\dot{\ell}_n(\theta_0)\right]=\Var\left[\frac{1}{\sqrt{n}}\dot{\ell}_n(\theta_0)\right]\asarrow -\E \Psi_1(\theta_0)=I(\theta_0),
\]
as $n\to\infty$, and as $I(\theta_0)<\infty$ by (A4) we have that $\lim_{n\to\infty}(K_n)_k =\infty$ for all $k=1,\ldots,p$, and in addition
\[
\lim_{n\to\infty}\frac{K_n}{\sqrt{n}}=K,
\] 
where $K$ is defined through $K K^\top=I(\theta_0)$. Therefore by \cite[Thm.\ 12.6]{book_H1997} (which is the multi-dimensional counterpart of \cite[Thm.\ 18.1]{book_Billingsley1999}) we conclude that $\sum_{i=1}^n Z_{ni}$ converges to a standard normal random variable, and thus
\[
\frac{1}{\sqrt{n}}\dot{\ell}_n(\theta_0)=\frac{1}{\sqrt{n}}K_n\sum_{i=1}^n Z_{ni}\darrow K\,\mathrm{N}(0,1),
\]
which is equivalent to \eqref{eq:cont_dl_clt}.
\end{proof}

\begin{proof}[Proof of Theorem \ref{thm:mle_constant}]
This proof consists of a lower bound and an upper bound.

\noindent
\textit{$\rhd$ Lower bound.}
Note that for $\theta_0>W_{k-1}$, we have
\begin{equation}\label{eq:Ak}
A_k 
\sim
\left\lbrace
\begin{array}{ll}
\E_{\lambda,k}, & W_{k-1}+X_{k-1}\leqslant \theta_0, 
\\
W_{k-1}+X_{k-1}-\theta_0+{\rm E}_{\lambda,k}, &
W_{k-1}+X_{k-1} > \theta_0, \end{array}\right.
\end{equation}
where $(\E_{\lambda,k})_{k=1,2,\ldots}$ denotes a sequence of i.i.d.\ exponentially distributed random variables with parameter $\lambda$.
Let $\mathcal{N}_n^+ \subset \{1,2,\ldots,n\}$ denote the set of
indices of observed virtual waiting times such that
\[
\mathcal{N}_n^+ = \{k \in \{1,2,\ldots,n\}:\, W_{k-1}+X_{k-1} \geqslant \theta_0\}.
\]
We see from (\ref{eq:Vk_recursion}) and (\ref{eq:Ak}) that 
given $k \in \mathcal{N}_n^+$, the observed virtual waiting time $W_k$ is
stochastically identical to $\max\{0,\theta_0 - A_k'\}$, where 
$(A_k')_{k = 0,1,\ldots}$ denotes a sequence of i.i.d.\
random variables that are exponentially distributed with parameter
$\lambda$. Therefore, $(W_k)_{k \in \mathcal{N}_n^+}$ are i.i.d.\ with 
\begin{equation}
\P(W_k \leqslant x \,|\, k \in \mathcal{N}_n^+) = \mathrm{e}^{-\lambda (\theta_0 - x)},
\quad
x \in [0, \theta_0],
\label{eq:Z-cdf}
\end{equation}
so that
\begin{equation}
\P\left(\max_{i \in \mathcal{N}_n^+} W_i \leqslant x\right) 
= 
\E\left[
\prod_{i \in \mathcal{N}_n^+} {\boldsymbol 1}_{\{W_i \leqslant x\}}
\right]
=
\E\left[
\mathrm{e}^{-\lambda (\theta_0 - x) \cdot |\mathcal{N}_n^+| }
\right],
\label{eq:lower_n}
\end{equation}
where $|\mathcal{N}_n^+|$ denotes the number of elements in
$\mathcal{N}_n^+$. 

Obviously, we have $\theta_0\geqslant \hat\theta_n$ and 
\begin{equation}
\hat{\theta}_n \geqslant \max_{i \in \mathcal{N}_n^+} W_i,
\quad
\mbox{a.s.}
\label{eq:hat-theta-lower}
\end{equation}
Let, as before, $v(\cdot)$ and $w(\cdot)$ denote the pdf\,s of
the stationary virtual waiting time and the stationary waiting time of non-balking
customers, and let $P_{\ell}$ denote the stationary
loss probability. From the ergodicity, we have 
\begin{equation}
\frac{|\mathcal{N}_n^+|}{n} 
\asarrow 
\frac{q_{s-1}}{1-P_{\ell}} \cdot \overline{J}_{\mid 0}(\theta_0)
+ 
\int_0^{\theta_0} w(y)\overline{J}_{\mid y}(\theta_0-y) \diff y
=
\frac{v(\theta_0)}{\lambda (1-P_{\ell})},
\label{eq:N_n^+-limit}
\end{equation}
as $n \to \infty$, where the second equality follows from
(\ref{eq:v-Volterra-multi}) and (\ref{eq:w-general}).

Because (\ref{eq:N_n^+-limit}) implies $|\mathcal{N}_n^+| \asarrow
\infty$, we have for any $\epsilon > 0$, using \eqref{eq:lower_n} and \eqref{eq:hat-theta-lower},
\begin{align*}
\P\left(|\hat{\theta}_n - \theta_0| > \epsilon\right) 
=
\P\left(\hat{\theta}_n < \theta_0 - \epsilon\right) 
&\leq
\P\left(\max_{i \in \mathcal{N}_n^+} W_i \leqslant \theta_0 - \epsilon\right) 
\\
&=
\E\left[
\mathrm{e}^{-\lambda \epsilon \cdot |\mathcal{N}_n^+| }
\right]
\to 0
\end{align*}
as $n\to\infty$.
We therefore have that $\hat{\theta}_n \parrow \theta_0$ as $n \to \infty$.
Convergence in probability implies that there exists
a subsequence $(\hat{\theta}_{n_m})_{m=1}^\infty$ such that
$\hat{\theta}_{n_m}\asarrow \theta_0$ as $m \to \infty$, and as
$\hat{\theta}_n$ is a monotone non-decreasing sequence we conclude \eqref{eq:maxV-limit}, i.e., the MLE is strongly consistent.

Furthermore, it follows from (\ref{eq:lower_n}) and (\ref{eq:N_n^+-limit}) that
\begin{align}
\P\left(n\Bigl\{\theta_0 - \max_{i \in \mathcal{N}_n^+} W_i\Bigr\} \geqslant x\right) 
&=
\P\left(
\max_{i \in \mathcal{N}_n^+} W_i 
\leqslant 
\theta_0 
-\frac{x}{n}
\nonumber
\right) 
\\
&=
\left\{
\begin{array}{l@{\qquad}l}
\E\left[
\mathrm{e}^{-\lambda x \cdot |\mathcal{N}_n^+|/n }
\right],
&
0 \leqslant x < n\theta_0,
\\
0, & x \geqslant n\theta_0,
\end{array}
\right.
\label{eq:lower-scaled}
\end{align}
so that for each $x \geqslant 0$, we obtain from the continuous mapping theorem,
\begin{align}
\lim_{n \to \infty}
\P\left(n\Bigl\{\theta_0 - \max_{i \in \mathcal{N}_n^+} W_i\Bigr\} \geqslant x\right) 
=
\mathrm{e}^{-v(\theta_0) x/(1-P_{\ell})}.
\label{eq:lower-scaled-limit}
\end{align}

\smallskip
\noindent
\textit{$\rhd$ Upper bound.}
Let $\mathcal{N}_n^{\star} \subset \{1,2,\ldots,n\}$ denote the set of
indices of observed virtual waiting times such that
\[
\mathcal{N}_n^{\star} 
= 
\left\{
k \in \{1,2,\ldots,n\};\, W_{k-1}+X_{k-1} > \max_{i \in \{1,2,\ldots,k-1\}} W_i
\right\}.
\]
By definition, we have
\begin{equation}
\hat{\theta}_n = \max_{i \in \mathcal{N}_n^{\star}} W_i.
\label{eq:hat-theta-N-star}
\end{equation}
For $k = 1,2,\ldots,n$, we define 
\begin{align*}
Z_k :=& \max\{0, \theta_0 - [A_k - \max\{0,W_{k-1}+X_{k-1}-\theta_0\}] \}
\\
=&
\left\{
\begin{array}{l@{\quad}l}
\max\{0,\theta_0 - A_k\}, & W_{k-1}+X_{k-1} \leqslant \theta_0,
\\
\max\{0, W_{k-1}+X_{k-1} - A_k\}, & W_{k-1}+X_{k-1} > \theta_0.
\end{array}
\right.
\end{align*}
From \eqref{eq:Vk_recursion} and \eqref{eq:Ak} it
follows that $W_i \leqslant Z_i$ a.s.\ for $i=0,1,\ldots,n$. 
We then have from (\ref{eq:hat-theta-N-star}),
\begin{equation}
\hat{\theta}_n 
\leqslant 
\max_{i \in \mathcal{N}_n^{\star}} Z_i,
\quad
\mbox{a.s.}
\label{eq:hat-theta-upper}
\end{equation}
Furthermore, $Z_k$ is stochastically identical to $\max\{0,\theta_0 -
{\rm E}_{\lambda,k}\}$, where $({\rm E}_{\lambda,k})_{k=1,2,\ldots}$ are i.i.d.\ exponential random variables 
with parameter $\lambda$ as defined above. $\{Z_k\}_{k \in
\mathcal{N}_n^{\star}}$ thus forms a sequence of i.i.d.\ non-negative
random variables with the same cdf as (\ref{eq:Z-cdf}),
so that (cf.\ (\ref{eq:lower_n}))
\[
\P\left(
\max_{i \in \mathcal{N}_n^{\star}} Z_i \leqslant x
\right)
=
\E\left[
\mathrm{e}^{-\lambda(\theta_0 -x) \cdot |\mathcal{N}_n^{\star}|}
\right],
\]
which implies (cf.\ (\ref{eq:lower-scaled}))
\begin{equation}
\P\left(
n\left\{\theta_0 -\max_{i \in \mathcal{N}_n^{\star}} Z_i\right\} \geqslant x
\right)
=
\left\{
\begin{array}{l@{\qquad}l}
\E\left[\mathrm{e}^{-\lambda x\cdot  |\mathcal{N}_n^{\star}|/n}\right],
&
0 \leqslant x < n\theta_0,
\\
0, & x \geqslant n\theta_0.
\end{array}
\right.
\label{eq:upper-scaled}
\end{equation}

Let $\phi_{\theta_0}(x)$ ($x \in (0,\theta_0]$) denote the stationary probability
that the virtual waiting time just after an acceptance of a customer takes a value in
$(\theta_0-x,\theta_0)$:
\[
\phi_{\theta_0}(x) 
= 
\pi_0 \,\P(\theta_0 - x < X_{\mid 0} < \theta_0)
+
\int_0^{\theta_0 }
v(y)\P(\theta_0-y-x < X_{\mid y} < \theta_0-y)\diff y.
\]
Because we have $\int_0^{\theta_0}v(y)\diff y < \infty$,
it follows from the dominated convergence theorem that
\[
\lim_{x \to 0+}\phi_{\theta_0}(x) = 0,
\]
so that we have
\begin{equation}
\forall \delta > 0, \exists d_0 > 0:
\quad
\phi_{\theta_0}(d_0) < \frac{\delta}{4}.
\label{eq:d_0}
\end{equation}

Now define $\mathcal{N}_n^{\star\star} := \mathcal{N}_n^{\star} \setminus
\mathcal{N}_n^+$, which concretely means that
\[
\mathcal{N}_n^{\star\star} 
= 
\left\{
k \in \{1,2,\ldots,n\};\, \max_{i=1,2,\ldots,k-1} W_i <
W_{k-1}+X_{k-1} < \theta_0
\right\}.
\]
As $\mathcal{N}_n^+$ and $\mathcal{N}_n^{\star\star}$ are two disjoint sets whose union is $\mathcal{N}_n^{\star}$, we have that
\begin{equation}
|\mathcal{N}_n^{\star}| 
=
|\mathcal{N}_n^+|
+
|\mathcal{N}_n^{\star\star}|.
\label{eq:N_n^-star-by-sum}
\end{equation}
Note that (\ref{eq:maxV-limit}) implies
\begin{equation}
\exists n_0 \in \{0,1,\ldots\},
\,
\P(\hat{\theta}_n < \theta_0 - d_0) < \frac{\delta}{4}
\;\;
\mbox{for $n = n_0, n_0+1,\ldots$}.
\label{eq:n_0}
\end{equation}
Because $\hat{\theta}_n$ is non-decreasing in $n$ a.s., 
we have for $n = n_0, n_0+1,\ldots$,
\begin{align*}
\E\Bigl[|\mathcal{N}_n^{\star\star}|\Bigr]
&=
\E\Bigl[|\mathcal{N}_{n_0}^{\star\star}|
+
|\mathcal{N}_n^{\star\star} 
\setminus \mathcal{N}_{n_0}^{\star\star}|\Bigr]
\\
&=
\E\Bigl[|\mathcal{N}_{n_0}^{\star\star}|\Bigr]
+
\P(\hat{\theta}_{n_0} < \theta_0 - d_0)
\cdot
\E\biggl[
|\mathcal{N}_n^{\star\star} \setminus \mathcal{N}_{n_0}^{\star\star}|
\ \Bigr|\ 
\hat{\theta}_{n_0} < \theta_0 - d_0
\biggr]
\\
&\hspace{10em}
{}+
\P(\hat{\theta}_{n_0} \geqslant \theta_0 - d_0)
\cdot
\E\biggl[
|\mathcal{N}_n^{\star\star} \setminus \mathcal{N}_{n_0}^{\star\star}|
\ \Bigr|\
\hat{\theta}_{n_0} \geqslant \theta_0 - d_0
\biggr]
\\
&\leq
\E\Bigl[
|\mathcal{N}_{n_0}^{\star\star}|
\Bigr]
{}+
\P(\hat{\theta}_{n_0} < \theta_0 - d_0)
\cdot
(n-n_0+1)
\\
&\quad
{}+
\P(\hat{\theta}_{n_0} \geqslant \theta_0 - d_0)
\cdot
\E\Bigl[
\left|
\left\{
k \in \{n_0, n_0+1,\ldots,n\};\, d_0 < W_{k-1}+X_{k-1} < \theta_0
\right\}
\right|
\Bigr].
\end{align*}
It then follows from (\ref{eq:d_0}) and (\ref{eq:n_0}) that
\begin{align*}
\E\left[
\frac{|\mathcal{N}_n^{\star\star}|}{n}
\right]
&\leq
\E\left[
\frac{|\mathcal{N}_{n_0}^{\star\star}|}{n}
\right]
+
\frac{\delta}{4}
\cdot
\frac{n-n_0+1}{n}
\\
&\quad
{}+
\frac{n-n_0+1}{n}
\E\left[
\frac{
\left|
\left\{
k \in \{n_0, n_0+1,\ldots,n\};\, d_0 < W_{k-1}+X_{k-1} < \theta_0
\right\}
\right|
}{n-n_0+1}
\right]
\\
&\to
\frac{\delta}{4}
+
\phi_{\theta_0}(d_0)
\end{align*}
as $n\to\infty$, which is majorized by $\delta/2.$
We conclude that we have established that $\E[|\mathcal{N}_n^{\star\star}|/n]
\leqslant \delta$ for sufficiently large $n$.
Because $\delta > 0$ is arbitrary, we find
\begin{equation}
\lim_{n \to
\infty}\E\left[\frac{|\mathcal{N}_n^{\star\star}|}{n}\right] = 0.
\label{eq:EN_n^ss-limit}
\end{equation}
Using the Markov Inequality, we have for any $\epsilon > 0$, 
\[
\P\left(\frac{|\mathcal{N}_n^{\star\star}|}{n} > \epsilon\right)
\leq
\frac{1}{\epsilon}
\cdot
\E\left[\frac{|\mathcal{N}_n^{\star\star}|}{n}\right],
\]
so that (\ref{eq:EN_n^ss-limit}) immediately implies that ${|\mathcal{N}_n^{\star\star}|}/{n} \parrow 0$
as $n \to \infty$.
Upon combining the above, we have from (\ref{eq:N_n^+-limit}) and (\ref{eq:N_n^-star-by-sum}),
\begin{equation}
\frac{|\mathcal{N}_n^{\star}|}{n} \parrow \frac{v(\theta_0)}{\lambda(1-P_{\ell})}
\quad
(n \to \infty).
\label{eq:N_n^star-limit}
\end{equation}
Therefore, we obtain from (\ref{eq:upper-scaled}),
\begin{equation}
\lim_{n \to \infty}
\P\left(
n\left\{\theta_0 -\max_{i \in \mathcal{N}_n^{\star}} Z_i\right\} \geqslant x
\right)
=
\mathrm{e}^{-v(\theta_0)x/(1-P_{\ell})}.
\label{eq:upper-scaled-limit}
\end{equation}

Finally, noting that (\ref{eq:hat-theta-lower}) and
(\ref{eq:hat-theta-upper}) imply
\[
\P\left(
n\left\{\theta_0 -\max_{i \in \mathcal{N}_n^{\star}} Z_i\right\} \geqslant x
\right)
\leq
\P\left(n(\theta_0 - \hat{\theta}_n) \geqslant x\right)
\leq
\P\left(n\left\{\theta_0 - \max_{i \in \mathcal{N}_n^+} W_i\right\}
\geqslant x\right), 
\]
we have from (\ref{eq:lower-scaled-limit}) and (\ref{eq:upper-scaled-limit}),
\[
\lim_{n \to \infty}
\P\left(n(\theta_0 - \hat{\theta}_n) \geqslant x\right)
=
\mathrm{e}^{-v(\theta_0)x/(1-P_{\ell})}.
\]
Also, (\ref{eq:var_MLE-def}), (\ref{eq:hat-theta-lower}), and (\ref{eq:hat-theta-upper}) yield the following bounds for the scaled variance of the estimation error:
\begin{align*}
n^2\Var\left[\theta_0 - \hat{\theta}_n\right]
&\geqslant
\E\left[\left\{n\left(\theta_0 -\max_{i \in \mathcal{N}_n^{\star}} Z_i\right)\right\}^2\right]
-
\E\left[
\left\{n\left(\theta_0 - \max_{i \in \mathcal{N}_n^+} W_i\right)\right\}
\right]^2,
\\
n^2\Var\left[\theta_0 - \hat{\theta}_n\right]
&\leqslant
\E\left[
\left\{n\left(\theta_0 - \max_{i \in \mathcal{N}_n^+} W_i\right)\right\}^2
\right]
-
\E\left[n\left(\theta_0 -\max_{i \in \mathcal{N}_n^{\star}} Z_i\right)\right]^2.
\end{align*}
We can (i) explicitly calculate the right-hand sides of these inequalities using (\ref{eq:lower-scaled}) and (\ref{eq:upper-scaled}), and (ii) verify (with (\ref{eq:N_n^+-limit}) and (\ref{eq:N_n^star-limit})) that their limits (as $n \to \infty$) coincide and are given by $((1-P_{\ell})/v(\theta_0))^2$; we omit the details of these straightforward calculations.
\end{proof}

\end{document}